\documentclass[12pt,a4paper,reqno]{amsart}
\usepackage{amsmath}
\usepackage{amsthm}
\usepackage{amsfonts}
\usepackage{amssymb}
\usepackage{color}
\usepackage{mathrsfs}
\usepackage{graphicx}

\numberwithin{equation}{section}

\theoremstyle{plain}
\newtheorem{thm}{Theorem}[section]

\theoremstyle{plain}
\newtheorem{lem}{Lemma}[section]

\theoremstyle{plain}
\newtheorem*{q1}{Question 1}

\theoremstyle{plain}
\newtheorem*{q2}{Question 2}

\theoremstyle{plain}
\newtheorem*{theorema}{Theorem A}

\theoremstyle{plain}
\newtheorem*{theoremb}{Theorem B}

\theoremstyle{plain}
\newtheorem*{theoremc}{Theorem C}

\theoremstyle{plain}
\newtheorem*{theoremd}{Theorem D}

\theoremstyle{plain}
\newtheorem*{theoreme}{Theorem E}

\theoremstyle{remark}
\newtheorem*{remark}{Remark}

\theoremstyle{remark}
\newtheorem*{remarks}{Remarks}

\theoremstyle{definition}
\newtheorem*{casem1}{Case $\bfitm=\mathbf{1}$}

\theoremstyle{definition}
\newtheorem*{casem2}{Case $\bfitm=\mathbf{2}$}

\theoremstyle{definition}
\newtheorem*{casem3}{Case $\bfitm=\mathbf{3}$}

\theoremstyle{definition}
\newtheorem*{casem4}{Case $\bfitm=\mathbf{4}$}

\theoremstyle{definition}
\newtheorem*{dec}{Double Even Criterion}

\theoremstyle{definition}
\newtheorem*{gcdc}{GCD Criterion}

\theoremstyle{definition}
\newtheorem*{dpca}{Double Periodic Coloring Algorithm}

\theoremstyle{definition}
\newtheorem*{claim}{Claim}

\theoremstyle{definition}
\newtheorem*{case1a}{Case 1A}

\theoremstyle{definition}
\newtheorem*{case1b}{Case 1B}

\theoremstyle{definition}
\newtheorem*{case1c}{Case 1C}

\theoremstyle{definition}
\newtheorem*{case1}{Case 1}

\theoremstyle{definition}
\newtheorem*{case2}{Case 2}

\theoremstyle{definition}
\newtheorem*{case3}{Case 3}

\theoremstyle{definition}
\newtheorem*{case4}{Case 4}

\theoremstyle{definition}
\newtheorem*{case5}{Case 5}

\theoremstyle{definition}
\newtheorem*{fact1}{Fact 1}

\theoremstyle{definition}
\newtheorem*{fact2}{Fact 2}

\def\bfitm{\boldsymbol{m}}

\def\dd{\mathrm{d}}
\def\ee{\mathrm{e}}

\def\eps{\varepsilon}

\def\Nn{\mathbb{N}}

\def\Zz{\mathbb{Z}}

\def\AAA{\mathcal{A}}
\def\BBB{\mathcal{B}}

\def\EEE{\mathcal{E}}

\def\III{\mathcal{I}}

\def\KKK{\mathcal{K}}
\def\LLL{\mathcal{L}}
\def\MMM{\mathcal{M}}

\def\PPP{\mathcal{P}}

\def\WWW{\mathcal{W}}
\def\XXX{\mathcal{X}}

\def\frakc{\mathfrak{c}}

\DeclareMathOperator{\meas}{meas}

\DeclareMathOperator{\length}{length}

\DeclareMathOperator{\ls}{LS}
\DeclareMathOperator{\rs}{RS}
\DeclareMathOperator{\parity}{parity}
\DeclareMathOperator{\sign}{sign}

\renewcommand{\le}{\leqslant}
\renewcommand{\ge}{\geqslant}

\title[New Kronecker--Weyl type equidistribution results]
{New Kronecker--Weyl type equidistribution\\
results and diophantine approximation}

\author[Beck]{J. Beck}

\address{Department of Mathematics, Rutgers University, Hill Center for the Mathematical Sciences, Piscataway NJ 08854, USA}

\email{jbeck@math.rutgers.edu}

\author[Chen]{W.W.L. Chen}

\address{Department of Mathematics and Statistics, Macquarie University, Sydney NSW 2109, Australia}

\email{william.chen@mq.edu.au}

\author[Yang]{Y. Yang}

\address{Department of Mathematics, Rutgers University, Hill Center for the Mathematical Sciences, Piscataway NJ 08854, USA}

\email{yy458@math.rutgers.edu}

\keywords{equidistribution, diophantine approximation, geodesics}

\subjclass[2010]{11K38, 37E35}

\begin{document}

\begin{abstract}
An interesting result of Veech more than 50 years ago is a parity, or mod~$2$, version of the Kronecker--Weyl equidistribution theorem concerning the irrational rotation sequence $\{q\alpha\}$, $q=0,1,2,3,\ldots.$
If $\alpha$ is badly approximable and $b\in(0,1)$ satisfies $b\ne\{m\alpha\}$ for any $m\in\Zz$, then the parity of cardinalities of the sets
$\{1\le q\le N:\{q\alpha\}\in[0,b)\}$ as $N\to\infty$ is evenly distributed.

We first answer a question of Veech and establish a stronger form of the mod~$n$ analog of his result (Theorem~3.1).
Furthermore, for irrational $\alpha$ and $b=\{m\alpha\}$ for some $m\in\Nn$, we give a simple yet precise characterization of those cases that give rise to even distribution (Theorem~2.1).
We also obtain time-quantitative description of some very striking violations of uniformity -- this part is particularly number theoretic in nature, and involves
Ostrowski representations of positive integers and $\alpha$-expansions of real numbers (Theorem~3.4).

The Veech discrete $2$-circle problem can also be visualized as a problem that concerns $1$-direction geodesic flow on a surface obtained by modifying the surface comprising two side-by-side squares by the inclusion of symmetric barriers and gates on the vertical edges, with appropriate modification of the vertical edge identifications.
We establish a far-reaching generalization of this case to ones that concern $1$-direction geodesic flow on surfaces obtained by modifying a finite square tiled translation surface in analogous but not necessarily symmetric ways (Theorem~3.2).
\end{abstract}

\maketitle

\thispagestyle{empty}

%
%

\section{Introduction}\label{sec1}

Our starting point is the famous Kronecker--Weyl equidistribution theorem which refers to the uniformity result concerning the irrational rotation sequence.

This says that the sequence $\{q\alpha\}$, $q=0,1,2,3,\ldots,$ where $\alpha$ is irrational and $\{z\}$ denotes the fractional part of~$z$, is uniformly distributed in the unit interval $[0,1)$, so that for any subinterval $[a,b)\subset[0,1)$, we have
\begin{equation}\label{eq1.1}
\lim_{N\to\infty}\frac{1}{N}\sum_{\substack{{1\le q\le N}\\{\{q\alpha\}\in[a,b)}}}1=b-a.
\end{equation}
It is easy to show that \eqref{eq1.1} holds for all $[a,b)\subset[0,1)$ if and only if it holds for all $[0,b)\subset[0,1)$.
Furthermore, we can consider the more general sequence $\{\tau+q\alpha\}$, $q=0,1,2,3,\ldots,$ with an arbitrary starting point $\tau\in[0,1)$.
Then for any $\tau\in[0,1)$ and $b\in[0,1)$, we have
\begin{displaymath}
\lim_{N\to\infty}\frac{1}{N}\sum_{\substack{{1\le q\le N}\\{\{\tau+q\alpha\}\in[0,b)}}}1=b.
\end{displaymath}

This sequence is called the irrational rotation sequence because if we take a circle with circumference $1$ and radius~$1/2\pi$, then the unit interval can be represented by this circle, and moving from one term to the next corresponds to an anticlockwise rotation by an angle $2\pi\alpha$, as shown in Figure~1.1.

\begin{displaymath}
\begin{array}{c}
\includegraphics{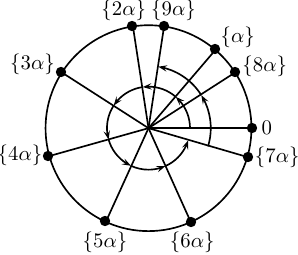}
\\
\mbox{Figure 1.1: the irrational rotation sequence}
\end{array}
\end{displaymath}

The uniformity result concerning the irrational rotation sequence is the first \textit{equidistribution} type result, proved independently by Bohl, Sierpinski and Weyl around 1910, followed soon by the multidimensional version and also the continuous version concerning the torus line, both due to Weyl.
And of course Birkhoff's ergodic theorem, proved about 20 years later, says that in general every ergodic measure-preserving transformation is a \textit{rich source}, namely that it provides half-infinite orbits that exhibit equidistribution relative to the invariant measure.

An interesting problem studied about 50 years ago by Veech~\cite{V1} is the following \textit{parity}, or mod~$2$, version of the classical equidistribution theorem.
Take two copies of the  circle with circumference~$1$ and radius~$1/2\pi$, and mark off a segment $[0,b)$ of length~$b$ in the anticlockwise direction on each circle.
Let $J_1=J_1(b)$ denote this segment on the first circle and let $J_2=J_2(b)$ denote this segment on the second circle.
We now take an irrational number~$\alpha$, and consider the discrete dynamical system illustrated in Figure~1.2.

\begin{displaymath}
\begin{array}{c}
\includegraphics{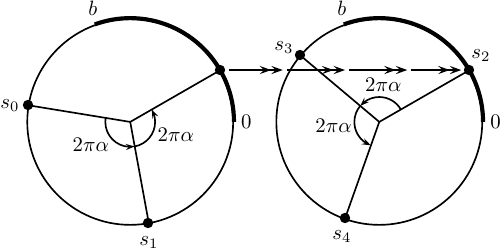}
\\
\mbox{Figure 1.2: the parity version of the classical equidistribution theorem}
\end{array}
\end{displaymath}

Start with an arbitrary point $s_0$ on the first circle~$C_1$.
Rotating in the anticlockwise direction by an angle $2\pi\alpha$, we arrive at a point~$s_1$.
If $s_1$ does not lie on~$J_1$, then we leave it where it is.
If $s_1$ lies on~$J_1$, then we move it to the corresponding point on the second circle~$C_2$.

In general, suppose that the point $s_i$ lies on the circle~$C_j$, where $j=1,2$.
Rotating in the anticlockwise direction by an angle $2\pi\alpha$, we arrive at a point~$s_{i+1}$.
If $s_{i+1}$ does not lie on~$J_j$, then we leave it where it is.
If $s_{i+1}$ lies on~$J_j$, then we move it to the corresponding point on the other circle~$C_k$, where $k\in\{1,2\}\setminus\{j\}$.

Clearly the sequence $s_0,s_1,s_2,s_3,\ldots$ keeps alternating between the two circles.
The problem is to describe the distribution of this half-infinite orbit on the union of the two circles, and find cases that exhibit equidistribution.

There are at least two different ways of visualizing the Veech discrete $2$-circle system as a \textit{continuous} flat dynamical system.
This is motivated by the observation that the problem of torus lines with irrational slopes in the unit square as well as the problem of point billiards with initial irrational slopes on a square table are basically equivalent continuous representations of the problem concerning discrete irrational rotation sequences.
More precisely, the $1$-dimensional irrational rotation sequence arises from these two continuous $2$-dimensional flat dynamical systems with irrational slopes via \textit{discretization}, the general method of converting the problem of describing the distribution of a continuous orbit to the discrete problem of studying where the orbit hits the boundary. 

We first discuss a simple continuous system which gives arguably the best way to visualize the Veech discrete $2$-circle system.
In this simple continuous model, we replace the $2$-circle underlying set by a flat surface, and replace the discrete orbit by a geodesic, or generalized torus line.
This flat surface, which we call the \textit{$2$-square-$b$ surface}, is constructed from joining two unit squares side by side and adding an extra vertical barrier, a wall of length $1-b$ between them, as shown in Figure~1.3.
The vertical complement of the barrier, indicated by the line in light gray, is a \textit{$b$-size gate}, or \textit{$b$-gate}, in the middle which makes it possible to travel from one square to the other.
To make this a surface, we identify pairs of boundary edges with the same label via perpendicular translation.
Note that the two sides of the vertical barrier in the middle are \textit{different} edges.
Note also that the $2$-square-$b$ surface actually has two $b$-gates.
Apart from the obvious one in the middle, there is a second $b$-gate on the far right vertical edge $v_1$ which is identified with the far left vertical edge~$v_1$.
This is a $b$-gate, as it is clear that a geodesic that reaches the far right vertical edge $v_1$ continues from the corresponding point on the far left vertical edge~$v_1$, and in doing so, passes from the right square to the left square.

\begin{displaymath}
\begin{array}{c}
\includegraphics{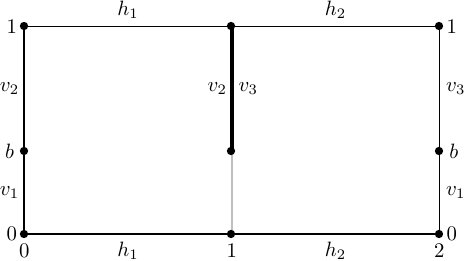}
\\
\mbox{Figure 1.3: $2$-square-$b$ surface}
\end{array}
\end{displaymath}

Since the $2$-square-$b$-surface is a flat translation surface, geodesics on this surface are $1$-direction generalized torus lines.
If the slope of a geodesic is~$\alpha$, then we call it an $\alpha$-geodesic or $\alpha$-line.

Let us now clarify the connection between the Veech discrete $2$-circle system in Figure~1.2 and the $1$-direction geodesic flow with slope $\alpha$ on the $2$-square-$b$ surface in Figure~1.4.

First of all, we can represent the two circles in Figure~1.2 by two \textit{circles in the vertical direction} in Figure~1.4.
We can first view the far left edges $v_1$ and $v_2$ of the $2$-square-$b$ surface as forming a circle, due to the identification of the point $(0,0)$ at the bottom with the point $(0,1)$ at the top.
Thus we visualize the left vertical edge of the left square of the $2$-square-$b$ surface as the left circle in Figure~1.2.

We can next view the middle edge $v_3$ and the $b$-gate below it of the $2$-square-$b$ surface as forming a circle, due to the identification of the point $(1,0)$ at the bottom with the point $(1,1)$ at the top.
Thus we visualize the left vertical edge of the right square of the $2$-square-$b$ surface as the right circle in Figure~1.2.

\begin{displaymath}
\begin{array}{c}
\includegraphics{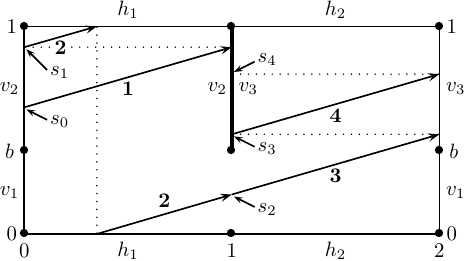}
\\
\mbox{Figure 1.4: a geodesic with slope $\alpha$ on the $2$-square-$b$ surface}
\end{array}
\end{displaymath}

Indeed, we can go back and forth between Figures 1.2 and~1.4.

Consider the point $s_0$ on the left circle in Figure~1.2.
Based on the representation of the two circles just discussed, we find $s_0$ on the left vertical edge of the left square of the $2$-square-$b$ surface as shown in Figure~1.4, and $s_0$ is the initial point of the geodesic segment~$\mathbf{1}$.
The point $s_1$ is obtained from $s_0$ by rotating in the anticlockwise direction by an angle~$2\pi\alpha$, and we see from Figure~1.2 that it does not lie on~$J_1$, so it stays on the left circle.
Now the point $s_1$ is related to the terminal point of the geodesic segment~$\mathbf{1}$.
As shown in Figure~1.4, this terminal point lies on the edge $v_2$ in the middle, but in view of the identification of the edges~$v_2$, we can place the point $s_1$ on the left vertical edge of the left square of the $2$-square-$b$ surface that corresponds to the left circle.

As shown in Figure~1.4, $s_1$ is the initial point of the geodesic segment~$\mathbf{2}$.
The point $s_2$ is obtained from $s_1$ by rotating in the anticlockwise direction by an angle~$2\pi\alpha$, and we see from Figure~1.2 that it lies on~$J_1$, so it moves to the corresponding point on the right circle.
Now the point $s_2$ is related to the terminal point of the geodesic segment~$\mathbf{2}$.
As shown in Figure~1.4, this terminal point lies on the $b$-gate in the middle, on the left vertical edge of the right square of the $2$-square-$b$ surface that corresponds to the right circle.

As shown in Figure~1.4, $s_2$ is the initial point of the geodesic segment~$\mathbf{3}$.
The point $s_3$ is obtained from $s_2$ by rotating in the anticlockwise direction by an angle~$2\pi\alpha$, and we see from Figure~1.2 that it does not lie on~$J_2$, so it stays on the right circle.
Now the point $s_3$ is related to the terminal point of the geodesic segment~$\mathbf{3}$.
As shown in Figure~1.4, this terminal point lies on the edge $v_3$ on the right, but in view of the identification of the edges~$v_3$, we can place the point $s_3$ on the left vertical edge of the right square of the $2$-square-$b$ surface that corresponds to the right circle.

And so on.

There is a fundamental difference between torus line flow on a square and geodesic flow on the $2$-square-$b$ surface.
Torus line flow in a square, or in any cube of higher dimensions, exhibits remarkable stability and predictability, where two particles moving on two parallel torus lines and close to each other with the same speed remain close forever. 
Thus such dynamical systems are said to be \textit{integrable}.

How about the analogous question for geodesic flow on the $2$-square-$b$ surface?
Here, there are singular points, and two particles moving with the same speed on two parallel geodesic segments close to each other do \textit{not} remain close forever after they pass through opposite sides of a split singularity, as shown in Figure~1.5.

\begin{displaymath}
\begin{array}{c}
\includegraphics{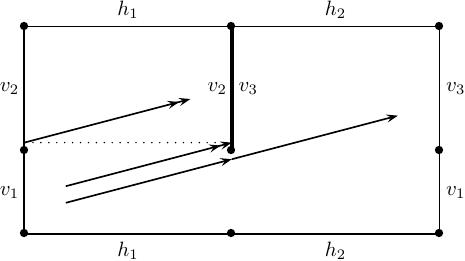}
\\
\mbox{Figure 1.5: singular points on the $2$-square-$b$ surface}
\end{array}
\end{displaymath}

Thus this dynamical system is said to be \textit{non-integrable}.

We next discuss the second model, a billiard system due to Masur.
Billiards have the advantage that they represent a more-or-less  \textit{legitimate} mechanical system, one step closer to physics.
The billiard table in this second model is the underlying double-square of the $2$-square-$b$ surface.
For convenience, we take a copy scaled by half, as shown in the picture on the left in Figure~1.6.

The billiard flow is a $4$-direction flow.
The well-known trick of \textit{unfolding}, first introduced by K\"{o}nig and Sz\"{u}cs~\cite{KS} in 1913, converts the $4$-direction billiard flow on the table in the picture on the left in Figure~1.6 to a $1$-direction linear flow on the corresponding \textit{$4$-copy flat surface}, obtained by a reflection across the right vertical side, followed by a reflection of the whole image across the top horizontal side, as shown in the picture on the right in Figure~1.6.
Here the left and right vertical edges are identified, the top and bottom horizontal edges are identified, and the two sides of the two walls are appropriately identified as shown.

\begin{displaymath}
\begin{array}{c}
\includegraphics{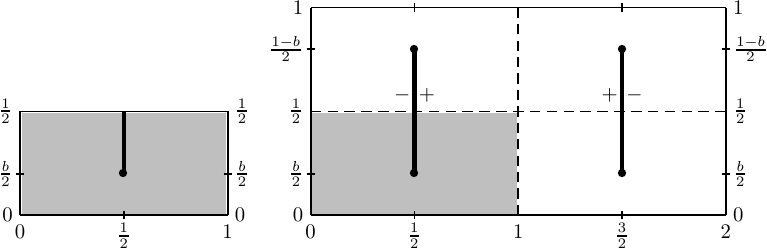}
\\
\mbox{Figure 1.6: the underlying double-square of the $2$-square-$b$ surface}
\\
\mbox{and a torus with two vertical walls}
\end{array}
\end{displaymath}

In particular, the right side of the left wall and the left side of the right wall, both indicated by $+$, are identified, while the left side of the left wall and the right side of the right wall, both indicated by $-$, are identified.
Now a torus has genus~$1$.
However, with the two walls, the surface in the picture on the right in Figure~1.6 has genus~$2$.
And $1$-direction geodesic flow on this surface
is a $4$-fold covering of billiard flow on the table in the picture on the left in Figure~1.6.

Let us now clarify the connection between the Veech discrete $2$-circle system in Figure~1.2 and the $1$-direction geodesic flow with slope $\alpha$ on the torus with two vertical walls in Figure~1.7.

First of all, we can represent the two circles in Figure~1.2 by two \textit{circles in the vertical direction} in Figure~1.7.
We can first view the right side of the left wall and its vertical extension to the points $(1/2,0)$ and $(1/2,1)$ as forming a circle, with the extension forming the $b$-gate, due to the identification of the point $(1/2,0)$ at the bottom with the point $(1/2,1)$ at the top.
Thus we visualize this as the left circle in Figure~1.2.

We can next view the right side of the right wall and its vertical extension to the points $(3/2,0)$ and $(3/2,1)$ as forming a circle, with the extension forming the $b$-gate, due to the identification of the point $(3/2,0)$ at the bottom with the point $(3/2,1)$ at the top.
Thus we visualize this as the right circle in Figure~1.2.

\begin{displaymath}
\begin{array}{c}
\includegraphics{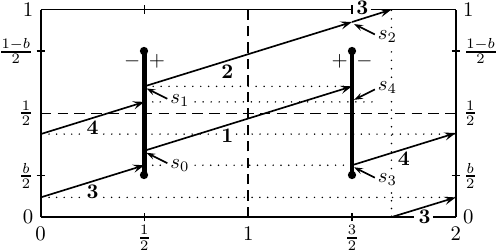}
\\
\mbox{Figure 1.7: a geodesic with slope $\alpha$ on the torus}
\\
\mbox{with two vertical walls in the middle}
\end{array}
\end{displaymath}

Indeed, we can go back and forth between Figures 1.2 and~1.7.

Consider the point $s_0$ on the left circle in Figure~1.2.
Based on the representation of the two circles just discussed, we find $s_0$ on the right side of the left wall in Figure~1.7, and $s_0$ is the initial point of the geodesic segment~$\mathbf{1}$.
The point $s_1$ is obtained from $s_0$ by rotating in the anticlockwise direction by an angle~$2\pi\alpha$, and we see from Figure~1.2 that it does not lie on~$J_1$, so it stays on the left circle.
Now the point $s_1$ is related to the terminal point of the geodesic segment~$\mathbf{1}$.
As shown in Figure~1.7, this terminal point lies on the left side of the right wall, but in view of the identification of the left side of the right wall with the right side of the left wall, we can place the point $s_1$ at the corresponding position on the right side of the left wall.
This corresponds to the left circle.

As shown in Figure~1.7, $s_1$ is the initial point of the geodesic segment~$\mathbf{2}$.
The point $s_2$ is obtained from $s_1$ by rotating in the anticlockwise direction by an angle~$2\pi\alpha$, and we see from Figure~1.2 that it lies on~$J_1$, so it moves to the corresponding point on the right circle.
Now the point $s_2$ is related to the terminal point of the geodesic segment~$\mathbf{2}$.
As shown in Figure~1.7, this terminal point lies on the extension of the right side of the right wall that forms the $b$-gate.
This corresponds to the right circle.

As shown in Figure~1.7, $s_2$ is the initial point of the geodesic segment~$\mathbf{3}$.
The point $s_3$ is obtained from $s_2$ by rotating in the anticlockwise direction by an angle~$2\pi\alpha$, and we see from Figure~1.2 that it does not lie on~$J_2$, so it stays on the right circle.
Now the point $s_3$ is related to the terminal point of the geodesic segment~$\mathbf{3}$.
As shown in Figure~1.7, this terminal point lies on the left side of the left wall, but in view of the identification of the left side of the left wall with the right side of the right wall, we can place the point $s_3$ at the corresponding position on the right side of the right wall.
This corresponds to the right circle.

And so on.

For the rest of this paper, we shall represent the Veech discrete $2$-circle system as $1$-direction geodesic flow on the $2$-square-$b$ surface.

The most  natural question is the following.
Since we are not interested in periodic orbits, we shall always assume that the slope $\alpha$ is irrational.

\begin{q1}
Let $\alpha$ be an irrational number.
When can we guarantee that every half-infinite $\alpha$-geodesic on the $2$-square-$b$ surface is uniformly distributed?
\end{q1}

An infinite discrete or continuous orbit is \textit{uniformly distributed} if, given a \textit{nice} test set~$A$, the asymptotic proportion of time the orbit visits $A$ is equal to the relative area of~$A$.
A classical result of Weyl~\cite{W} then says that it does not make any difference in the definition of uniformity of an infinite discrete or continuous orbit in the $2$-dimensional case whether we choose the family of \textit{nice} test sets to be (i) the class of all triangles, or (ii) the very different class of all circles, or (iii) the much larger class of all Jordan measurable sets which contains both (i) and (ii).

We recall that \textit{Jordan measurable} means that the $2$-dimensional Riemann integral of the characteristic function of the set is well defined.

In this paper \textit{uniformly distributed} and \textit{equidistributed} have the same meaning.

We can assume that the irrational number $\alpha$ satisfies $0<\alpha<1$. 
To see this, let $n\in\Zz$ be an arbitrary non-zero integer.
Starting from the same point on a vertical edge of the $2$-square-$b$ surface, it is clear that the $\alpha$-geodesic and the corresponding
$(\alpha+n)$-geodesic intersect the three vertical edges of the $2$-square-$b$ surface at the same points.
Thus if the $\alpha$-geodesic is equidistributed on the $2$-square-$b$ surface, then the corresponding $(\alpha+n)$-geodesic is also equidistributed on the $2$-square-$b$ surface, and \textit{vice versa}.

We shall formulate the main results of this long paper in Section~\ref{sec3}.
Before that, we discuss in Section~\ref{sec2} the interesting special case when $b=\{m\alpha\}$ for some non-zero integer $m\in\Zz$.
This special case, not considered by Veech~\cite{V1}, is in part \textit{simple} and in part \textit{difficult}.

Recall that geodesic flow on the $2$-square-$b$ surface is non-integrable, in view of the singularities in the orbit space, making it difficult to predict
the long-term behavior of any given half-infinite geodesic.
Assuming that a particle moves on the geodesic with constant speed, it is often difficult to predict \textit{which square} contains the particle at any given time instance $t$, when $t$ is large.
On the other hand, there are only two candidates for the location of the particle, with one in each square, since the $\alpha$-flow on the
$2$-square-$b$ surface \textit{modulo~$1$} reduces to a torus line in the unit square, giving rise to a well-predictable integrable system, namely, a straight line on the plane modulo~$1$.
So the difficult question is which one of these two candidates is the true location of the particle.

The special case when $b=\{m\alpha\}$ for some non-zero integer $m\in\Zz$ is \textit{simple}, in the sense that there is a particularly simple and efficient algorithm that answers the question of \textit{which square}.
Indeed, this question is equivalent to the following number-theoretic parity type problem.
Consider the infinite irrational rotation sequence
\begin{displaymath}
s_q=\{\tau+q\alpha\},
\quad
q=0,1,2,3,\ldots,
\end{displaymath}
with arbitrary starting point $\tau\ge0$.
For every $N\in\Nn$, let $\Psi(\alpha;\tau;b;N)$ denote the number of integers $q$ satisfying $0\le q\le N-1$ such that $0\le s_q<b$.
It is easy to see that the parity of $\Psi(\alpha;\tau;b;N)$ answers the question of \textit{which square} contains the particle.
This follows from discretization of the $\alpha$-geodesic and studying the consecutive intersection points on the vertical edges of the $2$-square-$b$ surface.
Note that an $\alpha$-geodesic moves from one square to the other if and only if it crosses one of the two $b$-gates, and any two consecutive
gate crossings always happen with different $b$-gates.

We first consider the special case $0<b=\alpha<1$ and $\tau=0$.
For $N\ge2$, we have
\begin{equation}\label{eq1.2}
\Psi(\alpha;0;b;N)=\lceil(N-1)\alpha\rceil,
\end{equation}
where $\lceil\beta\rceil$ denotes the upper integral part of a real number~$\beta$.
To see this, consider the numbers
\begin{equation}\label{eq1.3}
\tau+q\alpha,
\quad
0\le q\le N-1.
\end{equation}
Clearly they fall into the interval $[0,\lceil(N-1)\alpha\rceil]$.
Now for every integer $n$ satisfying $0\le n<\lceil(N-1)\alpha\rceil$, there is a unique number in \eqref{eq1.3} such that $\tau+q\alpha\in[n,n+\alpha)$, so that $0\le s_q<b$.
On the other hand, $0\le s_q<b$ if and only if $\tau+q\alpha\in[n,n+\alpha)$ for some integer $n$ satisfying $0\le n<\lceil(N-1)\alpha\rceil$.

A somewhat similar argument shows that for every integer $N\ge2$, we have
\begin{equation}\label{eq1.4}
\Psi(\alpha;\tau;b;N)=\left\{\begin{array}{ll}
\lceil\{\tau\}+(N-1)\alpha\rceil,&\mbox{if $0\le\{\tau\}<b$},\\
\lfloor\{\tau\}+(N-1)\alpha\rfloor,&\mbox{if $b\le\{\tau\}<1$},
\end{array}\right.
\end{equation}
where $\lfloor\beta\rfloor$ denotes the lower integral part of a real number~$\beta$.
Note that in the first case $0\le\{\tau\}<b$, the first term $s_0<b$, whereas in the second case $b\le\{\tau\}<1$, the first term $s_0\ge b$.

We next consider the special case $0<b=\{2\alpha\}<1$ and $\tau=0$.
Here we apply \eqref{eq1.4} to each of the two subsequences
\begin{displaymath}
s_0,s_2,s_4,s_6,\ldots
\quad\mbox{and}\quad
s_1,s_3,s_5,s_7,\ldots,
\end{displaymath}
with the same gap $b=\{2\alpha\}$.

Suppose first that $0<\alpha<1/2$, so that $b=\{2\alpha\}=2\alpha$.
Then
\begin{equation}\label{eq1.5}
\Psi(\alpha;0;\{2\alpha\};N)
=\left\lceil\left\lfloor\frac{N-1}{2}\right\rfloor2\alpha\right\rceil
+\left\lceil\alpha+\left\lfloor\frac{N-2}{2}\right\rfloor2\alpha\right\rceil.
\end{equation}
Note here that $s_1=\alpha<2\alpha=b$.

Suppose next that $1/2<\alpha<1$, so that $b=\{2\alpha\}=2\alpha-1$.
Then
\begin{equation}\label{eq1.6}
\Psi(\alpha;0;\{2\alpha\};N)
=\left\lceil\left\lfloor\frac{N-1}{2}\right\rfloor(2\alpha-1)\right\rceil
+\left\lfloor\alpha+\left\lfloor\frac{N-2}{2}\right\rfloor(2\alpha-1)\right\rfloor.
\end{equation}
Note here that $s_1=\alpha>2\alpha-1=b$.

For the special case $0<b=\{3\alpha\}<1$, we apply \eqref{eq1.4} separately to each of the three subsequences
\begin{displaymath}
s_0,s_3,s_6,s_9,\ldots,
\quad
s_1,s_4,s_7,s_{10},\ldots,
\quad
s_2,s_5,s_8,s_{11},\ldots,
\end{displaymath}
with the same gap $b=\{3\alpha\}$.

And so on.
In general, for any $b=\{m\alpha\}$, where $m\in\Zz$ is non-zero, we obtain an analogous explicit formula for $\Psi(\alpha;\tau;b;N)$, and this gets more complicated as $m$ increases.
Nevertheless, it is not difficult to determine from such an explicit formula the parity of $\Psi(\alpha;\tau;b;N)$, and this parity tells us \textit{which} square contains the particle.
This explains why, on the one hand, we say that this special case when $b=\{m\alpha\}$ for some non-zero integer $m\in\Zz$ is \textit{simple}.
More precisely, we may call it a non-integrable dynamical system with \textit{very low algorithmic complexity}.

On the other hand, this special case is still quite difficult.
For instance, even in the totally innocent looking special case $0<b=\{2\alpha\}$ with $1/2<\alpha<1$, it is not easy at all to determine whether a
half-infinite $\alpha$-geodesic is equidistributed on the whole $2$-square-$b$ surface.

We conclude Section~\ref{sec1} with a simple technical observation.
To study the special case $b=\{m\alpha\}$ for some non-zero integer $m\in\Zz$, it suffices to consider only positive integers~$m$.
This follows on combining the trivial identity $b=\{m\alpha\}=\{-m(1-\alpha)\}$ with the following result.

\begin{lem}\label{lem11}
A half-infinite $\alpha$-geodesic on the $2$-square-$b$ surface with starting point $(0,x)$ lying on the far left vertical edge of the surface is equidistributed on the surface if and only if the half-infinite $(1-\alpha)$-geodesic with starting point $(0,\{b+1-x\})$ lying on the same far left vertical edge of the surface is equidistributed on the surface.
\end{lem}

\begin{proof}
The proof follows from combining three simple transformations.

The first simple transformation is illustrated in Figure~1.8.
It maps an $\alpha$-geodesic with starting point $(0,x)$ on the far left vertical edge of the $2$-square-$b$ surface to an $(\alpha-1)$-geodesic with the same starting point $(0,x)$ on the far left vertical edge of the $2$-square-$b$ surface.
It is clear that they hit the same point $(1,\{x+\alpha\})$ on the middle vertical line of the $2$-square-$b$ surface.
The first geodesic is equidistributed on the $2$-square-$b$ surface if and only if the second geodesic is equidistributed on the $2$-square-$b$ surface.

\begin{displaymath}
\begin{array}{c}
\includegraphics{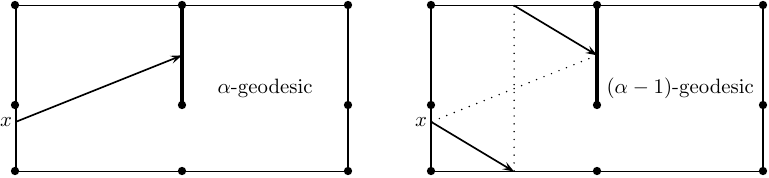}
\\
\mbox{Figure 1.8: $\alpha$-geodesic and $(\alpha-1)$-geodesic}
\end{array}
\end{displaymath}

The second simple transformation is illustrated in Figure~1.9.
It maps an $(\alpha-1)$-geodesic with starting point $(0,x)$ on the far left vertical edge of the $2$-square-$b$ surface to a $(1-\alpha)$-geodesic with starting point $(0,1-x)$ on the far left vertical edge of the $2$-square-$b$ surface reflected across the horizontal line $y=1/2$.
It is clear that the first geodesic hits the point $(1,\{x+\alpha\})$ on the middle vertical line of the $2$-square-$b$ surface, whereas the second geodesic hits the point $(1,1-\{x+\alpha\})$ on the middle vertical line of the $2$-square-$b$ surface reflected across the horizontal line $y=1/2$.
The first geodesic is equidistributed on the $2$-square-$b$ surface if and only if the second geodesic is equidistributed on the $2$-square-$b$ surface reflected across the horizontal line $y=1/2$.

\begin{displaymath}
\begin{array}{c}
\includegraphics{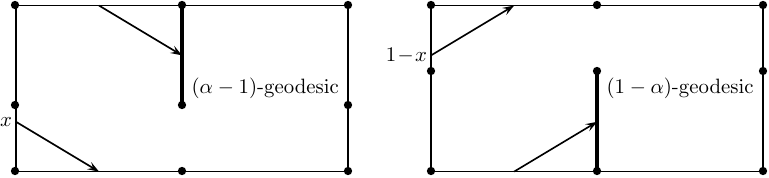}
\\
\mbox{Figure 1.9: $(\alpha-1)$-geodesic and $(1-\alpha)$-geodesic}
\end{array}
\end{displaymath}

The third simple transformation is illustrated in Figure~1.10, which also shows that the $2$-square-$b$ surface can be recovered from the
$2$-square-$b$ surface reflected across the horizontal line $y=1/2$ by a vertical translation by $b$ modulo~$1$.
It now maps a $(1-\alpha)$-geodesic with starting point $(0,1-x)$ on the far left vertical edge of the $2$-square-$b$ surface reflected across the horizontal line $y=1/2$ to a $(1-\alpha)$-geodesic with starting point $(0,\{b+1-x\})$ on the far left vertical edge of the
$2$-square-$b$ surface.
It is clear that the two geodesics hit corresponding points on the middle vertical line of their respective $2$-square-$b$ surfaces.
The first geodesic is equidistributed on the $2$-square-$b$ surface reflected across the horizontal line $y=1/2$ if and only if the second geodesic is equidistributed on the $2$-square-$b$ surface.

\begin{displaymath}
\begin{array}{c}
\includegraphics{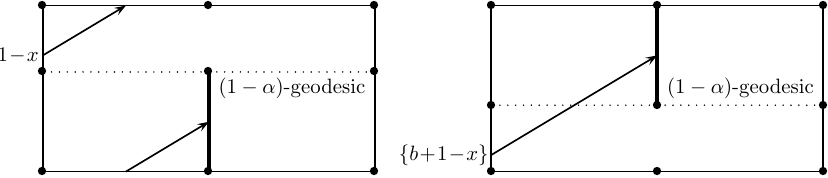}
\\
\mbox{Figure 1.10: vertical translation by $b$ modulo~$1$}
\end{array}
\end{displaymath}

This completes the proof.
\end{proof}

\begin{remark}
Strictly speaking, the $2$-square-$b$ surface reflected across the horizontal line $y=1/2$ followed by a vertical translation by $b$ modulo~$1$ leads to another copy of the $2$-square-$b$ surface if and only if the gates are open intervals or closed intervals.
However, for formulas such as \eqref{eq1.2}, \eqref{eq1.4}, \eqref{eq1.5} and \eqref{eq1.6} to hold precisely, the gates and barriers need to be intervals that are closed at the bottom end and open at the top end.
In any case, an $\alpha$-geodesic can hit any singularity of the $2$-square-$b$ surface at most once, so equidistribution is not affected by altering the openness or closedness of the gates.
\end{remark}

%
%

\section{Some interesting special cases,
and polygonal invariant sets}\label{sec2}

We briefly consider the special case $b=\{m\alpha\}$ for some non-zero integer $m\in\Zz$.
As explained in Section~\ref{sec1}, we may assume that $m$ is positive and $0<\alpha<1$.

\begin{casem1}
In the special case $0<b=\alpha<1$, we can show equidistribution for any half-infinite $\alpha$-geodesic with irrational~$\alpha$.
Here is a relatively simple proof.
The idea is summarized in Figure~2.1.

\begin{displaymath}
\begin{array}{c}
\includegraphics{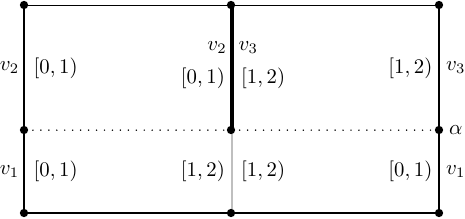}
\\
\mbox{Figure 2.1: the case $0<b=\alpha<1$}
\end{array}
\end{displaymath}

Consider the sequence $\tau+q\alpha$, $q=0,1,2,3,\ldots.$
Without loss of generality, we can assume that $0\le\tau<1$.
Consider a geodesic $\LLL$ on the $2$-square-$b$ surface with slope~$\alpha$, starting at a point on the left vertical edge at height~$\tau$, and hitting the vertical edges of the $2$-square-$b$ surface successively at height $\{\tau+q\alpha\}$, $q=0,1,2,3,\ldots.$
For every such integer~$q$, consider the following assertion:

\begin{itemize}
\item[$P(q)$:]
The condition $\tau+q\alpha\bmod{2}$ is in $[0,1)$ corresponds to a hitting point on the $2$-square-$b$ surface on the left vertical edge, or on the left side of the middle vertical edge above the gate, or on the right vertical edge at the gate, while the condition $\tau+q\alpha\bmod{2}$ is in $[1,2)$ corresponds to a hitting point on the $2$-square-$b$ surface on the right vertical edge above the gate, or on the right side of the middle vertical edge above the gate, or on the middle vertical edge at the gate.
\end{itemize}

It is clear that $P(0)$ holds by definition.
Assume now that $P(k)$ holds for some integer~$k$.

Suppose first that $\tau+k\alpha\bmod{2}$ is in $[0,1)$.
In view of vertical edge identification, we may assume without loss of generality that the corresponding hitting point of $\LLL$ lies on the left vertical edge.
We have one of the following two possibilities:

(i)
If $\tau+(k+1)\alpha\bmod2$ is in $[0,1)$, then since $0<\alpha<1$, we must have
\begin{equation}\label{eq2.1}
\{\tau+(k+1)\alpha\}=\{\tau+k\alpha\}+\alpha.
\end{equation}
It follows that $\alpha\le\{\tau+(k+1)\alpha\}<1$, so that the corresponding hitting point of $\LLL$ is on the left side of the middle vertical edge above the gate, and so $P(k+1)$ holds.

(ii)
If $\tau+(k+1)\alpha\bmod2$ is in $[1,2)$, then since $0<\alpha<1$, we must have
\begin{equation}\label{eq2.2}
\{\tau+(k+1)\alpha\}+1=\{\tau+k\alpha\}+\alpha.
\end{equation}
It follows that $0\le\{\tau+(k+1)\alpha\}<\alpha$, so that the corresponding hitting point of $\LLL$ is on the middle vertical edge at the gate, and so $P(k+1)$ holds.

Suppose next that $\tau+k\alpha\bmod{2}$ is in $[1,2)$.
In view of vertical edge identification, we may assume without loss of generality that the corresponding hitting point of $\LLL$ lies on the right side of the middle vertical edge above the gate, or on the middle vertical edge at the gate.

(i)
If $\tau+(k+1)\alpha\bmod2$ is in $[0,1)$, then since $0<\alpha<1$, we must have \eqref{eq2.2}.
It follows that $0\le\{\tau+(k+1)\alpha\}<\alpha$, so that the corresponding hitting point of $\LLL$ is on the right vertical edge at the gate, and so $P(k+1)$ holds.

(ii)
If $\tau+(k+1)\alpha\bmod2$ is in $[1,2)$, then since $0<\alpha<1$, we must have \eqref{eq2.1}.
It follows that $\alpha\le\{\tau+(k+1)\alpha\}<1$, so that the corresponding hitting point of $\LLL$ is on the right vertical edge above the gate, and so $P(k+1)$ holds.

Thus the statement $P(q)$ holds for every $q=0,1,2,3,\ldots.$

Finally, note that the sequence $\tau+q\alpha$, $q=0,1,2,3,\ldots,$ is uniformly distributed in the double interval $[0,2)$.
\end{casem1}

Next come some surprises.

\begin{casem2}
A pleasant first surprise comes from the special case $b=\{2\alpha\}$ with $0<\alpha<1/2$, so that $0<b=2\alpha<1$.
Figure~2.2 summarizes a very quick proof that any $\alpha$-geodesic on the $2$-square-$b$ surface is not dense or equidistributed.

\begin{displaymath}
\begin{array}{c}
\includegraphics{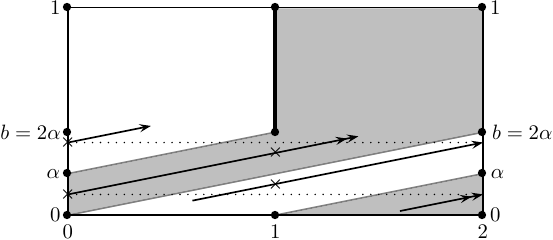}
\\
\mbox{Figure 2.2: the case $b=\{2\alpha\}$ with $0<\alpha<1/2$}
\end{array}
\end{displaymath}

Note that any $\alpha$-geodesic on the $2$-square-$b$ surface \textit{modulo~$1$} reduces to a torus line of slope $\alpha$ in the unit square, and we know that this projected torus line is uniformly distributed as long as $\alpha$ is irrational.
On the other hand, Figure~2.2 shows two invariant subsets of the $2$-square-$b$ surface under geodesic flow of slope~$\alpha$.
It is easy to see that any $\alpha$-geodesic that passes through the shaded part of the $2$-square-$b$ surface remains forever in the shaded part and never reaches the white part, and \textit{vice versa}, so it is not dense on the $2$-square-$b$ surface.

It is easy to see that an $\alpha$-geodesic in the shaded part has \textit{visit density} $\alpha$ on the left square of the surface and $1-\alpha$ on the right square of the surface.
Likewise, an $\alpha$-geodesic in the white part has \textit{visit density} $\alpha$ on the right square of the surface and $1-\alpha$ on the left square of the surface.
Since $\alpha\ne1/2$, this means that there cannot possibly be equidistribution.

Note also from Figure~2.2 that the \textit{square-crossings}, \textit{i.e.}, instances of passing from one square to the other, occur in pairs along any
$\alpha$-geodesic.

\begin{remark}
It is easy to see that a similar argument works for the special case $b=\{m\alpha\}$ for any even positive integer~$m$ with $0<\alpha<1/m$.
Any $\alpha$-geodesic on the $2$-square-$b$ surface is not dense or equidistributed.
\end{remark}

A second surprise is that the case $b=\{2\alpha\}$ with $1/2<\alpha<1$ turns out to be completely different from when $0<\alpha<1/2$.
In this case, every half-infinite $\alpha$-geodesic on the $2$-square-$b$ surface with irrational $\alpha$ is equidistributed.
We do not have a quick proof of this result.
It follows instead from the general Theorem~\ref{thm21} which we shall state later in this section. 
This general result has a fairly non-trivial proof.
\end{casem2}

\begin{casem3}
The special case $b=\{3\alpha\}$ gives rise to equidistribution for every half-infinite $\alpha$-geodesic on the $2$-square-$b$ surface with irrational~$\alpha$.
Again, we do not know a quick proof, and refer the reader to Theorem~\ref{thm21}.
\end{casem3}

Next come more surprises.

\begin{casem4}
Let us first consider the special case $b=\{4\alpha\}$ with
\begin{displaymath}
0<\alpha<\tfrac{1}{4}
\quad\mbox{or}\quad
\tfrac{1}{2}<\alpha<\tfrac{2}{3}
\quad\mbox{or}\quad
\tfrac{2}{3}<\alpha<\tfrac{3}{4}.
\end{displaymath}
Figures 2.3--2.5 summarize very quick proofs that any $\alpha$-geodesic on the $2$-square-$b$ surface is not dense or equidistributed.

In Figure~2.3, an $\alpha$-geodesic in the shaded part has visit density $2\alpha$ on the left square of the surface and $1-2\alpha$ on the right square of the surface.

\begin{displaymath}
\begin{array}{c}
\includegraphics{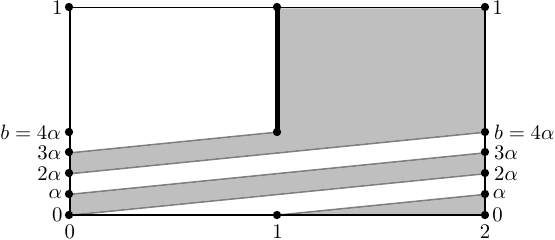}
\\
\mbox{Figure 2.3: the case $b=\{4\alpha\}$ with $0<\alpha<1/4$}
\end{array}
\end{displaymath}

In Figure~2.4, an $\alpha$-geodesic in the shaded part has visit density
\begin{displaymath}
(\{3\alpha\}-\alpha)+\{2\alpha\}
=(3\alpha-1-\alpha)+(2\alpha-1)
=4\alpha-2
\end{displaymath}
on the left square of the surface and
\begin{align}
&
(1-\{3\alpha\})+(\alpha-\{4\alpha\})+\{2\alpha\}
\nonumber
\\
&\quad
=(1-3\alpha+1)+(\alpha-4\alpha+2)+(2\alpha-1)
=3-4\alpha
\nonumber
\end{align}
on the right square of the surface.

\begin{displaymath}
\begin{array}{c}
\includegraphics{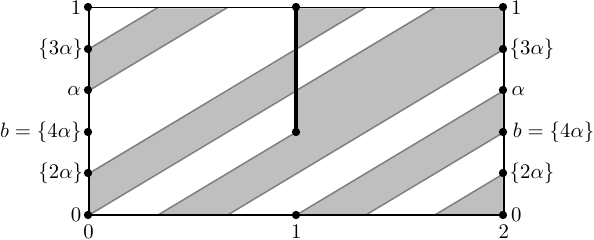}
\\
\mbox{Figure 2.4: the case $b=\{4\alpha\}$ with $1/2<\alpha<2/3$}
\end{array}
\end{displaymath}

In Figure~2.5, an $\alpha$-geodesic in the shaded part has visit density
\begin{displaymath}
(\alpha-\{2\alpha\})+\{3\alpha\}
=(\alpha-2\alpha+1)+(3\alpha-2)
=2\alpha-1
\end{displaymath}
on the left square of the surface and
\begin{align}
&
(1-\{4\alpha\})+(\alpha-\{2\alpha\})+\{3\alpha\}
\nonumber
\\
&\quad
=(1-4\alpha+2)+(\alpha-2\alpha+1)+(3\alpha-2)
=2-2\alpha
\nonumber
\end{align}
on the right square of the surface.

\begin{displaymath}
\begin{array}{c}
\includegraphics{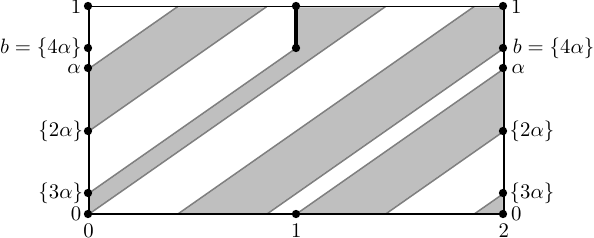}
\\
\mbox{Figure 2.5: the case $b=\{4\alpha\}$ with $2/3<\alpha<3/4$}
\end{array}
\end{displaymath}

However, for the special case $b=\{4\alpha\}$ with $1/4<\alpha<1/2$ or $3/4<\alpha<1$, every half-infinite $\alpha$-geodesic on the $2$-square-$b$ surface with irrational $\alpha$ is equidistributed.
Again, we do not know a quick proof, and refer the reader to Theorem~\ref{thm21}.
\end{casem4}

At first sight this case study may seem hopelessly complicated and mysterious.
However, there is a simple underlying rule that explains everything.
We call this the \textit{Double Even Criterion}.

If the Double Even Criterion fails, then every half-infinite $\alpha$-geodesic on the $2$-square-$b$ surface with irrational $\alpha$ is equidistributed.
This case forms the hard part of the case $n=2$ of Theorem~\ref{thm21}.

On the other hand, if the Double Even Criterion holds, then there is a reasonably simple algorithm to construct $2$ non-trivial $\alpha$-flow invariant subsets of the $2$-square-$b$ surface.
Clearly density and equidistribution for any $\alpha$-geodesic on the $2$-square-$b$ surface are impossible.

Let $b=\{m\alpha\}$, where $m\ge2$ is an integer and $\alpha$ is an irrational number satisfying $0<\alpha<1$.
We take the parameter $\Upsilon(m;\alpha)$ to denote the total number of integers $q$ such that $1\le q\le m$ and $\{q\alpha\}<\alpha$.
For example, as clearly shown in Figures 2.3--2.5, we have
\begin{displaymath}
\Upsilon(4;\alpha)=\left\{\begin{array}{ll}
0,&\mbox{if $0<\alpha<1/4$},\\
2,&\mbox{if $1/2<\alpha<2/3$},\\
2,&\mbox{if $2/3<\alpha<3/4$}.
\end{array}\right.
\end{displaymath}

\begin{dec}
The integer $m\ge2$ and the parameter $\Upsilon(m;\alpha)$ are both even.
\end{dec}

The Double Even Criterion is a special case of a more general criterion which applies to the $n$-square-$b$ surface for any integer $n\ge2$, the natural generalization of the $2$-square-$b$ surface to a surface consisting of a horizontal row of $n$ consecutive unit squares with $n-1$ $b$-size gates between the squares, and with appropriate edge identification.
The $3$-square-$b$ surface is shown in Figure~2.6.

\begin{displaymath}
\begin{array}{c}
\includegraphics{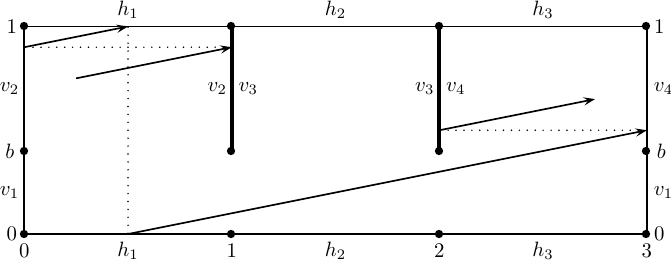}
\\
\mbox{Figure 2.6: the $3$-square-$b$ surface}
\end{array}
\end{displaymath}

\begin{gcdc}
For the $n$-square-$b$ surface with $b=\{m\alpha\}$, the greatest common divisor $d$ of the three integers $n$, $m$ and $\Upsilon(m;\alpha)$ satisfies $d>1$.
\end{gcdc}

If the GCD Criterion fails, then every half-infinite $\alpha$-geodesic on the $n$-square-$b$ surface with irrational $\alpha$ is equidistributed.
This case forms the hard part of Theorem~\ref{thm21}.

On the other hand, if the GCD Criterion holds with greatest common divisor $d>1$, then there is a reasonably simple algorithm to construct $d$ non-trivial
$\alpha$-flow invariant subsets of the $n$-square-$b$ surface.
Clearly density and equidistribution for any $\alpha$-geodesic on the $n$-square-$b$ surface are impossible.
This case is relatively short, and we discuss it now.

Suppose that the GCD Criterion holds.
We now show how we can construct $d$ non-trivial $\alpha$-flow invariant subsets of the $n$-square-$b$ surface.

Consider the finite sequence
\begin{displaymath}
0,\{\alpha\},\{2\alpha\},\ldots,\{m\alpha\}=b
\end{displaymath}
of $m+1$ terms, and arrange it in increasing order
\begin{equation}\label{eq2.3}
0=b_0<b_1<\ldots<b_\Upsilon<b_{\Upsilon+1}=\alpha<b_{\Upsilon+2}<\ldots<b_m,
\end{equation}
where the index $\Upsilon$ denotes the parameter $\Upsilon=\Upsilon(m;\alpha)$,
\begin{displaymath}
b_m=\max_{1\le q\le m}\{q\alpha\}<1,
\end{displaymath}
and $b$ is one of the elements in \eqref{eq2.3}, so that $b=b_\nu$ for some $\nu=1,\ldots,m$.

If we remove the term $b_\nu$ from the sequence \eqref{eq2.3}, then we obtain a subsequence
\begin{equation}\label{eq2.4}
0=b'_0<b'_1<\ldots<b'_{m-1}
\end{equation}
of $m$ terms, where for every integer $j=0,\ldots,m-1$,
\begin{displaymath}
b'_j=\left\{\begin{array}{ll}
b_j,&\mbox{if $0\le j<\nu$},\\
b_{j+1},&\mbox{if $\nu\le j\le m-1$}.
\end{array}\right.
\end{displaymath}
Note that the elements of the subsequence \eqref{eq2.4} are in one-to-one correspondence with the collection of division points
\begin{equation}\label{eq2.5}
\{q\alpha\},
\quad
q=0,1,\ldots,m-1.
\end{equation}
This subsequence also leads to a partition of the unit interval $[0,1)$ into $m$ intervals
\begin{equation}\label{eq2.6}
I_j=[b'_j,b'_{j+1}),
\quad
j=0,1,\ldots,m-1,
\end{equation}
with the convention that $b'_m=1$.

Since $d$ divides~$m$, we can color the intervals \eqref{eq2.6} from top to bottom with distinct colors $\frakc_1,\ldots,\frakc_d$, repeated periodically $m/d$ times.

We now proceed to $d$-color the $n$-square-$b$ surface as follows.

\begin{dpca}
Suppose that the integer $d$ divides both $n$ and~$m$.
Let $\frakc_1,\ldots,\frakc_d$ denote $d$ distinct colors.

(1)
Suppose that $\ell=1,\ldots,d$.
Identify the left vertical edge of the $\ell$-th square face of the $n$-square-$b$ surface with the interval $[0,1)$, consisting of the $m$ intervals \eqref{eq2.6}.
We color these intervals from top to bottom by the colors $\frakc_\ell,\ldots,\frakc_d,\frakc_1,\ldots,\frakc_{\ell-1}$, repeated periodically $m/d$ times.
This clearly gives rise to a periodic $d$-coloring of this edge.
Using the $\alpha$-flow, we can extend this periodic $d$-coloring to a $d$-coloring $C(\ell)$ of the $\ell$-th square face of the $n$-square-$b$ surface.

(2)
We then $d$-color the other square faces of the $n$-square-$b$ surface by repeating the $d$-colorings $C(1),\ldots,C(d)$ of the first $d$ square faces periodically $n/d$ times.
\end{dpca}

\begin{remark}
The $m\times n$ array
\begin{equation}\label{eq2.7}
\begin{array}{cccccccccc}
&&C(1)&C(2)&C(3)&\ldots&C(d-1)&C(d)&&\ldots
\vspace{5pt}\\
I_{m-1}&&\frakc_1&\frakc_2&\frakc_3&\ldots&\frakc_{d-1}&\frakc_d&&\ldots\\
I_{m-2}&&\frakc_2&\frakc_3&\frakc_4&\ldots&\frakc_d&\frakc_1&&\ldots\\
I_{m-3}&&\frakc_3&\frakc_4&\frakc_5&\ldots&\frakc_1&\frakc_2&&\ldots\\
\vdots&&\vdots&\vdots&\vdots&&\vdots&\vdots\\
I_{m-d+1}&&\frakc_{d-1}&\frakc_d&\frakc_1&\ldots&\frakc_{d-3}&\frakc_{d-2}&&\ldots\\
I_{m-d}&&\frakc_d&\frakc_1&\frakc_2&\ldots&\frakc_{d-2}&\frakc_{d-1}&&\ldots
\vspace{5pt}\\
\vdots&&\vdots&\vdots&\vdots&&\vdots&\vdots
\end{array}
\end{equation}
shows the coloring on each subinterval of the left vertical edge of each square face of the $n$-square-$b$ surface.
The $d\times d$ sub-array on the top left repeats throughout the whole array, with periodicty of the coloring vertically and horizontally.
This explains the terminology Double Periodic Coloring Algorithm.
\end{remark}

It becomes particularly interesting if the GCD Criterion holds, so that the integer $d$ also divides the parameter $\Upsilon(m;\alpha)$.
Note that this is the case for $n=2$ in each of Figures 2.2--2.5, and in each case, we are able to give $2$ non-trivial $\alpha$-flow invariant subsets of the $2$-square-$b$ surface.
The next lemma is a far-reaching generalization of this observation.

\begin{lem}\label{lem21}
Suppose that an integer $d$ divides both $n$ and~$m$.
Then the Double Periodic Coloring Algorithm gives rise to $d$ non-trivial $\alpha$-flow invariant subsets of the $n$-square-$b$ surface if and only if $d$ also divides $\Upsilon(m;\alpha)$.
\end{lem}

\begin{proof}
The $d$-coloring $C(1)$ from the Double Periodic Coloring Algorithm also gives the periodic $d$-coloring $C_0$ of the far left vertical edge of the $n$-square-$b$ surface, viewed as the unit torus $[0,1)$, with $m$ division points given by \eqref{eq2.4} or \eqref{eq2.5}.
In particular, the color pattern from the top is $\frakc_1,\ldots,\frakc_d$, with periodic repetition until it reaches the bottom.

Let $C^*$ denote a new $d$-coloring of the unit torus, obtained from $C_0$ by translating each point in $[0,1)$ by $\alpha$ modulo~$1$.
Noting \eqref{eq2.4} and \eqref{eq2.5}, it is clear that the division points of $C^*$ are given by
\begin{equation}\label{eq2.8}
\{q\alpha\},
\quad
q=1,\ldots,m.
\end{equation}
Thus the division points of $C_0$ and $C^*$ are essentially the same, apart from $0$ being replaced by $b=\{m\alpha\}$.

Let $C^{**}$ denote another new $d$-coloring of the unit torus, obtained from $C_0$ by keeping the colors in the interval $[b,1)=[\{m\alpha\},1)$ and replacing any color $\frakc_j$ in the interval $[0,b)=[0,\{m\alpha\})$ by the next color $\frakc_{j+1}$ along the chain $\frakc_1,\ldots,\frakc_d$ modulo~$d$.
Note that in $C^{**}$, the two sides of $0$ now have the same color, so $0$ is no longer a division point.
On the other hand, note that $b=\{m\alpha\}$ is not a division point of~$C_0$.
However, switching from $C_0$ to $C^{**}$, we switch the color below $b$ and keep the color above~$b$, so $b=\{m\alpha\}$ is clearly a division point of~$C^{**}$.

It follows that $C^*$ and $C^{**}$ are two $d$-colorings of the unit torus with precisely the same division points \eqref{eq2.8}.

We shall first show that the two $d$-colorings $C^*$ and $C^{**}$ are equal if and only if $d$ divides $\Upsilon(m;\alpha)$.
In view of the vertical periodicity of the $d$-colorings, to show that $C^*$ and $C^{**}$ are equal, it clearly suffices to check the equality of colors in just one interval.
We distinguish two cases.

Case 1:
Suppose that $b=\{m\alpha\}<\alpha$.
Since $b_{\Upsilon+1}=\alpha$ and $b=b_\nu<\alpha$, it follows that $1\le\nu\le\Upsilon$.
Recall that $b=b_\nu$ is not a division point of~$C_0$.
Hence
\begin{displaymath}
0=b_0<b_1<\ldots<b_{\nu-1}<b_{\nu+1}<\ldots<b_{\Upsilon+1}=\alpha
\end{displaymath}
are successive division points of~$C_0$.
Hence the intervals $[b_0,b_1)$ and $[b_{\Upsilon+1},b_{\Upsilon+2})$ have the same color $\frakc_d$ in $C_0$ if and only if $d$ divides~$\Upsilon$.
Next, note that the interval $[b_{\Upsilon+1},b_{\Upsilon+2})=[\alpha,b_{\Upsilon+2})$ is obtained from the interval $[b_0,b_1)=[0,b_1)$ by translation by $\alpha$ modulo~$1$.
It follows that $[b_{\Upsilon+1},b_{\Upsilon+2})$ has the same color $\frakc_d$ in $C^*$ as $[b_0,b_1)$ has in~$C_0$.
On the other hand, the interval $[b_{\Upsilon+1},b_{\Upsilon+2})=[\alpha,b_{\Upsilon+2})$ is not in the interval $[0,b)$, and so it has the same color in $C^{**}$ as in~$C_0$.
It now follows that the interval $[b_{\Upsilon+1},b_{\Upsilon+2})=[\alpha,b_{\Upsilon+2})$ has the same color $\frakc_d$ in $C^*$ as in $C^{**}$ if and only if $d$ divides $\Upsilon(m;\alpha)$.

Case 2:
Suppose that $b=\{m\alpha\}>\alpha$.
Since $b_{\Upsilon+1}=\alpha$ and $b=b_\nu>\alpha$, it follows that $\nu>\Upsilon+1$.
Hence
\begin{displaymath}
0=b_0<b_1<\ldots<b_{\Upsilon+1}=\alpha
\end{displaymath}
are successive division points of~$C_0$.
Hence the intervals $[b_0,b_1)$ and $[b_{\Upsilon+1},b_{\Upsilon+2})$ have different colors $\frakc_d$ and $\frakc_{d-1}$ respectively in $C_0$ if and only if $d$ divides~$\Upsilon$.
As in Case~1, $[b_{\Upsilon+1},b_{\Upsilon+2})$ has the same color $\frakc_d$ in $C^*$ as $[b_0,b_1)$ has in~$C_0$.
On the other hand, the interval $[b_{\Upsilon+1},b_{\Upsilon+2})=[\alpha,b_{\Upsilon+2})$ is in the interval $[0,b)$, and so its color in $C^{**}$ is the next color $\frakc_d$ along the chain $\frakc_1,\ldots,\frakc_d$ from its color $\frakc_{d-1}$ in~$C_0$.
It now follows that the interval $[b_{\Upsilon+1},b_{\Upsilon+2})=[\alpha,b_{\Upsilon+2})$ has the same color $\frakc_d$ in $C^*$ as in $C^{**}$ if and only if $d$ divides $\Upsilon(m;\alpha)$.

Finally, note that the equality of $C^*$ and $C^{**}$ and periodicity represent precisely the division of the $n$-square-$b$ surface into $d$ monochromatic sets that represent $d$ non-trivial $\alpha$-flow invariant subsets of the $n$-square-$b$ surface.
Indeed, $C^{**}$ exhibits the key difference between the intervals $[0,b)$ and $[b,1)$, that an $\alpha$-geodesic can freely cross the $b$-gate and is obstructed above it.
This completes the proof.
\end{proof}

\begin{remarks}
Lemma~\ref{lem21} basically says that from the viewpoint of equidistribution on the $n$-square-$b$ surface, the GCD Criterion can be considered an \textit{obstacle}.
Note, however, that any $\alpha$-geodesic with irrational $\alpha$ in any monochromatic subset of the $n$-square-$b$ surface is equidistributed in that subset.
We only need to recall that any $\alpha$-geodesic on the $n$-square-$b$ surface modulo~$1$ reduces to a torus line of slope $\alpha$ in the unit square.
Since $\alpha$ is irrational, this projected torus line is uniformly distributed in the unit square.

If the GCD Criterion holds, then we can always compute the corresponding visit densities, analogous to the cases illustrated in Figures~2.2--2.5.
It is not difficult to see that each visit density is necessarily of the form $u\alpha+v$, where $u,v\in\Zz$.
Since this is strictly between $0$ and~$1$, it follows that $u\ne0$, and since $\alpha$ is irrational, the visit density can never be equal to~$1/n$. 
Thus any half-infinite $\alpha$-geodesic on the $n$-square-$b$ surface is always unevenly distributed between the squares. 
\end{remarks}

Lemma~\ref{lem21} clearly establishes one half of the following result.

\begin{thm}\label{thm21}
Suppose that $b=\{m\alpha\}$, where $\alpha$ is irrational and $m$ is a positive integer.
Then any $\alpha$-geodesic on the $n$-square-$b$ surface is equidistributed on the surface if and only if the GCD Criterion fails.
\end{thm}

An interesting consequence of Theorem~\ref{thm21} is the following.
If $b=\{m\alpha\}$, where $\alpha$ is irrational and $m$ is a positive integer, and an
$\alpha$-geodesic on the $n$-square-$b$ surface is dense on the surface, then the geodesic exhibits the stronger property of equidistribution.

We shall prove the remainder of Theorem~\ref{thm21} later; see Section~\ref{sec4} and the end of Section~\ref{sec6}.

%
%

\section{More on the $2$-square-$b$ surface and beyond}\label{sec3}

We now consider the general case of the $n$-square-$b$ surface when $b\ne\{m\alpha\}$ for any $m\in\Zz$.
Here the answer is rather tricky.

For the original case $n=2$, the paper of Veech~\cite{V1} contains a study of the following special case of Question~1 where the test sets are simply the two squares of the $2$-square-$b$ surface.

\begin{q2}
Let $\alpha$ be an irrational number.
When can we guarantee that every half-infinite $\alpha$-geodesic on the $2$-square-$b$ surface is evenly distributed between the two constituent squares?
\end{q2}

In other words, assuming that a particle moves along the $\alpha$-geodesic with unit speed, under what condition can we guarantee that for every starting point, the left square is visited \textit{half the time}?
More precisely, we want the asymptotic visit-density of this particle to the left square of the $2$-square-$b$ surface to exist, and to be equal to~$1/2$.

Veech~\cite{V1} has the following positive answer to Question~2.

\begin{theorema}
Suppose that the slope $\alpha$ is badly approximable.
Suppose further that the gate-size $b\ne\{m\alpha\}$ for any $m\in\Zz$.
Then every half-infinite $\alpha$-geodesic on the $2$-square-$b$ surface is evenly distributed between the two constituent squares.
\end{theorema}

We recall that \textit{badly approximable} numbers are characterized by the property that the continued fraction digits have a common upper bound.
A well-known subclass of badly approximable numbers is the set of all quadratic irrationals, \textit{i.e.}, real algebraic numbers of degree~$2$, which are characterized by the property that the continued fraction expansions are eventually periodic.

Given a badly approximable slope~$\alpha$, the condition $b\ne\{m\alpha\}$ for any $m\in\Zz$ in Theorem~A excludes a countable set of values of~$b$.
For these excluded values of~$b$, we now have a complete understanding of the situation.
As explained in the Remarks after the proof of Lemma~\ref{lem21}, what happens depends on the Double Even Criterion.
Suppose that the Double Even Criterion fails.
Then it follows as a consequence of Theorem~\ref{thm21} that any half-infinite $\alpha$-geodesic is evenly distributed between the two constituent squares. 
On the other hand, suppose that the Double Even Criterion holds.
Then each constituent square has a well-defined visit-density, depending on the starting point of the $\alpha$-geodesic, which is \textit{never} equal to~$1/2$, so the half-infinite $\alpha$-geodesic is never evenly distributed between the two constituent squares.

If the slope $\alpha$ is not badly approximable, then Veech~\cite{V1} has the following very interesting negative result.

\begin{theoremb}
Suppose that the irrational slope $\alpha$ is not badly approximable.
Then there exists an explicit  construction of an uncountable set of values $b$ with strong violation of uniformity in the sense that for some half-infinite
$\alpha$-geodesics on such a $2$-square-$b$ surface, the visit-densities of the constituent squares do not even exist.
\end{theoremb}

So far, we have considered a fixed irrational slope $\alpha$ and asked the question of what values of $b$ lead to half-infinite $\alpha$-geodesics on the
$2$-square-$b$-surface that are evenly distributed between the two constituent squares.

Suppose instead that we consider a fixed gate size~$b$.
Then it is reasonable to ask what irrational slopes $\alpha$ give rise to half-infinite $\alpha$-geodesics on the $2$-square-$b$ surface that are evenly distributed between the two constituent squares.

Veech~\cite{V1} has the following result which shows that $2$-square-$b$ surfaces with rational values of $b$ are exceptional.

\begin{theoremc}
Suppose that the number $b$ is rational.
Then for any irrational slope~$\alpha$, every half-infinite $\alpha$-geodesic on the $2$-square-$b$ surface gives rise to equal visit-densities of the two constituent squares.
\end{theoremc}

We also have the following negative result of Masur and Smillie on the $2$-square-$b$ surface; see \cite[Theorem~3.2]{MT} or \cite[Theorem~2]{Mas}.

\begin{theoremd}
Suppose that the number $b$ is irrational.
Then there exist uncountably many slopes $\alpha$ such that for almost every starting point, a half-infinite $\alpha$-geodesic on the 2-square-$b$ surface is not uniformly distributed.
\end{theoremd}

Note that the uncountable set of \textit{bad} slopes $\alpha$ in Theorem~D can be extended to a set of positive Hausdorff measure, but not to a set of positive Lebesgue measure.
This follows from a well known general result of Kerckhoff, Masur, and Smillie~\cite{KMS} concerning geodesic flow on any rational polygonal surface.
This important general theorem, which works for almost every slope, unfortunately does not say anything about any explicit slope, which is our main interest.
For more about non-integrable flat dynamical systems, the reader is referred to the survey papers \cite{MT} and~\cite{Z}. 

Theorems A--D are very satisfactory results that give us a very good understanding of the distribution of half-infinite $\alpha$-geodesics on the
$2$-square-$b$ surface.
We can view this as the mod~$2$ case.
However, the corresponding mod~$n$ version, concerning the $n$-square-$b$ surface, remains open for any integer $n\ge3$.

Veech~\cite{V1} has asked the question of whether or not his method can be extended to prove the mod~$n$ versions of Theorem~A for $n\ge3$.
Here we can establish such a result, but we do not use Veech's method which is quite complicated. 
In fact, we can prove the following stronger result that answers the mod~$n$ analog of Question~1.

\begin{thm}\label{thm31}
Suppose that $n\ge2$ and the slope $\alpha$ is badly approximable.
Suppose also that the gate-size $b\ne\{m\alpha\}$ for any $m\in\Zz$.
Then every half-infinite $\alpha$-geodesic on the $n$-square-$b$ surface is uniformly distributed.
\end{thm}

Furthermore, we can establish a far-reaching generalization of Theorem~\ref{thm31}.
We consider the larger class of \textit{flat finite polysquare}, \textit{or square tiled\/}, \textit{translation surfaces with $b$-rational gates}.

A \textit{finite polysquare}, \textit{or square tiled\/}, \textit{region} is a connected, but not necessarily simply-connected, polygon $P$ on the plane which is tiled with unit squares, assumed to be closed, that we call the \textit{atomic squares} of~$P$, and which satisfies the following conditions:

(i) Any two atomic squares in $P$ either are disjoint, or intersect at a single point, or have a common edge.

(ii) Any two atomic squares in $P$ are joined by a chain of atomic squares where any two neighbors in the chain have a common edge.

To turn a given finite polysquare region $P$ into a \textit{flat finite polysquare translation surface}~$\PPP$, we need identification of pairs of horizontal edges as well as identification of pairs of vertical edges.
In Figure~3.1, we show examples of the identification of horizontal edges on the two leftmost columns of atomic squares as well as examples of the identification of vertical edges on the two topmost rows of atomic squares.

Note that the finite polysquare surface $\PPP$ may have \textit{holes}, and we also allow \textit{whole barriers} which are horizontal or vertical \textit{walls} that consist of one or more boundary edges of atomic squares.
For example, the finite polysquare surface in Figure~3.1 has $32$ atomic squares, $2$ holes as well as $3$ horizontal walls and $4$ vertical walls.

\begin{displaymath}
\begin{array}{c}
\includegraphics{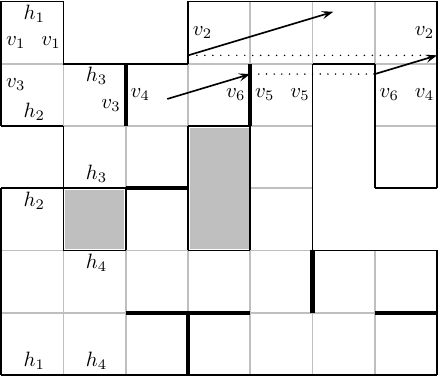}
\\
\mbox{Figure 3.1: a flat finite polysquare translation surface}
\end{array}
\end{displaymath}

Geodesic flow on a flat finite polysquare translation surface is always $1$-direction linear flow.

\begin{remark}
Geodesic flow on a general finite polysquare surface may sometimes be a $4$-direction flow.
Consider, for example, geodesic flow on the cube surface.
It is well known that this $4$-direction geodesic flow on the cube surface can be converted to a $1$-direction geodesic flow by using a $4$-copy construction, where we take $4$ \textit{rotated} copies of the cross-shaped net of the cube surface, and glue together corresponding edges in the different copies to obtain a flat finite polysquare translation surface.
Indeed, an analog of this $4$-copy construction works for any finite polysquare surface with $4$-direction geodesic flow.

Meanwhile, it can also be shown that any $4$-direction billiard orbit in a finite polysquare region is equivalent to $1$-direction geodesic flow in a corresponding flat finite polysquare translation surface.

For a more detailed discussion, see \cite[Section~1]{BC}.

It is therefore sufficient to study $1$-direction geodesic flow on flat finite polysquare translation surfaces.
\end{remark}

The $2$-dimensional continuous Kronecker--Weyl equidistribution theorem for the torus line in a square leads to an interesting \textit{uniform-periodic dichotomy}, in the sense that every torus line with irrational slope is uniformly distributed, whereas every torus line with rational slope is periodic.

We have the following remarkable extension of this classical result by Gutkin and Veech about 70 years later; see \cite{G,V2,V3}.

\begin{theoreme}
On any flat finite polysquare translation surface, every half-infinite $1$-direction geodesic with irrational slope is uniformly distributed, whereas every half-infinite $1$-direction geodesic with rational slope is periodic.
\end{theoreme}

Note that we consider here only half-infinite $1$-direction geodesics, as we need to exclude any geodesic that hits a singularity of the polysquare surface after which there is no well defined unique continuation.

If the gate size $b$ is irrational, then the $2$-square-$b$ surface is not a polysquare surface, so Theorem~E does not apply.
Furthermore, as Theorem~B shows, for any irrational slope which is not badly approximable, there is clearly no uniform-periodic dichotomy.
There is an uncountable set of values $b$ for which even the simplest test sets, namely the two constituent squares of the $2$-square-$b$ surface, violate uniformity.
On the other hand, any half-infinite $1$-direction geodesic with irrational slope on any $2$-square-$b$ surface cannot be periodic.

As in Theorem~A, we study uniformity in the case of badly approximable slopes. 
Theorem~\ref{thm31} is such a result.
Next we formulate a far-reaching generalization of it, to the class of flat finite polysquare translation surfaces with $b$-rational gates.

An example of such a surface is the $(L;b)$-surface, an L-shaped $4$-square surface with three $b$-size gates and one $b/2$-size gate, as shown in Figure~3.2.

\begin{displaymath}
\begin{array}{c}
\includegraphics{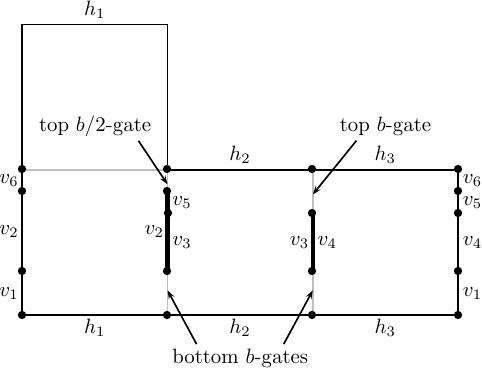}
\\
\mbox{Figure 3.2: the $(L;b)$-surface}
\end{array}
\end{displaymath}

Here the left vertical edge of the bottom middle atomic square has two division points $b$ and $1-b/2=\{-b/2\}$ which determine the left bottom
$b$-gate between $0$ and~$b$, as well as the top $b/2$-gate between $1-b/2$ and~$1$, separated by the fractional vertical barrier between $b$ and $1-b/2$.
On the other hand, the left vertical edge of the bottom right atomic square has two division points $b$ and $1-b=\{-b\}$ which determine the left bottom $b$-gate between $0$ and~$b$, as well as the top $b$-gate between $1-b$ and~$1$, separated by the fractional vertical barrier between $b$
and~$1-b$.

We now extend the class of flat finite polysquare translation surfaces to the larger class of \textit{flat finite polysquare-$b$-rational translation surfaces} by following and then extending the pattern of the $(L;b)$-surface.
For any vertical side of an atomic square, we may place any number of $b$-rational division points located at distance $\{rb\}$ from the bottom of the edge, where $0<b<1$ is fixed and $r$ is a rational number.
These division points, often called the \textit{division numbers}, determine vertical gates separated by fractional vertical barriers, where every gate and barrier is a subinterval of the vertical edge, with endpoints which are $b$-rational division points.
To obtain a translation surface, we identify pairs of horizontal edges and pairs of vertical edges in an appropriate manner.
Then geodesic flow is $1$-direction linear flow.

The flat finite polysquare-$b$-rational translation surface in Figure~3.3 is modified from the flat finite polysquare translation surface in Figure~3.1 in this way.
We have not included the edge identifications.

\begin{displaymath}
\begin{array}{c}
\includegraphics{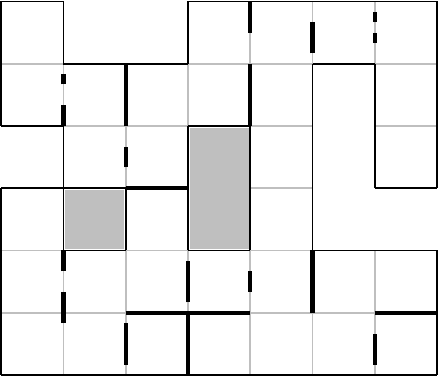}
\\
\mbox{Figure 3.3: a flat finite polysquare-$b$-rational translation surface}
\end{array}
\end{displaymath}

In Sections \ref{sec5}--\ref{sec7}, we shall prove the following generalization of Theorem~\ref{thm31}.

\begin{thm}\label{thm32}
Let $\PPP$ be a flat finite polysquare-$b$-rational translation surface, where $b$ is irrational, and with division numbers $\{r_ib\}$, $i=1,\ldots,R$, where each $r_i$ is rational.
Let $\alpha$ be a badly approximable number such that $\{r_ib\}\ne\{m\alpha\}$ for any $i=1,\ldots,R$ and $m\in\Zz\setminus\{0\}$.
Then every half-infinite $\alpha$-geodesic on $\PPP$ is uniformly distributed.
\end{thm}

\begin{remark}
The study of geodesic flow on a flat finite polysquare-$b$-rational translation surface is related to a suitable generalization of the Veech $2$-circle problem.
Here the number of circles corresponds to the number of vertical streets of the underlying finite polysquare surface, and the circumference of a circle is the length of the vertical street that corresponds to it.
This remains the case if the division numbers are replaced by a finite set of real numbers, at least one of which is irrational, resulting in surfaces that can be more general than polysquare-$b$-rational translation surfaces.
Unfortunately, we are not able to extend Theorem~\ref{thm32} to this more general setting, as we are not able to establish a suitable generalization of the separation lemma as given by Lemma~\ref{lem52}.
\end{remark}

We can show that billiard in any finite polysquare-$b$-rational region is equivalent to a $1$-direction geodesic flow on a corresponding flat finite
polysquare-$b$-rational translation surface.
This follows from a generalization of the concept of \textit{unfolding}, pioneered by K\"{o}nig and Sz\"{u}cs~\cite{KS} in 1913, to show that billiard in the unit square is equivalent to $1$-direction geodesic flow in the square torus.
Indeed, it can be shown that billiard in any finite polysquare region is equivalent to a $1$-direction geodesic flow on a corresponding flat finite polysquare translation surface; see, for instance, \cite[Section~1]{BC}.

Thus we have immediately the following result concerning billiards.

\begin{thm}\label{thm33}
Let $P$ be a finite polysquare-$b$-rational translation region, where $b$ is irrational, and with division numbers $\{r_ib\}$, $i=1,\ldots,R$, where each $r_i$ is rational.
Let $\alpha$ be a badly approximable number such that $\{r_ib\}\ne\{m\alpha\}$ for any $i=1,\ldots,R$ and $m\in\Zz\setminus\{0\}$.
Then every half-infinite billiard orbit in $P$ with initial slope $\alpha$ is uniformly distributed.
\end{thm}

Next we return to the $2$-square-$b$ surface and the somewhat negative Theorem~B of Veech.
If the irrational slope $\alpha$ is not badly approximable, then there exists an uncountable set of values of $b$ such that the visit-densities of the constituent squares do not even exist.
For such gate-sizes~$b$, it is perhaps natural then to call them \textit{bad}.
This raises the question of finding a quantitative description of this phenomenon, that extreme violation of uniformity can be exhibited by a concrete geodesic.

We shall give such a quantitative result which demonstrates serious violations of uniformity.
For appropriate pairs of the parameters $\alpha$ and~$b$, we shall construct a half-infinite $\alpha$-geodesic $\LLL$ on the $2$-square-$b$ surface which demonstrates \textit{extra-large one-sidedness exhibited in an alternating way}.
Such a geodesic $\LLL$ also violates any form of \textit{quasi-periodicity}.
Using a completely different method from those that give Theorems B and~D, we shall prove in Sections \ref{sec8} and~\ref{sec9} the following result.

For any $2$-square-$b$ surface, we denote by $\ls(b)$ the left constituent square of the surface, and by $\rs(b)$ the right constituent square of the surface.

\begin{thm}\label{thm34}
Suppose that $\eps>0$ is arbitrarily small but fixed, and that $\alpha\in(0,1)$ is any irrational number with continued fraction
\begin{displaymath}
\alpha=\frac{1}{a_1+\frac{1}{a_2+\frac{1}{a_3+\cdots}}}=[a_1,a_2,a_3,\ldots],
\end{displaymath}
where the digits $a_1,a_2,a_3,\ldots$ satisfy the condition
\begin{equation}\label{eq3.1}
\sum_{i=1}^\infty\frac{1}{a_i}<\frac{\eps}{300}.
\end{equation}

There exists an explicitly given gate-size $\beta_0=\beta_0(\alpha)$ such that the $\alpha$-geodesic $\LLL_0(t)$, starting from some explicitly given point on the $2$-square-$\beta_0$ surface, satisfies the following simultaneously, where $C$ is any positive integer satisfying $C<200/\eps$:

\emph{(i)}
There exists an infinite sequence $T^*_n$, $n=1,2,3,\ldots,$ of positive real numbers satisfying $T^*_{n+1}>2T^*_n$ such that for every integer $n=1,2,3,\ldots$ and for every integer $b=0,1,\ldots,C$ apart from $b=1$,
\begin{equation}\label{eq3.2}
\frac{1}{T^*_n}\vert\{t\in[bT^*_n,(b+1)T^*_n]:\LLL_0(t)\in\ls(\beta_0)\}\vert>1-\eps,
\end{equation}
with an overwhelming bias for the left constituent square of the surface, as well as
\begin{equation}\label{eq3.3}
\frac{1}{T^*_n}\vert\{t\in[T^*_n,2T^*_n]:\LLL_0(t)\in\rs(\beta_0)\}\vert>1-\eps,
\end{equation}
with an overwhelming bias for the right constituent square of the surface.

\emph{(ii)}
There exists an infinite sequence $T^{**}_n$, $n=1,2,3,\ldots,$ of positive real numbers satisfying $T^{**}_{n+1}>2T^{**}_n$ such that for every integer $n=1,2,3,\ldots$ and for every integer $b=0,1,\ldots,C$ apart from $b=2$,
\begin{equation}\label{eq3.4}
\frac{1}{T^{**}_n}\vert\{t\in[bT^{**}_n,(b+1)T^{**}_n]:\LLL_0(t)\in\ls(\beta_0)\}\vert>1-\eps,
\end{equation}
with an overwhelming bias for the left constituent square of the surface, as well as
\begin{equation}\label{eq3.5}
\frac{1}{T^{**}_n}\vert\{t\in[2T^{**}_n,3T^{**}_n]:\LLL_0(t)\in\rs(\beta_0)\}\vert>1-\eps,
\end{equation}
with an overwhelming bias for the right constituent square of the surface.

On the other hand, for any large but fixed positive integer~$n$, there exists another explicitly given gate-size $\beta_1=\beta_1(\alpha,n)$ such that
$\vert\beta_1-\beta_0\vert<\eps$ and the $\alpha$-geodesic $\LLL_1(t)$, starting from some explicitly given point on the $2$-square-$\beta_1$ surface, satisfies the following simultaneously:

\emph{(iii)}
There exists a finite sequence $W_1,\ldots,W_n$ of positive real numbers satisfying $W_{i+1}>2W_i$ whenever $i<n$ such that for every integer
$i=1,\ldots,n,$
\begin{equation}\label{eq3.6}
\frac{1}{W_i}\vert\{t\in[0,W_i]:\LLL_1(t)\in\ls(\beta_1)\}\vert>1-\eps,
\end{equation}
with an overwhelming bias for the left constituent square of the surface, as well as
\begin{equation}\label{eq3.7}
\frac{1}{W_i}\vert\{t\in[W_i,2W_i]:\LLL_1(t)\in\rs(\beta_1)\}\vert>1-\eps,
\end{equation}
with an overwhelming bias for the right constituent square of the surface.

\emph{(iv)}
There exists a positive threshold $W^\star$ such that for every positive real number $W>W^\star$,
\begin{equation}\label{eq3.8}
\frac{1}{W}\vert\{t\in[0,W]:\LLL_1(t)\in\ls(\beta_1)\}\vert>\frac{2}{3}-\eps,
\end{equation}
with a significant bias for the left constituent square of the surface.
\end{thm}

Removing the vertical barrier on the $2$-square-$\beta_0$ surface or $2$-square-$\beta_1$ surface leads to a polysquare surface which is an integrable rectangle surface.
It can then be shown that applying some slow growth conditions on the continued fraction digits of $\alpha$ without violating \eqref{eq3.1}, we obtain essentially best possible time-quantitative uniformity for any geodesic with slope $\alpha$, with polylogarithmic error term, on this integrable surface.
Thus the barrier is the root cause of the polarizingly different uniformity properties of the two geodesics with the same slope.
We omit the details.

%
%

\section{Interval exchange transformation and ergodicity}\label{sec4}

A common tool in the proofs of Theorems \ref{thm21} and~\ref{thm32} is the concept of an interval exchange transformation which represents a natural discretization of the linear flow of slope $\alpha$ on the flat translation surface.
As a first step, we need to exhibit ergodicity of this transformation, and this step is summarized by Lemmas \ref{lem41} and~\ref{lem51}.

We discuss this standard technique here, and also illustrate a second key idea, which is an application of the so-called \textit{$3$-distance theorem}, as given in Lemma~\ref{lem42}, an idea used earlier in related work by Boshernitzan \cite[Theorem~7.2 $(r=2)$]{B}.

Theorem~\ref{thm32} concerns flat finite polysquare-$b$-rational translation surfaces which often can be far more complicated than the $n$-square-$b$ surface in Theorem~\ref{thm21}.
Thus to illustrate the idea of an interval exchange transformation, we shall use instead the special case of the $(L;b)$-surface shown in Figure~3.2, as this special case already well captures the whole difficulty of the situation in general.
We shall further assume that $0<b<\alpha<1/2$, where $\alpha$ is a given irrational slope.

Before we introduce the interval exchange transformation, we first consider the effect of the $\alpha$-flow.
For convenience, we shall assume that all the gates and barriers are closed at the bottom end and open at the top end.

Let $w_1,w_2,w_3,w_4$ denote the left vertical edges of the $4$ atomic squares that make up the $(L;b)$-surface, as shown in Figure~4.1.

\begin{displaymath}
\begin{array}{c}
\includegraphics{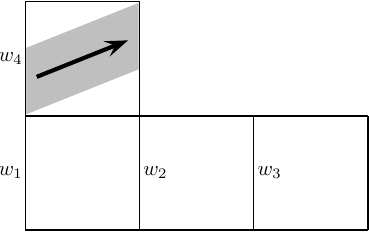}
\\
\mbox{Figure 4.1: the vertical edges $w_1,w_2,w_3,w_4$ and $\alpha$-flow}
\end{array}
\end{displaymath}

For the vertical edge~$w_1$, we denote by $w_1(0)$ and $w_1(1)$ the bottom endpoint and top endpoint of $w_1$ respectively, and denote by $w_1(x)$, where $0<x<1$, the point on $w_1$ which is a distance $x$ from $w_1(0)$.
Furthermore, for any set $S\subset[0,1]$, we let
\begin{displaymath}
w_1S=\{w_1(x):x\in S\},
\end{displaymath}
so that $w_1[0,1]=w_1$.

We now repeat this for the other $3$ vertical edges $w_2,w_3,w_4$.

Using Figures 3.2 and~4.1, we see that the $\alpha$-flow maps the interval $w_4[0,1-\alpha)$ to the interval $w_4[\alpha,1)$.
We denote this by
\begin{displaymath}
w_4[0,1-\alpha)
\mapsto
w_4[\alpha,1).
\end{displaymath}
Careful analysis now shows that the effect of the $\alpha$-flow is summarized by a collection of increasing bijective linear mappings
\begin{align}
w_1[0,1-\alpha-\tfrac{b}{2})
&\mapsto
w_1[\alpha,1-\tfrac{b}{2}),
\label{eq4.1}
\\
w_1[1-\alpha-\tfrac{b}{2},1-\alpha)
&\mapsto
w_2[1-\tfrac{b}{2},1),
\label{eq4.2}
\\
w_1[1-\alpha,1)
&\mapsto
w_4[0,\alpha),
\label{eq4.3}
\\
w_2[0,1-\alpha-b)
&\mapsto
w_2[\alpha,1-b),
\label{eq4.4}
\\
w_2[1-\alpha-b,1-\alpha)
&\mapsto
w_3[1-b,1),
\label{eq4.5}
\\
w_2[1-\alpha,1-\alpha+b)
&\mapsto
w_3[0,b),
\label{eq4.6}
\\
w_2[1-\alpha+b,1)
&\mapsto
w_2[b,\alpha),
\label{eq4.7}
\\
w_3[0,1-\alpha-b)
&\mapsto
w_3[\alpha,1-b),
\label{eq4.8}
\\
w_3[1-\alpha-b,1-\alpha-\tfrac{b}{2})
&\mapsto
w_2[1-b,1-\tfrac{b}{2}),
\label{eq4.9}
\\
w_3[1-\alpha-\tfrac{b}{2},1-\alpha)
&\mapsto
w_1[1-\tfrac{b}{2},1),
\label{eq4.10}
\\
w_3[1-\alpha,1-\alpha+b)
&\mapsto
w_1[0,b),
\label{eq4.11}
\\
w_3[1-\alpha+b,1)
&\mapsto
w_3[b,\alpha),
\label{eq4.12}
\\
w_4[0,1-\alpha)
&\mapsto
w_4[\alpha,1),
\label{eq4.13}
\\
w_4[1-\alpha,1-\alpha+b)
&\mapsto
w_2[0,b),
\label{eq4.14}
\\
w_4[1-\alpha+b,1)
&\mapsto
w_1[b,\alpha).
\label{eq4.15}
\end{align}
We next identify the edges $w_1,w_2,w_3,w_4$ with the intervals $[0,1),[1,2),[2,3),[3,4)$ respectively.
Using this identification and \eqref{eq4.1}--\eqref{eq4.15}, the effect of the $\alpha$-flow can then be described by a piecewise linear map
$T:[0,4)\to[0,4)$, where
\begin{align}
T([0,1-\alpha-\tfrac{b}{2}))
&=
[\alpha,1-\tfrac{b}{2}),
\label{eq4.16}
\\
T([1-\alpha-\tfrac{b}{2},1-\alpha))
&=
[2-\tfrac{b}{2},2),
\label{eq4.17}
\\
T([1-\alpha,1))
&=
[3,3+\alpha),
\label{eq4.18}
\\
T([1,2-\alpha-b))
&=
[1+\alpha,2-b),
\label{eq4.19}
\\
T([2-\alpha-b,2-\alpha))
&=
[3-b,3),
\label{eq4.20}
\\
T([2-\alpha,2-\alpha+b))
&=
[2,2+b),
\label{eq4.21}
\\
T([2-\alpha+b,2))
&=
[1+b,1+\alpha),
\label{eq4.22}
\\
T([2,3-\alpha-b))
&=
[2+\alpha,3-b),
\label{eq4.23}
\\
T([3-\alpha-b,3-\alpha-\tfrac{b}{2}))
&=
[2-b,2-\tfrac{b}{2}),
\label{eq4.24}
\\
T([3-\alpha-\tfrac{b}{2},3-\alpha))
&=
[1-\tfrac{b}{2},1),
\label{eq4.25}
\\
T([3-\alpha,3-\alpha+b))
&=
[0,b),
\label{eq4.26}
\\
T([3-\alpha+b,3))
&=
[2+b,2+\alpha),
\label{eq4.27}
\\
T([3,4-\alpha))
&=
[3+\alpha,4),
\label{eq4.28}
\\
T([4-\alpha,4-\alpha+b))
&=
[1,1+b),
\label{eq4.29}
\\
T([4-\alpha+b,4))
&=
[b,\alpha),
\label{eq4.30}
\end{align}
and each of \eqref{eq4.16}--\eqref{eq4.30} represents an increasing bijective linear map.
This map $T$ is known as the interval exchange transformation of the $\alpha$-flow on the $(L;b)$-surface.
It is clear that $T$ preserves Lebesgue measure.

A quick inspection of \eqref{eq4.16}--\eqref{eq4.30} shows that $T$ has many points of discontinuity.
However, if we take them modulo~$1$, then their values are given by
\begin{displaymath}
0,
\quad
1-\alpha-b,
\quad
1-\alpha-\tfrac{b}{2},
\quad
1-\alpha,
\quad
1-\alpha+b.
\end{displaymath}
We refer to these $5$ numbers as the singularities of $T$ modulo~$1$, or simply the singularities.
These are precisely the division numbers shifted by~$-\alpha$ modulo~$1$, together with $0$ and $1-\alpha$.

Suppose now that $\PPP$ is a flat finite polysquare-$b$-rational translation surface, with division numbers $\{r_ib\}$, $i=1,\ldots,R$, where each $r_i$ is rational.
Let
\begin{displaymath}
T=T_\alpha(\PPP;\{r_ib\},i=1,\ldots,R)
\end{displaymath}
denote the interval exchange transformation of the $\alpha$-flow on this surface.
Suppose that $s$ denotes the number of atomic squares in the underlying polysquare region.
Then $T:[0,s)\to[0,s)$ is a piecewise linear bijective map that preserves Lebesgue measure, and the singularities of $T$ modulo~$1$ are
\begin{equation}\label{eq4.31}
0,
\quad
1-\alpha
\quad\mbox{and}\quad
\{r_ib-\alpha\},
\quad
i=1,\ldots,R.
\end{equation}

For the remainder of this section, we concentrate on Theorem~\ref{thm21} concerning the $n$-square-$b$ surface in the special case $b=\{m\alpha\}$ for some integer $m\ge2$.
Our goal here is to establish equidistribution when the GCD Criterion fails.
We assume that $0<\alpha<1$.

The interval exchange transformation is a piecewise linear map $T:[0,n)\to[0,n)$.
It has $3$ singularities $0$, $1-\alpha$ and $\{(m-1)\alpha\}$ modulo~$1$.
The inverse transformation $T^{-1}$ has $3$ singularities $0$, $\alpha$ and $\{m\alpha\}$ modulo~$1$.

The simplest special case is $n=m=2$ with $b=\{2\alpha\}$ where $1/2<\alpha<1$.
It is easy to check that the Double Even Criterion, \textit{i.e.}, the GCD Criterion for $n=2$, fails.

To bring us one step closer to a complete proof of Theorem~\ref{thm21}, we have the following result on ergodicity.

\begin{lem}\label{lem41}
Consider $\alpha$-flow on the $n$-square-$b$ surface with $b=\{m\alpha\}$ for some integer $m\ge2$.
Suppose that the GCD Criterion fails.
Then the interval exchange transformation $T=T_{\alpha;m}:[0,n)\to[0,n)$ is ergodic.
\end{lem}

\begin{proof}
We shall prove this by contradiction.
Assume on the contrary that $T$ is not ergodic.
Then there exists a $T$-invariant measurable subset $S_0\subset[0,n)$ such that $0<\meas(S_0)<n$, where $\meas$ denotes $1$-dimensional Lebesgue measure.
Since $T$ reduces modulo~$1$ to irrational rotation on the unit interval with the same~$\alpha$, it follows that $T$ modulo~$1$ is ergodic, and so $S_0$ modulo~$1$ is the unit interval $[0,1)$, implying that $\meas(S_0)$ is an integer strictly between $0$ and $n$.

The irrational slope $\alpha\in(0,1)$ has an infinite continued fraction expansion
\begin{equation}\label{eq4.32}
\alpha=[a_1,a_2,a_3,\ldots]=\frac{1}{a_1+\frac{1}{a_2+\frac{1}{a_3+\cdots}}},
\end{equation}
where $a_i\ge1$, $i=1,2,3,\ldots,$ are integers.
The rational numbers
\begin{equation}\label{eq4.33}
\frac{p_k}{q_k}=\frac{p_k(\alpha)}{q_k(\alpha)}=[a_1,\ldots,a_k],
\quad
k=1,2,3,\ldots,
\end{equation}
where $p_k\in\Zz$ and $q_k\in\Nn$ are coprime, are the $k$-convergents of~$\alpha$.
It is well known that they give rise to the best rational approximations of the irrational number~$\alpha$, and we have
\begin{equation}\label{eq4.34}
\frac{p_0}{q_0}
<\frac{p_2}{q_2}
<\frac{p_4}{q_4}
<\ldots
<\alpha
<\ldots
<\frac{p_5}{q_5}
<\frac{p_3}{q_3}
<\frac{p_1}{q_1},
\end{equation}
with $p_0=0$ and $q_0=1$.

Let $\Vert y\Vert$ denote the distance of a real number $y$ from the nearest integer.
We shall make use of the fact that for an irrational number~$\alpha$, the sequence
\begin{displaymath}
\min_{1\le k\le n}\Vert k\alpha\Vert,
\quad
n=1,2,3,\ldots,
\end{displaymath}
is well described by the continued fraction expansion of~$\alpha$.

For every $k=0,1,2,3,\ldots,$ we have
\begin{equation}\label{eq4.35}
\Vert q\alpha\Vert\ge\Vert q_k\alpha\Vert,
\quad
1\le q<q_{k+1},
\end{equation}
\begin{displaymath}
\Vert q_{k+1}\alpha\Vert<\Vert q_k\alpha\Vert,
\end{displaymath}
as well as
\begin{equation}\label{eq4.36}
\frac{1}{q_{k+1}+q_k}\le\Vert q_k\alpha\Vert\le\frac{1}{q_{k+1}}.
\end{equation}
Indeed, the sequences $p_k$ and $q_k$, $k=0,1,2,3,\ldots,$ are given by the initial values
\begin{displaymath}
p_0=0,
\quad
p_1=1,
\quad
q_0=1,
\quad
q_1=a_1,
\end{displaymath}
and the recurrence relations
\begin{equation}\label{eq4.37}
p_{k+1}=a_{k+1}p_k+p_{k-1},
\quad
q_{k+1}=a_{k+1}q_k+q_{k-1},
\quad
k\ge1.
\end{equation}
We also have
\begin{displaymath}
p_{k-1}q_k-q_{k-1}p_k=(-1)^k,
\quad
k\ge1.
\end{displaymath}
On the other hand, using \eqref{eq4.34} and \eqref{eq4.37}, it is easy to show that
\begin{equation}\label{eq4.38}
\Vert q_{k+1}\alpha\Vert+a_{k+1}\Vert q_k\alpha\Vert=\Vert q_{k-1}\alpha\Vert.
\end{equation}

The following result is known as the $3$-distance theorem.
This surprising geometric fact, formulated as a conjecture by Steinhaus, has many proofs, by S\'{o}s~\cite{So1,So2}, Swierczkowski~\cite{Sw},
Sur\'{a}nyi~\cite{Su}, Halton~\cite{H} and Slater~\cite{Sl}, with others published more recently.

\begin{lem}\label{lem42}
Consider the $N+1$ numbers $0,\alpha,2\alpha,3\alpha,\ldots,N\alpha$ modulo~$1$ in the unit torus/circle $[0,1)$, leading to an $(N+1)$-partition.
This partition exhibits at most $3$ different distances between neighboring points.
Furthermore, every positive integer $N$ can be expressed uniquely in the form
\begin{displaymath}
N=\mu q_k+q_{k-1}+r,
\quad
\mbox{with $1\le \mu\le a_{k+1}$ and $0\le r<q_k$},
\end{displaymath}
in terms of the continued fraction \eqref{eq4.32} of~$\alpha$ and its convergents \eqref{eq4.33}, with the convention that $q_0=1$ and $q_{-1}=0$.
Then
\begin{itemize}
\item[(i)] the distance $\Vert q_k\alpha\Vert$ shows up precisely $N+1-q_k$ times;
\item[(ii)] the distance $\Vert q_{k-1}\alpha\Vert-\mu\Vert q_k\alpha\Vert$ shows up precisely $r+1$ times; and
\item[(iii)] the distance $\Vert q_{k-1}\alpha\Vert-(\mu-1)\Vert q_k\alpha\Vert$ shows up precisely $q_k-r-1$ times.
\end{itemize}
\end{lem}

Given an integer $k\ge1$, let $\AAA_k(\alpha)$ denote the partition of the unit torus/circle $[0,1)$ with $q_{k+1}=q_{k+1}(\alpha)$ division points
$\{q\alpha\}$, $-1\le q\le q_{k+1}-2$.
Note that the choices $q=-1,0$ in $\{q\alpha\}$ represent two of the singularities of the interval exchange transformation $T$ restricted to the interval $[0,1)$.

A consequence of the special choice $N=q_{k+1}-1$ is that the $3$-distance theorem simplifies to a $2$-distance theorem.
This in turn leads to some very useful information concerning the distances between neighboring points of the $q_{k+1}$-partition $\AAA_k(\alpha)$ of the unit torus/circle $[0,1)$.
Indeed, using the second recurrence relation in \eqref{eq4.37}, we have
\begin{displaymath}
N=q_{k+1}-1=a_{k+1}q_k+q_{k-1}-1=\mu q_k+q_{k-1}+r,
\end{displaymath}
with $\mu=a_{k+1}-1$ and $r=q_k-1$.
Since $q_k-r-1=0$, it follows from the $3$-distance theorem that there are only two distances
\begin{equation}\label{eq4.39}
\Vert q_k\alpha\Vert
\quad\mbox{and}\quad
\Vert q_{k-1}\alpha\Vert-(a_{k+1}-1)\Vert q_k\alpha\Vert=\Vert q_{k+1}\alpha\Vert+\Vert q_k\alpha\Vert,
\end{equation}
in view of \eqref{eq4.38}.

It follows immediately from \eqref{eq4.35} that one of the neighbors of $0$ in the partition $\AAA_k(\alpha)$ is $\{q_k\alpha\}$ which clearly has distance
$\Vert q_k\alpha\Vert$ from $0$ in the unit torus/circle.
Since $\alpha$ is irrational, the other neighbor of $0$ in the partition $\AAA_k(\alpha)$ must have distance
$\Vert q_{k+1}\alpha\Vert+\Vert q_k\alpha\Vert$ from $0$ in the unit torus/circle.
Simple calculation then shows that it is $\{((a_{k+1}-1)q_k+q_{k-1})\alpha\}$.
Thus the two neighbors
\begin{displaymath}
\{q_k\alpha\}
\quad\mbox{and}\quad
\{((a_{k+1}-1)q_k+q_{k-1})\alpha\}
\end{displaymath}
of $0$ in the partition $\AAA_k(\alpha)$ exhibit the two gaps in \eqref{eq4.39} in some order.
Similarly, the two neighbors
\begin{displaymath}
\{(q_k-1)\alpha\}
\quad\mbox{and}\quad
\{((a_{k+1}-1)q_k+q_{k-1}-1)\alpha\}
\end{displaymath}
of $1-\alpha=\{-\alpha\}$ in the partition $\AAA_k(\alpha)$ exhibit the same two gaps in \eqref{eq4.39} in the same order.
Furthermore, for every integer $q=1,\ldots,m-1$, the two neighbors
\begin{displaymath}
\{(q_k+q)\alpha\}
\quad\mbox{and}\quad
\{((a_{k+1}-1)q_k+q_{k-1}+q)\alpha\}
\end{displaymath}
of $\{q\alpha\}$ in the partition $\AAA_k(\alpha)$ also exhibit the same two gaps in \eqref{eq4.39} in the same order.

The union of the left and right neighborhoods of $0$ in the partition $\AAA_k(\alpha)$ has the form
\begin{equation}\label{eq4.40}
B(0)=(-d^{\ast},d^{\ast\ast}).
\end{equation}
Indeed, the union of the left and right neighborhoods of $\{q\alpha\}$, $q=-1,0,1,\ldots,m-1$, in the partition $\AAA_k(\alpha)$ has the form
\begin{equation}\label{eq4.41}
B(q)=(\{q\alpha\}-d^{\ast},\{q\alpha\}+d^{\ast\ast}),
\end{equation}
with the two gaps in the same order, where
\begin{equation}\label{eq4.42}
\{d^{\ast},d^{\ast\ast}\}=\{\Vert q_k\alpha\Vert,\Vert q_{k+1}\alpha\Vert+\Vert q_k\alpha\Vert\},
\end{equation}
but we have not specified which one is which.
We refer to $B(q)$, $q=-1,0,m-1$, as the buffer zones of the singularities $1-\alpha,0,\{(m-1)\alpha\}$ respectively of~$T$.

Now suppose that $q_{k+1}$ is much greater than~$m$.

We consider the short special intervals
\begin{equation}\label{eq4.43}
J_k(q)=J(\alpha;k;q)=(\{q\alpha\}-d^{\ast\ast},\{q\alpha\}+d^{\ast}),
\quad
m\le q\le q_{k+1}-2.
\end{equation}
Note that these short special intervals have three crucial properties: 

(i) They completely cover the $m+1$ long special intervals determined by the $m+1$ division points $\{q\alpha\}$, $q=-1,0,1,\ldots,m-1$, of the
torus/circle $[0,1)$.

(ii) They avoid all the division points $\{q\alpha\}$, $q=-1,0,1,\ldots,m-1$, in view of \eqref{eq4.40}--\eqref{eq4.43}.
In particular, they avoid the singularities $1-\alpha,0,\{(m-1)\alpha\}$ of~$T$.

(iii) Any two short special intervals contained inside the same long special interval in (i) and arising from neighboring partition points exhibit \textit{substantial overlapping}.
More precisely, if $q'\ne q''$ are two integers such that $m\le q',q''\le q_{k+1}-2$ and $\{q'\alpha\}$ and $\{q''\alpha\}$ are neighboring points in the partition $\AAA_k(\alpha)$, and both points are in the same long special interval in (i), then
\begin{equation}\label{eq4.44}
\length(J_k(q')\cap J_k(q''))\ge\min\{d^{\ast},d^{\ast\ast}\}=\Vert q_k\alpha\Vert.
\end{equation}
The trivial upper bound
\begin{equation}\label{eq4.45}
\length(J_k(q))=2\Vert q_k\alpha\Vert+\Vert q_{k+1}\alpha\Vert<3\Vert q_k\alpha\Vert
\end{equation}
and \eqref{eq4.44} together justify the term \textit{substantial overlapping}.

Since $T$ acts on the interval $[0,n)$, for every interval $J_k(q)$, $m\le q\le q_{k+1}-2$, given by \eqref{eq4.43}, we define its \textit{$n$-copy extension} $J_k(q;n)$ by
\begin{displaymath}
J_k(q;n)=J_k(q)\cup(1+J_k(q))\cup\ldots\cup(n-1+J_k(q))\subset[0,n),
\end{displaymath}
a union of $J_k(q)$ with $n-1$ of its translates.

\begin{lem}\label{lem43}
Let $\eps<1/100$ be positive and fixed.
Provided that the positive integer $k$ is sufficiently large, there exists an integer $q^*$ such that $m\le q^*\le q_{k+1}-2$ and for each
$\ell=0,1,\ldots,n-1$, we have either
\begin{equation}\label{eq4.46}
\meas((\ell+J_k(q^*))\cap S_0)>(1-\eps)\meas(J_k(q^*)),
\end{equation}
or
\begin{equation}\label{eq4.47}
\meas((\ell+J_k(q^*))\cap S_0)<\eps\meas(J_k(q^*)).
\end{equation}
\end{lem}

\begin{remark}
Lemma~\ref{lem43} resembles Lebesgue's Density Theorem.
It is rather tempting to say that the latter almost implies the former, or at least makes the former quite plausible.
Nevertheless, our formal proof below does not make use of Lebesgue's Density Theorem, just the definition of Lebesgue measure.
\end{remark}

\begin{proof}[Proof of Lemma~\ref{lem43}]
Since $S_0$ is Lebesgue measurable, given any $\eta>0$, there exists a finite set of disjoint intervals $I_h$, $1\le h\le H=H(S_0;\eta)$, such that the union
\begin{equation}\label{eq4.48}
V=\bigcup_{1\le h\le H}I_h
\end{equation}
gives an $\eta$-approximation of~$S_0$, in the sense that the symmetric difference $V\triangle S_0$ satisfies the condition
\begin{equation}\label{eq4.49}
\meas(V\triangle S_0)=\meas(V\setminus S_0)+\meas(S_0\setminus V)<\eta.
\end{equation}
We will specify a suitable value of $\eta=\eta(\eps)>0$ later.

A short special interval $\ell+J_k(q)$, where $\ell=0,1,\ldots,n-1$ and $m\le q\le q_{k+1}-2$, is said to be \textit{$V$-nice} if it is either completely contained in~$V$, or it is disjoint from~$V$.

Since $V$ given by \eqref{eq4.48} is a finite union of disjoint intervals, it is clear that there exists an integer-valued threshold $k=k(S_0;V;\eta)$ such that the union of the $V$-nice short special intervals $\ell+J_k(q)$, with $\ell=0,1,\ldots,n-1$ and $m\le q\le q_{k+1}-2$, has measure at least
$n(1-\eta)$.

On the other hand, let $\BBB$ denote the set of short special intervals $\ell+J_k(q)$, where $\ell=0,1,\ldots,n-1$ and $m\le q\le q_{k+1}-2$, that are \textit{bad} in the sense that
\begin{equation}\label{eq4.50}
\frac{\meas((V\triangle S_0)\cap(\ell+J_k(q)))}{\Vert q_k\alpha\Vert}\ge\eps.
\end{equation}
Then it follows from \eqref{eq4.49} and \eqref{eq4.50} that
\begin{displaymath}
\eta
>\frac{1}{3}\sum_{\substack{{\ell=0,1,\ldots,n-1}\\{m\le q\le q_{k+1}-2}\\{\ell+J_k(q)\in\BBB}}}\meas((V\triangle S_0)\cap(\ell+J_k(q)))
\ge\frac{1}{3}\eps\vert\BBB\vert\Vert q_k\alpha\Vert,
\end{displaymath}
where the factor $1/3$ arises from the observation that an interval $J_k(q)$ intersects at most two other such intervals, namely its left and right neighbors, and $\vert\BBB\vert$ denotes the cardinality of the set~$\BBB$.
Combining this with \eqref{eq4.36}, we deduce that
\begin{equation}\label{eq4.51}
\vert\BBB\vert\le\frac{3\eta}{\eps\Vert q_k\alpha\Vert}<\frac{6\eta q_{k+1}}{\eps}.
\end{equation}
Since $\eta>0$ can be arbitrarily small, we choose $\eta=\eps^2/6$.
Then \eqref{eq4.51} simplifies to $\vert\BBB\vert<\eps q_{k+1}$.
Since $\eps$ is small, the bad short special intervals in $\BBB$ form a small minority of the short special intervals under consideration.

Thus the overwhelming majority of the short special intervals under consideration are $V$-nice and violate \eqref{eq4.50}.
A routine application of the Pigeonhole Principle now implies the existence of an integer $q^*$ such that $m\le q^*\le q_{k+1}-2$ and each interval
$\ell+J_k(q^*)$, $\ell=0,1,\ldots,n-1$, is $V$-nice and violates \eqref{eq4.50}.
For such an interval $\ell+J_k(q^*)$, it follows from \eqref{eq4.42} and \eqref{eq4.43} that
\begin{displaymath}
\meas((V\triangle S_0)\cap(\ell+J_k(q^*)))<\eps\meas(J_k(q^*)).
\end{displaymath}
Since $V\triangle S_0=(V\setminus S_0)\cup(S_0\setminus V)$ is a disjoint union, it follows that
\begin{equation}\label{eq4.52}
\meas((V\setminus S_0)\cap(\ell+J_k(q^*)))+\meas((S_0\setminus V)\cap(\ell+J_k(q^*)))<\eps\meas(J_k(q^*)).
\end{equation}

Suppose first of all that $\ell+J_k(q^*)$ is completely contained in~$V$.
Then
\begin{align}\label{eq4.53}
&
\meas((V\setminus S_0)\cap(\ell+J_k(q^*)))
=\meas((\ell+J_k(q^*))\setminus S_0)
\nonumber
\\
&\quad
=\meas(\ell+J_k(q^*))-\meas((\ell+J_k(q^*))\cap S_0),
\end{align}
while
\begin{equation}\label{eq4.54}
\meas((S_0\setminus V)\cap(\ell+J_k(q^*)))=\meas(\emptyset)=0.
\end{equation}
The assertion \eqref{eq4.46} now follows on combining \eqref{eq4.52}--\eqref{eq4.54}.

Suppose next that $\ell+J_k(q^*)$ is disjoint from~$V$.
Then
\begin{equation}\label{eq4.55}
\meas((V\setminus S_0)\cap(\ell+J_k(q^*)))=\meas(\emptyset)=0,
\end{equation}
while
\begin{equation}\label{eq4.56}
\meas((S_0\setminus V)\cap(\ell+J_k(q^*)))=\meas(S_0\cap(\ell+J_k(q^*))).
\end{equation}
The assertion \eqref{eq4.47} now follows on combining \eqref{eq4.52} and \eqref{eq4.55}--\eqref{eq4.56}.
\end{proof}

In view of Lemma~\ref{lem43}, we can define an ordered $n$-tuple
\begin{displaymath}
\Theta(k,q^*)=(\theta_0(k,q^*),\theta_1(k,q^*),\ldots,\theta_{n-1}(k,q^*)),
\end{displaymath}
where, for $\ell=0,1,\ldots,n-1$,
\begin{displaymath}
\theta_\ell(k,q^*)=\left\{\begin{array}{ll}
1,&\mbox{if $\ell+J_k(q^*)$ satisfies \eqref{eq4.46}},\\
0,&\mbox{if $\ell+J_k(q^*)$ satisfies \eqref{eq4.47}}.
\end{array}\right.
\end{displaymath}

We are now in a position to complete the proof of Lemma~\ref{lem41}.

The first key step in our argument is the extension of the \textit{local set} $J_k(q^*;n)$ globally via a $T$-power argument.

Consider an arbitrary set $J_k(q;n)$ such that $m\le q\le q_{k+1}-2$ and $q\ne q^*$.
Then
\begin{displaymath}
J_k(q;n)=T^{q-q^*}J_k(q^*;n).
\end{displaymath}
Note that $S_0\subset [0,n)$ is $T$-invariant, and that the three singularities
\begin{displaymath}
1-\alpha=\{-\alpha\},
\quad
0,
\quad
\{(m-1)\alpha\}
\end{displaymath}
modulo~$1$ never split the intervals in the process of iterated applications of the transformation~$T$.
It follows that $J_k(q;n)\cap S_0$ defines an ordered $k$-tuple $\Theta(k,q)$ which is either equal to $\Theta(k,q^*)$ or has the entries permuted.

The second key step in our argument concerns taking advantage of property (iii) earlier concerning substantial overlappings of the intervals $J_k(q)$.

Recall that the division points
\begin{displaymath}
\{q\alpha\},
\quad
q=-1,0,1,\ldots,m-1,
\end{displaymath}
of the torus/circle $[0,1)$ give rise to $m+1$ long special intervals in the torus/circle $[0,1)$.
They lead naturally to $n(m+1)$ division points and $n(m+1)$ long special intervals in $[0,n)$.
Now the sets $J_k(q;n)$, $m\le q\le q_{k+1}-2$, lead to $n(q_{k+1}-m-1)$ intervals which give rise to $n(m+1)$ collections of substantially overlapping intervals in $[0,n)$.
These $n(m+1)$ collections cover the $n(m+1)$ disjoint long special intervals.
Due to the substantial overlappings, neighboring short special intervals in the same collection must have identical ordered $n$-tuples $\Theta(k,q)$.

It follows that the short special intervals within any given long special interval $\III\subset[0,n)$ must either all satisfy \eqref{eq4.46} or all satisfy \eqref{eq4.47}.
This means that the given long special interval $\III$ is $\eps$-almost entirely in~$S_0$, in the sense that
\begin{equation}\label{eq4.57}
\meas(\III\cap S_0)>(1-\eps)\meas(\III),
\end{equation}
or is $\eps$-almost disjoint from~$S_0$, in the sense that
\begin{displaymath}
\meas(\III\cap S_0)<\eps\meas(\III).
\end{displaymath}

Let the set $S_0^*\subset[0,n)$ be defined as follows, apart from the $n(m+1)$ division points that give rise to the long special intervals.
For every long special interval $\III\subset[0,n)$, we set
\begin{displaymath}
\III\subset S_0^*
\quad\mbox{if and only if}\quad
\mbox{$\III$ satisfies \eqref{eq4.57}}.
\end{displaymath}
Then each of the long special intervals in $[0,n)$ is either entirely contained in $S_0^*$ or disjoint from~$S_0^*$.
It then remains to prove that if the GCD Criterion fails, then such a set $S_0^*$ cannot exist.
Our argument is to show that the existence of such a set $S_0^*$ would give rise to a multi-coloring of the $n$-square-$b$ surface, sufficiently restricted as to allow us to derive the necessary contradiction.

The $n(m+1)$ division points in $[0,n)$ that give rise to the $n(m+1)$ long special intervals are
\begin{displaymath}
\ell+\{-\alpha\},\ell,\ell+\{\alpha\},\ldots,\ell+\{(m-1)\alpha\},
\quad
\ell=0,1,\ldots,n-1.
\end{displaymath}
Here, for each $\ell=0,1,\ldots,n-1$, we view the left vertical edge of the $(\ell+1)$-th square face of the $n$-square-$b$ surface as the interval
$[\ell,\ell+1)$.
We can then $2$-color these $n$ intervals to distinguish points in $S_0^*$ from points not in~$S_0^*$.
This gives rise to a $2$-coloring of the set $[0,n)$.

For ease of description, let us denote the bottom left vertex of the $1$-st square face and the top right vertex of the $n$-th square face of the
$n$-square-$b$ surface by $(0,0)$ and $(n,1)$ respectively.
Using the $\alpha$-flow, the $n(m+1)$ division points now lead to $n(m+1)$ line segments, linking pairs of points
\begin{equation}\label{eq4.58}
\begin{array}{ccc}
(\ell,\{-\alpha\})
&
\mbox{and}
&
(\ell+1,1),
\\
(\ell,0)
&
\mbox{and}
&
(\ell+1,\{\alpha\}),
\\
(\ell,\{\alpha\})
&
\mbox{and}
&
(\ell+1,\{2\alpha\}),
\\
&\vdots
\\
(\ell,\{(m-1)\alpha\})
&
\mbox{and}
&
(\ell+1,\{m\alpha\}),
\end{array}
\quad
\ell=0,1,\ldots,n-1,
\end{equation}
lying on the $(\ell+1)$-th square face of the $n$-square-$b$ surface.

We shall show later that the line segments linking the points
\begin{equation}\label{eq4.59}
(\ell,\{-\alpha\})
\quad\mbox{and}\quad
(\ell+1,1),
\quad
\ell=0,1,\ldots,n-1,
\end{equation}
do not come into the argument.

Note first of all that $\{m\alpha\}$ is not a division point of the vertical edges of the $n$-square-$b$ surface.
As in Section~\ref{sec2}, write $b=b_\nu=\{m\alpha\}$.
Using \eqref{eq2.3}, we see that $\{m\alpha\}$ is in the interior of the interval $[b_{\nu-1},b_{\nu+1})$, of the form \eqref{eq2.6}.

Suppose that $\meas(S_0^*)=\meas(S_0)=\tau$, so that precisely $\tau$ of the intervals
\begin{equation}\label{eq4.60}
\ell+[b_{\nu-1},b_{\nu+1}),
\quad
\ell=0,1,\ldots,n-1,
\end{equation}
belong to $S_0^*$.
Assign the color $B$ or $W$ to an interval in \eqref{eq4.60} according to whether it is contained in $S_0^*$ or is disjoint from $S_0^*$.
This gives rise to a $2$-coloring sequence of $n$ terms, corresponding to the $n$ intervals \eqref{eq4.60} on the left vertical edges of the square faces of the $n$-square-$b$ surface and made up of $\tau$ copies of $B$ and $(n-\tau)$ copies of~$W$.
We determine a shortest subsequence of consecutive terms of this $2$-coloring sequence of length $d>1$ such that the $2$-coloring sequence modulo $n$ is the $d$-term subsequence repeated $n/d$ times.
In particular, the number $d>1$ must divide~$n$.
In view of cyclic periodicity, we may restrict our attention to the $d$ square faces corresponding to this $d$-term subsequence.

For these $d$ square faces under consideration, we color the interval $[b_{\nu-1},b_{\nu+1})$ on the left vertical edges according to the $2$-coloring subsequence of length~$d$, and then use the $\alpha$-flow to spread this $2$-coloring of the intervals to the relevant square faces.
Figure~4.2 below, which is not to scale, illustrates our observations thus far in the case $d=4$, where the $2$-coloring subsequence $W,B,B,B$ has length~$4$.

\begin{displaymath}
\begin{array}{c}
\includegraphics{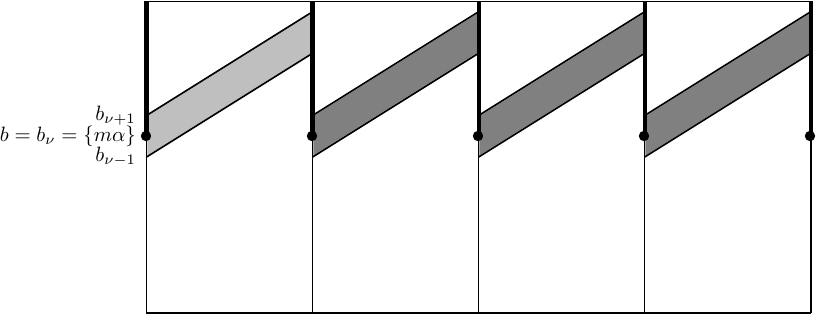}
\\
\mbox{Figure 4.2: a partial $2$-coloring on $4$ consecutive square faces}
\\
\mbox{of the $n$-square-$b$ surface where $4$ divides $n$}
\end{array}
\end{displaymath}

We next consider the line segments linking the pairs of points
\begin{equation}\label{eq4.61}
(\ell,\{(m-1)\alpha\})
\quad\mbox{and}\quad
(\ell+1,\{m\alpha\})
\end{equation}
on the square faces under consideration.
Since $\{(m-1)\alpha\}$ is one of the division points, it follows from \eqref{eq2.3} that there exists a unique $\mu=0,1,\ldots,m$ such that $\mu\ne\nu$ and $\{(m-1)\alpha\}=b_\mu$.
Let $b_{\mu-1}$ and $b_{\mu+1}$ be the closest division points to $b_\mu$ from below and above respectively.
We next investigate the coloring of the intervals $[b_{\mu-1},b_\mu)$ and $[b_\mu,b_{\mu+1})$ on the left vertical edges of the square faces.
The $T$-invariance of~$S_0$, and hence~$S_0^*$, clearly dictates that the interval $[b_\mu,b_{\mu+1})$ must have the same coloring as the interval
$[b_{\nu-1},b_{\nu+1})$ on the left vertical edge of the same square face, whereas the interval $[b_{\mu-1},b_\mu)$ must have the same coloring as the interval $[b_{\nu-1},b_{\nu+1})$ on the left vertical edge of the square face immediately to the right.
Figure~4.3 continues with our example, and we see that the $2$-coloring sequence of the intervals $[b_\mu,b_{\mu+1})$ in the $4$ square faces remain $W,B,B,B$, whereas the $2$-coloring sequence of the intervals $[b_{\mu-1},b_\mu)$ in the $4$ square faces becomes $B,B,B,W$, representing a shift by $1$ to the left of the original $2$-coloring sequence.
Furthermore, the color pattern sequence across the line segments \eqref{eq4.61} is
\begin{equation}\label{eq4.62}
WB,BB,BB,BW,
\end{equation}
where, for instance, $WB$ denotes $W$ above and $B$ below.

\begin{displaymath}
\begin{array}{c}
\includegraphics{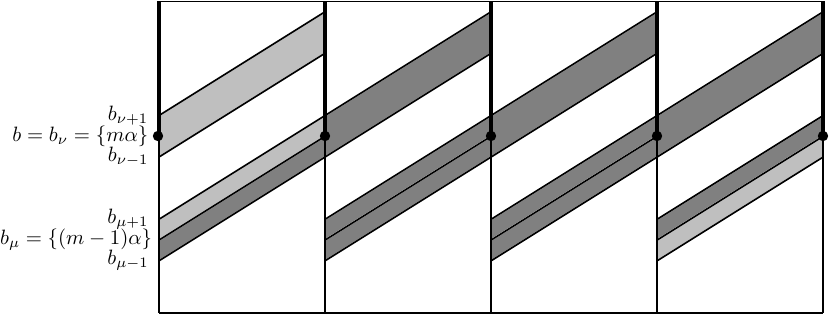}
\\
\mbox{Figure 4.3: extending the $2$-coloring on $4$ consecutive square faces}
\\
\mbox{of the $n$-square-$b$ surface where $4$ divides $n$}
\end{array}
\end{displaymath}

Let $b_{\mu-2}$ denote the closest division point to $b_{\mu-1}$ from below, and let $b_{\nu-2}$ denote the closest division point to $b_{\nu-1}$ from below.
Note that $b_{\nu-2}=\{b_{\mu-2}+\alpha\}$.
We next consider the line segments linking the pairs of points $(\ell,b_{\mu-2})$ and $(\ell+1,b_{\nu-2})$ on the square faces under consideration.
As these line segments can be obtained from those in \eqref{eq4.61} under the inverse transformation~$T^{-1}$, the color patterns across them must be preserved.
Cyclic periodicity and the maximality of $d$ then dictate that they must appear in the same order modulo~$d$.
The unique solution to this problem is a shift by $1$ to the left of the earlier color pattern sequence across the line segments.
Figure~4.4 illustrates this observation in our continuing example.

\begin{displaymath}
\begin{array}{c}
\includegraphics{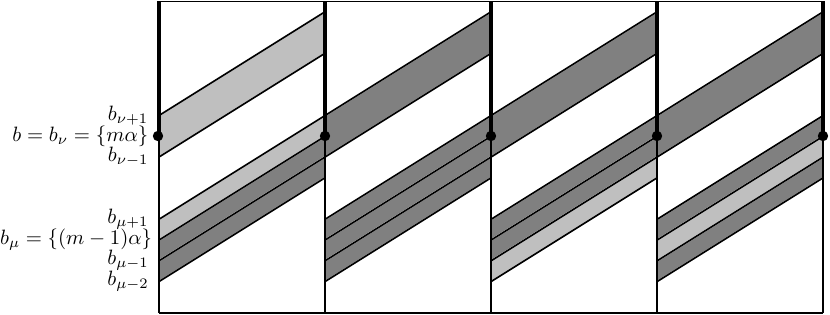}
\\
\mbox{Figure 4.4: cyclic periodicity at work on $4$ consecutive square faces}
\\
\mbox{of the $n$-square-$b$ surface where $4$ divides $n$}
\end{array}
\end{displaymath}

Note that the new color pattern sequence across the line segments is
\begin{equation}\label{eq4.63}
BB,BB,BW,WB,
\end{equation}
a shift to the left by $1$ of the sequence \eqref{eq4.62}.

Let $b_{\mu-3}$ denote the closest division point to $b_{\mu-2}$ from below, and let $b_{\nu-3}$ denote the closest division point to $b_{\nu-2}$ from below.
Note that $b_{\nu-3}=\{b_{\mu-3}+\alpha\}$.
We next consider the line segments linking the pairs of points $(\ell,b_{\mu-3})$ and $(\ell+1,b_{\nu-3})$ on the square faces under consideration.
Again, the color patterns across them must be preserved.
Cyclic periodicity and the maximality of $d$ then dictate that they must again appear in the same order modulo~$d$.
The unique solution to this problem is a shift by $1$ to the left of the earlier color pattern sequence across the line segments.
Figure~4.5 illustrates this observation in our continuing example.

\begin{displaymath}
\begin{array}{c}
\includegraphics{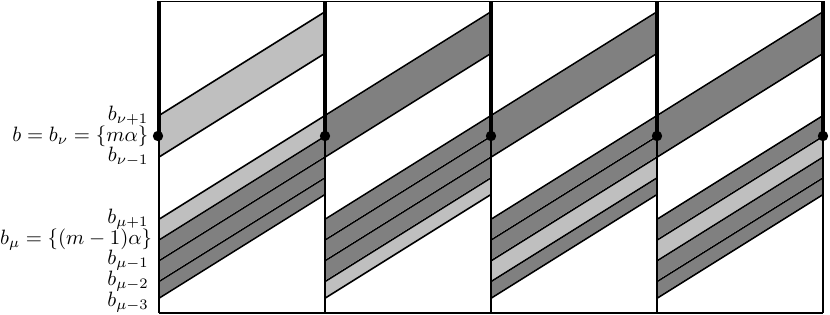}
\\
\mbox{Figure 4.5: cyclic periodicity at work on $4$ consecutive square faces}
\\
\mbox{of the $n$-square-$b$ surface where $4$ divides $n$}
\end{array}
\end{displaymath}

Note that the new color pattern sequence across the line segments is
\begin{displaymath}
BB,BW,WB,BB
\end{displaymath}
a shift to the left by $1$ of the sequence \eqref{eq4.63}.

The reader will by now observe that this $2$-coloring can be replaced by a $d$-coloring, and all the properties we have described so far will be preserved.
For instance, for our continuing example, Figure~4.5 can be replaced by Figure~4.6.

\begin{displaymath}
\begin{array}{c}
\includegraphics{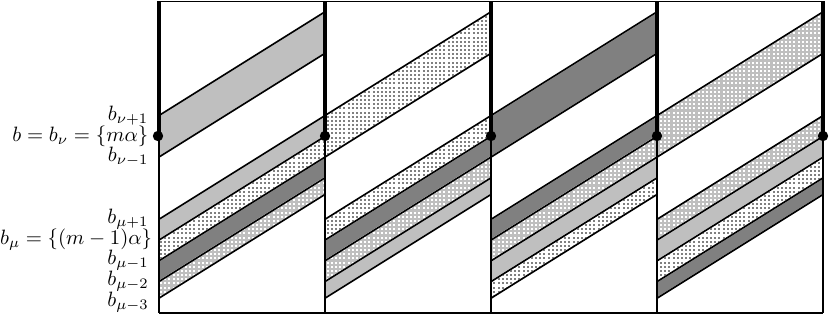}
\\
\mbox{Figure 4.6: a partial $4$-coloring on $4$ consecutive square faces}
\\
\mbox{of the $n$-square-$b$ surface where $4$ divides $n$}
\end{array}
\end{displaymath}

Proceeding in the same way will allow us eventually to $d$-color the entire $d$ square faces under consideration.

We have claimed earlier that the line segments linking the points \eqref{eq4.59} do not come into the argument.
It can easily be shown that the color pattern across each of these line segments is monochromatic.
Thus there are only $m$ division points on the left vertical edge of each square face.
The cyclic periodicity described earlier now shows that we have a double periodic $d$-coloring pattern on the $d$ square faces analogous to \eqref{eq2.7} and which can be obtained by the Double Periodic Coloring Algorithm.
It follows that $d$ must divide~$m$.
Furthermore, each color represents a $T$-invariant subset of the $n$-square-$b$ surface.
It now follows from Lemma~\ref{lem21} that $d$ also divides $\Upsilon(m;\alpha)$.
Since $d\ge2$, it follows that the GCD Criterion is satisfied, a contradiction.
Hence such a non-trivial $T$-invariant subset $S_0$ of $[0,n)$ does not exist, and it follows that $T$ is ergodic.
\end{proof}

Lemma~\ref{lem41} tells us that if the GCD Criterion fails, then $T$ is ergodic.
Birkhoff's ergodic theorem then gives equidistribution of the half-infinite $\alpha$-geodesic on the $n$-square-$b$ surface for \textit{almost every} starting point.
This time-qualitative result arises from our $2$-distance method.
Later in Section~\ref{sec6}, we shall extend this result to half-infinite geodesics with \textit{any} starting point.

%
%

\section{Starting the proof of Theorem~\ref{thm32}: proving ergodicity}\label{sec5}

Let $\PPP$ be an arbitrary flat finite polysquare-$b$-rational translation surface with $s$ atomic squares and with division numbers $\{r_ib\}$, $i=1,\ldots,R$, where each $r_i$ is rational and $b$ is irrational.
Let
\begin{displaymath}
T=T_\alpha(\PPP;\{r_ib\},i=1,\ldots,R)
\end{displaymath}
denote the interval exchange transformation of the $\alpha$-flow on this surface.
Then $T$ maps the interval $[0,s)$ to itself.

Since reflecting a polysquare-$b$-rational translation surface across a horizontal or vertical line gives rise to another polysquare-$b$-rational translation surface, we can assume, without loss of generality, that the slope satisfies $0<\alpha<\infty$.

To bring us one step closer to a complete proof of Theorem~\ref{thm32}, we have the following analog of Lemma~\ref{lem41} on ergodicity.

\begin{lem}\label{lem51}
Consider $\alpha$-flow on an arbitrary polysquare-$b$-rational translation surface $\PPP$ with $s$ atomic squares and division numbers $\{r_ib\}$, $i=1,\ldots,R$, where each $r_i$ is rational and $b$ is irrational.
Assume that
\begin{equation}\label{eq5.1}
\{r_ib\}\ne\{m\alpha\},
\quad
i=1,\ldots,R,
\quad
m\in\Zz\setminus\{0\}.
\end{equation}
Then the interval exchange transformation $T:[0,s)\to[0,s)$ is ergodic.
\end{lem}

\begin{proof}
We shall prove this by contradiction.
Assume on the contrary that $T$ is not ergodic.
Then there exists a $T$-invariant measurable subset $S_0\subset[0,s)$ such that $0<\meas(S_0)<s$.
Since $T$ reduces modulo~$1$ to irrational rotation on the unit interval with the same~$\alpha$, it follows that $T$ modulo~$1$ is ergodic, and so $S_0$ modulo~$1$ is the unit interval $[0,1)$, implying that $\meas(S_0)\in\{1,2,\ldots,s-1\}$.

The irrational slope $\alpha\in(0,\infty)$ has an infinite continued fraction expansion
\begin{equation}\label{eq5.2}
\alpha=[a_0;a_1,a_2,a_3,\ldots]=a_0+\frac{1}{a_1+\frac{1}{a_2+\frac{1}{a_3+\cdots}}},
\end{equation}
where $a_0\ge0$, $a_i\ge1$, $i=1,2,3,\ldots,$ are integers.
The rational numbers
\begin{displaymath}
\frac{p_k}{q_k}=\frac{p_k(\alpha)}{q_k(\alpha)}=[a_0;a_1,\ldots,a_k],
\quad
k=0,1,2,3,\ldots,
\end{displaymath}
where $p_k\in\Zz$ and $q_k\in\Nn$ are coprime, are the $k$-convergents of~$\alpha$.
Since $\alpha$ is badly approximable, there exists an integer $A$ such that
\begin{equation}\label{eq5.3}
a_0,a_1,a_2,a_3,\ldots\le A.
\end{equation}

As in the proof of Lemma~\ref{lem41}, we again use the $3$-distance theorem as stated in Lemma~\ref{lem42}.
We also work with the same partition $\AAA_k(\alpha)$ of the unit torus/circle $[0,1)$ with $q_{k+1}=q_{k+1}(\alpha)$ division points $\{q\alpha\}$,
$-1\le q\le q_{k+1}-2$.
Note that the choices $q=-1,0$ in $\{q\alpha\}$ represent two of the singularities of the interval exchange transformation $T$ restricted to the interval $[0,1)$.
Here $k\ge1$ is an integer chosen to be sufficiently large.

Recall that for the special choice $n=q_{k+1}-1$, the $3$-distance theorem simplifies to a $2$-distance theorem, and that there are only two distances
\begin{displaymath}
\Vert q_k\alpha\Vert
\quad\mbox{and}\quad
\Vert q_{k+1}\alpha\Vert+\Vert q_k\alpha\Vert
\end{displaymath}
between adjacent partition points in~$\AAA_k(\alpha)$.
Thus the union of the left and right neighborhoods of $0$ in the partition $\AAA_k(\alpha)$ has the form
\begin{equation}\label{eq5.4}
B(0)=(-d^{\ast},d^{\ast\ast}),
\end{equation}
while the union of the left and right neighborhoods of $\{-\alpha\}$ in the partition $\AAA_k(\alpha)$ has the form
\begin{equation}\label{eq5.5}
B(-1)=(\{-\alpha\}-d^{\ast},\{-\alpha\}+d^{\ast\ast}),
\end{equation}
with the two gaps in the same order, where
\begin{equation}\label{eq5.6}
\{d^{\ast},d^{\ast\ast}\}=\{\Vert q_k\alpha\Vert,\Vert q_{k+1}\alpha\Vert+\Vert q_k\alpha\Vert\},
\end{equation}
but we have not specified which one is which.

We consider the short special intervals
\begin{equation}\label{eq5.7}
J_k(q)=J(\alpha;k;q)=[\{q\alpha\}-d^{\ast\ast},\{q\alpha\}+d^{\ast}),
\quad
1\le q\le q_{k+1}-2.
\end{equation}
Note that these short special intervals have three crucial properties: 

(i) They completely cover the two long special intervals determined by the two division points $0$ and $\{-\alpha\}$ of the torus/circle $[0,1)$.

(ii) They avoid the singularities $0$ and $\{-\alpha\}$ of~$T$, in view of \eqref{eq5.4}--\eqref{eq5.7}.

(iii) Any two short special intervals inside the same long special interval in (i) arising from neighboring partition points exhibit \textit{substantial overlapping}.

Recall from \eqref{eq4.31} that the singularities of $T$ modulo~$1$ are $0$ and $\{-\alpha\}$, together with $\{r_ib-\alpha\}$, $i=1,\ldots,R$.
These latter singularities require extra care, and we deviate from the proof of Lemma~\ref{lem41}.
Our new argument depends on a crucial but rather complicated technical lemma.
To formulate this, we first need some notation and definitions.

For each $i=1,\ldots,R$, we write the rational number $r_i$ in the form
\begin{equation}\label{eq5.8}
r_i=\frac{u_i}{v_i},
\quad
\mbox{with coprime $u_i\in\Zz$ and $v_i\in\Nn$}.
\end{equation}
Furthermore, we write
\begin{equation}\label{eq5.9}
U=\max_{1\le i\le R}\vert u_i\vert
\quad\mbox{and}\quad
V=\max_{1\le i\le R}\vert v_i\vert.
\end{equation}
On the other hand, let the integer $k\ge1$ be given.
For every $i=1,\ldots,R$, let
\begin{displaymath}
h_k(i;+)=h_k(\alpha;r_ib;+)
\quad\mbox{and}\quad
h_k(i;-)=h_k(\alpha;r_ib;-)
\end{displaymath}
denote the two integers satisfying $-1\le h_k(i;+),h_k(i;-)\le q_{k+1}-2$ such that
\begin{displaymath}
\{h_k(i;+)\alpha\}
\quad\mbox{and}\quad
\{h_k(i;-)\alpha\}
\end{displaymath}
are the two neighbors of the singularity $\{r_ib-\alpha\}$ in the partition $\AAA_k(\alpha)$ of the unit torus/circle $[0,1)$.
Clearly, for $\sigma=\pm$, we have
\begin{align}\label{eq5.10}
&
\Vert\{h_k(i;\sigma)\alpha\}-\{r_ib-\alpha\}\Vert
<\Vert\{h_k(i;+)\alpha\}-\{h_k(i;-)\alpha\}\Vert
\nonumber
\\
&\quad
\le\Vert q_{k+1}\alpha\Vert+\Vert q_k\alpha\Vert
\le\frac{1}{q_{k+2}}+\frac{1}{q_{k+1}}
<\frac{2}{q_{k+1}}.
\end{align}
It then follows from \eqref{eq5.10} that for every $i=1,\ldots,R$ and $\sigma =\pm$,
\begin{displaymath}
h_k(i;\sigma)=h_k(\alpha;r_ib;\sigma)\to\infty
\quad\mbox{as $k\to\infty$},
\end{displaymath}
for otherwise there exists an integer value $m$ such that $h_k(i;\sigma)=m$ for infinitely many distinct values of~$k$, and so the corresponding limit in \eqref{eq5.10} must have the value $\Vert m\alpha-\{r_ib-\alpha\}\Vert=0$, which contradicts the hypothesis \eqref{eq5.1}.

We have the following separation lemma.

\begin{lem}\label{lem52}
Let $\alpha\in(0,\infty)$ be badly approximable, with continued fraction \eqref{eq5.2} and digits satisfying \eqref{eq5.3}, where $A$ is a fixed positive integer.
Write
\begin{equation}\label{eq5.11}
\delta=\frac{1}{100(A+2)^2U^4V^5},
\end{equation}
where $U$ and $V$ are given by \eqref{eq5.8} and \eqref{eq5.9}.
Then there exists an infinite set
\begin{displaymath}
\KKK_0=\KKK_0(\alpha;r_ib,i=1,\ldots,R)
\end{displaymath}
of positive integers such that for every $k\in\KKK_0$, the following hold:

\emph{(i)}
For every $i=1,\ldots,R$ and $\sigma=\pm$, we have
\begin{equation}\label{eq5.12}
\delta q_{k+1}<h_k(i;\sigma)<\left( 1-\delta\right)q_{k+1}.
\end{equation}

\emph{(ii)}
For every $i_1,i_2=1,\ldots,R$ and $\sigma_1,\sigma_2=\pm$ such that $(i_1,\sigma_1)\ne(i_2,\sigma_2)$, we have
\begin{equation}\label{eq5.13}
\vert h_k(i_1;\sigma_1)-h_k(i_2;\sigma_2)\vert>\delta q_{k+1}.
\end{equation}
\end{lem}

The underlying idea of the proof of Lemma~\ref{lem52} is quite simple.
Unfortunately, the details are rather complicated and involve a case study.
We thus postpone the proof to Section~\ref{sec7}.

Since $T$ acts on the interval $[0,s)$, for every interval $J_k(q)$, $1\le q\le q_{k+1}-2$, given by \eqref{eq5.7}, we define its \textit{$s$-copy extension} $J_k(q;s)$ by
\begin{displaymath}
J_k(q;s)=J_k(q)\cup(1+J_k(q))\cup\ldots\cup((s-1)+J_k(q))\subset[0,s),
\end{displaymath}
a union of $J_k(q)$ with $s-1$ of its translates.

We have a more complicated variant of Lemma~\ref{lem43}.

\begin{lem}\label{lem53}
Let $\delta$ be given by \eqref{eq5.11}, and let
\begin{displaymath}
0<\eps<\frac{\delta}{100}.
\end{displaymath}
Provided that the positive integer $k$ is sufficiently large, for any subset
\begin{displaymath}
\WWW\subset\{1,2,3,\ldots,q_{k+1}-2\}
\end{displaymath}
such that the cardinality $\vert\WWW\vert\ge\delta q_{k+1}$, there exists an integer $q^*\in\WWW$ such that for each $\ell=0,1,\ldots,s-1$, we have either
\begin{equation}\label{eq5.14}
\meas((\ell+J_k(q^*))\cap S_0)>(1-\eps)\meas(J_k(q^*)),
\end{equation}
or
\begin{equation}\label{eq5.15}
\meas((\ell+J_k(q^*))\cap S_0)<\eps\meas(J_k(q^*)).
\end{equation}
\end{lem}

\begin{proof}
Since $S_0$ is Lebesgue measurable, given any $\eta>0$, there exists a finite set of disjoint intervals $I_h$, $1\le h\le H=H(S_0;\eta)$, such that the union
\begin{displaymath}
V=\bigcup_{1\le h\le H}I_h
\end{displaymath}
gives an $\eta$-approximation of~$S_0$.
Let $\BBB$ denote the set of \textit{bad} short special intervals $\ell+J_k(q)$, where $\ell=0,1,\ldots,s-1$ and $1\le q\le q_{k+1}-2$, in the sense that
\begin{displaymath}
\frac{\meas((V\triangle S_0)\cap(\ell+J_k(q)))}{\Vert q_k\alpha\Vert}\ge\eps.
\end{displaymath}
Mimicking the proof of Lemma~\ref{lem43} and choosing $\eta=\eps^2/6$, we deduce that
\begin{equation}\label{eq5.16}
\vert\BBB\vert\le\frac{3\eta}{\eps\Vert q_k\alpha\Vert}<\frac{6\eta q_{k+1}}{\eps}<\eps q_{k+1}<\frac{\delta q_{k+1}}{100}.
\end{equation}

Suppose on the contrary that the conclusion of Lemma~\ref{lem53} fails.
Again mimicking the proof of Lemma~\ref{lem43}, we can show that there exists a subset
\begin{displaymath}
\WWW_0\subset\{1,2,3,\ldots,q_{k+1}-2\},
\end{displaymath}
with cardinality $\vert\WWW_0\vert\ge\delta q_{k+1}$, such that for every integer $q\in\WWW_0$, there exists $\ell=\ell(q)$ satisfying $0\le\ell\le s-1$ such that
\begin{displaymath}
\frac{\meas((V\triangle S_0)\cap(\ell(q)+J_k(q)))}{\Vert q_k\alpha\Vert}\ge\eps.
\end{displaymath}
Thus $\vert\BBB\vert\ge\vert\WWW_0\vert\ge\delta q_{k+1}$, contradicting \eqref{eq5.16}.
This completes the proof.
\end{proof}

In view of Lemma~\ref{lem53}, we can define an ordered $s$-tuple
\begin{equation}\label{eq5.17}
\Theta(k,q^*)=(\theta_0(k,q^*),\theta_1(k,q^*),\ldots,\theta_{s-1}(k,q^*)),
\end{equation}
where, for $\ell=0,1,\ldots,s-1$,
\begin{equation}\label{eq5.18}
\theta_\ell(k,q^*)=\left\{\begin{array}{ll}
1,&\mbox{if $\ell+J_k(q^*)$ satisfies \eqref{eq5.14}},\\
0,&\mbox{if $\ell+J_k(q^*)$ satisfies \eqref{eq5.15}}.
\end{array}\right.
\end{equation}

The $R+2$ singularities $0$, $\{-\alpha\}$ and $\{r_ib-\alpha\}$, $i=1,\ldots,R$, of $T$ modulo~$1$ divide the unit torus/circle $[0,1)$ into $R+2$ disjoint long critical intervals.

The next lemma follows from combining Lemmas \ref{lem52} and~\ref{lem53}.

\begin{lem}\label{lem54}
If $k\in\KKK_0=\KKK_0(\alpha;r_ib,i=1,\ldots,R)$ is sufficiently large, then for every short special interval $J_k(q)$, $1\le q\le q_{k+1}-2$, that is fully contained inside one of the $R+2$ long critical intervals, the intersection $J_k(q;s)\cap S_0$ defines an ordered $s$-tuple $\Theta(k,q)$ that is either equal to $\Theta(k,q^*)$ or has the entries permuted.
\end{lem}

\begin{proof}
Suppose that $J_k(q)$ is fully contained inside one of the $R+2$ long critical intervals with endpoints $z_1<z_2$ which are two adjacent singularities of $T$ modulo~$1$.
Suppose first that $z_1,z_2\not\in\{0,\{-\alpha\}\}$, where $z_2=0$ denotes the endpoint $1=0$.
Then $z_1=\{r_{i_1}b\}$ and $z_2=\{r_{i_2}b\}$ for some $i_1\ne i_2$ satisfying $1\le i_1,i_2\le R$, and there exist $\sigma_1,\sigma_2\in\{\pm\}$ such that
\begin{displaymath}
z_1<\{h_k(i_1;\sigma_1)\alpha\}<\{h_k(i_2;\sigma_2)\alpha\}<z_2.
\end{displaymath}
It follows from \eqref{eq5.13} that the finite sequence of consecutive integers with integer endpoints $h_k(i_1;\sigma_1)$ and $h_k(i_2;\sigma_2)$ has at least $\delta q_{k+1}$ terms and contains the integer~$q$.
If $z_1$ or $z_2$ belongs to $\{0,\{-\alpha\}\}$, then a modification of the argument, using \eqref{eq5.12} as well as \eqref{eq5.13}, will also lead to a sequence of consecutive integers with at least $\delta q_{k+1}$ terms and which contains the integer~$q$.
It then follows from Lemma~\ref{lem53} that this finite sequence of consecutive integers also contains an integer $q^*$ such that an ordered $s$-tuple
$\Theta(k,q^*)$ of the form \eqref{eq5.17} exists and satisfies \eqref{eq5.18} for every $\ell=0,1,\ldots,s-1$.
Note next that
\begin{displaymath}
J_k(q;s)=T^{q-q^*}J_k(q^*;s).
\end{displaymath}
Note that $S_0\subset[0,s)$ is $T$-invariant, and the $R+2$ singularities never split the intervals in the process of iterated applications of the transformation~$T$.
The desired conclusion follows immediately.
\end{proof}

Our next step is to take advantage of property (iii) earlier concerning substantial overlappings of the intervals $J_k(q)$.

Recall that $R+2$ division points $0$, $\{-\alpha\}$ and $\{r_ib-\alpha\}$, $i=1,\ldots,R$, of the torus/circle $[0,1)$ give rise to $R+2$ long critical intervals in the torus/circle $[0,1)$.
They lead naturally to $s(R+2)$ division points and $s(R+2)$ long critical intervals in $[0,s)$.

Consider the $s(q_{k+1}-2)$ short special intervals
\begin{displaymath}
\ell+J_k(q),
\quad
\ell=0,1,\ldots,s-1,
\quad
1\le q\le q_{k+1}-2.
\end{displaymath}
For large values of~$k$, the overwhelming majority of these short special intervals are fully contained in some long critical interval in $[0,s)$, and give rise to $s(R+2)$ collections of substantially overlapping intervals in $[0,s)$.
These $s(R+2)$ collections essentially cover the $s(R+2)$ disjoint long critical intervals; more precisely, each long critical interval has an extremely short subinterval at each end which may not be covered.
Due to the substantial overlappings, neighboring short special intervals in the same collection must have identical ordered $s$-tuples $\Theta(k,q)$.

It follows that the short special intervals fully contained within any given long critical interval $\III\subset[0,s)$ must either all satisfy \eqref{eq5.14} or all satisfy \eqref{eq5.15}.
This means that the given long critical interval $\III$ is essentially $\eps$-almost entirely in~$S_0$, or is  essentially $\eps$-almost disjoint from~$S_0$. 

Let the set $S_0^*\subset[0,s)$ be defined as follows, apart from the $s(R+2)$ division points that give rise to the long critical intervals.
For every long critical interval $\III\subset[0,s)$, we set
\begin{displaymath}
\III\subset S_0^*
\quad\mbox{if and only if}\quad
\mbox{$\III$ is essentially $\eps$-almost entirely in $S_0$}.
\end{displaymath}
Then each of the long critical intervals in $[0,s)$ is either entirely contained in $S_0^*$ or disjoint from~$S_0^*$.
It then remains to prove that such a set $S_0^*$ cannot exist.

The two sets $S_0^*$ and $[0,s)\setminus S_0^*$ lead naturally to a $2$-coloring of the interval $[0,s)$, which in turn lead to a $2$-coloring of the $s$ vertical edges of the underlying finite polysquare-$b$-rational translation surface~$\PPP$.
We can then use the $\alpha$-flow to extend this $2$-coloring to the whole of~$\PPP$.

It is clear that $\PPP$ cannot be monochromatic, for otherwise $\meas(S_0^*)$ must be equal to $0$ or~$s$, implying that $\meas(S_0)$ is close to $0$ or~$s$, contradicting our earlier conclusion that $\meas(S_0)\in\{1,2,\ldots,s-1\}$.
Thus the $2$-coloring must have a color switch across some division point.

Suppose first of all that there is a color switch across some division point~$\{-\alpha\}$, as illustrated in the picture on the left in Figure~5.1.
Applying the reverse $\alpha$-flow takes $\{-\alpha\}$ to the point $\{-2\alpha\}$ on some vertical edge of~$\PPP$.
As $S_0$ is $T$-invariant, there must be a color switch across $\{-2\alpha\}$, a contradiction since this is not a division point.

\begin{displaymath}
\begin{array}{c}
\includegraphics{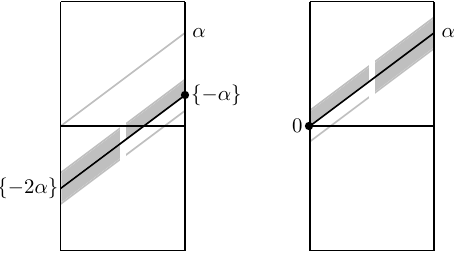}
\\
\mbox{Figure 5.1: contradicting a $2$-coloring}
\end{array}
\end{displaymath}

Suppose next that there is a color switch across some division point~$0$, as illustrated in the picture on the right in Figure~5.1.
Applying the $\alpha$-flow takes $0$ to the point $\{\alpha\}$ on some vertical edge of~$\PPP$.
As $S_0$ is $T$-invariant, there must be a color switch across $\{\alpha\}$, a contradiction since this is not a division point.

Suppose finally that the $2$-coloring has a color switch across a division point $\{r_{i_0}b-\alpha\}$, $1\le i_0\le R$, on some vertical edge of~$\PPP$.
The $\alpha$-flow moves this point to a new point on some vertical edge of~$\PPP$.
As $S_0$ is $T$-invariant, there must be a color switch across this new point, so this new point must be a division point.
Repeating this argument sufficiently long, this process must visit some division point twice, meaning that there exist some positive integers $n_1<n_2$ such that
\begin{displaymath}
\{r_{i_0}b-\alpha+n_1\alpha\}=\{r_{i_0}b-\alpha+n_2\alpha\}.
\end{displaymath}
But this implies that $(n_2-n_1)\alpha$ is an integer, contradicting the assumption that $\alpha$ is irrational.

This completes the proof of Lemma~\ref{lem51}.
\end{proof}

%
%

\section{Extending ergodicity to unique ergodicity}\label{sec6}

Lemmas \ref{lem41} and~\ref{lem51} establish ergodicity of some Lebesgue measure preserving transformation~$T$.
This means that we can apply Birkhoff's well known pointwise ergodic theorem concerning measure preserving systems $(X,\AAA,\mu,T)$.
The triple $(X,\AAA,\mu)$ is a measure space, where $X$ is the underlying space, $\AAA$ is a $\sigma$-algebra of sets in~$X$, while $\mu$ is a
non-negative $\sigma$-additive measure on $X$ with $\mu(X)<\infty$, and $T:X\to X$ is a measure-preserving transformation, so that
$T^{-1}A\in\AAA$ and $\mu(T^{-1}A)=\mu(A)$ for every $A\in\AAA$.

Let $L^1(X,\AAA,\mu)$ denote the space of measurable and integrable functions in the measure space $(X,\AAA,\mu)$.
Then Birkhoff's pointwise ergodic theorem says that for every function $f\in L^1(X,\AAA,\mu)$, the limit
\begin{equation}\label{eq6.1}
\lim_{m\to\infty}\frac{1}{m}\sum_{j=0}^{m-1}f(T^jx)=f^*(x)
\end{equation}
exists for $\mu$-almost every $x\in X$, where $f^*\in L^1(X,\AAA,\mu)$ is a $T$-invariant measurable function satisfying the condition
\begin{displaymath}
\int_Xf\,\dd\mu=\int_Xf^*\,\dd\mu.
\end{displaymath}

A particularly important special case is when $T$ is \textit{ergodic}, when every measurable $T$-invariant set $A\in\AAA$ is \textit{trivial} in the precise sense that $\mu(A)=0$ or $\mu(A)=\mu(X)$.
This is equivalent to the assertion that every measurable $T$-invariant function is constant $\mu$-almost everywhere.

If $T$ is ergodic, then \eqref{eq6.1} simplifies to
\begin{equation}\label{eq6.2}
\lim_{m\to\infty}\frac{1}{m}\sum_{j=0}^{m-1}f(T^jx)=\int_Xf\,\dd\mu,
\end{equation}
and the right-hand side of \eqref{eq6.1} is the same constant for $\mu$-almost every $x\in X$.

The remarkable intuitive interpretation of \eqref{eq6.2} is that the \textit{time average} on the left hand side is equal to the \textit{space average} on the right hand side.

Unfortunately, Birkhoff's theorem does not say anything about the speed of convergence in \eqref{eq6.1} or \eqref{eq6.2}.

In the special case of Lemma~\ref{lem51}, we have $X=[0,s)$, formed from the $s$ vertical edges of the given finite polysquare-$b$-rational translation surface~$\PPP$, the measure $\mu$ is $1$-dimensional Lebesgue measure and $T=T_\alpha$ is the interval exchange transformation corresponding to the $\alpha$-flow on~$\PPP$.
Combining Lemma~\ref{lem51} with \eqref{eq6.2}, we immediately obtain that for \textit{almost every} starting point, a half-infinite $\alpha$-geodesic is uniformly distributed on the surface~$\PPP$, a weaker version of Theorem~\ref{thm32}.
Similarly, combining Lemma~\ref{lem41} with \eqref{eq6.2}, we obtain a corresponding weaker version of Theorem~\ref{thm21}.

To prove Theorem~\ref{thm32}, we need to extend \textit{almost every} starting point to \textit{every non-pathological} starting point that gives rise to a half-infinite $\alpha$-geodesic.
In other words, we need to establish the following result.

\begin{lem}\label{lem61}
Under the hypotheses of Lemma~\ref{lem51}, consider the measure-preserving interval exchange transformation $T=T_\alpha:X\to X$ of the polyrectangle-$b$-rational surface~$\PPP$, where $X=[0,s)$.
Then for every subinterval $J\subset X$ and for every non-pathological starting point $x\in X$, we have
\begin{displaymath}
\lim_{m\to\infty}\frac{1}{m}\sum_{j=0}^{m-1}\chi_J(T^jx)
=\frac{\lambda(J)}{s},
\end{displaymath}
where $\chi_J$ denotes the characteristic function of $J$ and $\lambda$ denotes $1$-dimensional Lebesgue measure.

Using the standard extension argument, this discrete result can be converted to the continuous version concerning the uniformity of $\alpha$-geodesics  on $\PPP$ and every non-pathological starting point.
\end{lem}

\begin{proof}
The proof is by contradiction, and consists of two parts.
The first part simply follows Furstenberg's argument, and the basic idea is to reformulate the problem of unique ergodicity in terms of $T$-invariant Borel measures, and to apply nontrivial results from functional analysis. 
The key idea of the second part is then an application of Birkhoff's ergodic theorem to a new measure that is different from~$\lambda$.

The first part of the argument is summarized in the following lemma.

\begin{lem}\label{lem62}
Suppose that there exist a non-pathological starting point $y_0\in X$ and an interval $J_0\subset X$ for which uniformity fails, so that the infinite sequence
\begin{equation}\label{eq6.3}
\frac{1}{m}\sum_{j=0}^{m-1}\chi_{J_0}(T^jy_0),
\quad
m\ge1,
\end{equation}
where $\chi_{J_0}$ is the characteristic function of~$J_0$, does not converge to $\lambda(J_0)/s$.
Then there exists an ergodic measure-preserving system $(X,\BBB,\nu,T)$, where $\BBB$ is the Borel $\sigma$-algebra on~$X$, and
$\nu$ is a new $T$-invariant Borel probability measure, such that
\begin{equation}\label{eq6.4}
\nu(J_0)\ne\frac{\lambda(J_0)}{\lambda(X)}=\frac{\lambda(J_0)}{s}.
\end{equation}
\end{lem}

\begin{proof}
In view of the assumption, there exists an infinite subsequence
\begin{equation}\label{eq6.5}
0\le h_0< h_1<h_2<h_3<\ldots
\end{equation}
of the non-negative integers such that the limit
\begin{equation}\label{eq6.6}
\lim_{m\to\infty}\frac{1}{h_m}\sum_{j=0}^{h_m-1}\chi_{J_0}(T^jy_0)
\end{equation}
exists, but is not equal to $\lambda(J_0)/s$.

\begin{claim}
There exist another infinite subsequence
\begin{displaymath}
1\le d_1<d_2<d_3<\ldots
\end{displaymath}
of the positive integers and an infinite sequence of corresponding starting points $y(m)\in X$, $m=1,2,3,\ldots,$ such that the limit
\begin{equation}\label{eq6.7}
\lim_{m\to\infty}\frac{1}{q_{d_m}}\sum_{j=0}^{q_{d_m}-1}\chi_{J_0}(T^jy(m))
\end{equation}
exists, but is not equal to $\lambda(J_0)/s$.
Here, for every integer $m=1,2,3,\ldots,$ the number $q_{d_m}$ represents the denominator of the $d_m$-th convergent of~$\alpha$.
\end{claim}

We shall justify this Claim at the end of our proof of Lemma~\ref{lem62}.

We now repeat and adapt some ideas in \cite[Sections 3.2--3.3]{F}.
For every integer $m\ge1$, we introduce the normalized counting measure~$\nu_m$, defined for every Borel set $B\subset X$ by
\begin{equation}\label{eq6.8}
\nu_m(B)=\frac{1}{q_{d_m}}\sum_{j=0}^{q_{d_m}-1}\chi_B(T^{j}y(m)),
\end{equation}
where $\chi_B$ is the characteristic function of~$B$.

Now we make use of a general theorem in functional analysis which says that the space of Borel probability measures on any compact set is compact in the so-called \textit{weak-star topology}. 
The latter means that
\begin{displaymath}
\mu_m\to\mu
\quad\mbox{if and only if}\quad
\int f\,\dd\mu_m\to\int f\,\dd\mu,
\end{displaymath}
where $f$ runs over all continuous functions on the compact space.

This compactness theorem is a non-trivial result.
The standard proof is based on the Riesz Representation Theorem.

Let $\MMM$ denote the set of Borel probability measures $\mu$ on~$X$.
By the general theorem, $\MMM$ is compact.
Let $\MMM_1\subset\MMM$ denote the set of those Borel probability measures $\mu$ on $X$ that are $T$-invariant and such that
$\mu\ne\lambda/s$.
It is obvious that $\MMM_1$ is a closed subset of~$\MMM$ and therefore compact. 

The compactness of $\MMM$ implies that there is a subsequence $\nu_{m_i}$ of the sequence $\nu_m$ defined by \eqref{eq6.8} such that
$\nu_{m_i}\to\nu_{\infty}$ as $i\to\infty$, where $\nu_{\infty}$ is a Borel probability measure on~$X$.
It easily follows from \eqref{eq6.8} that $\nu_{\infty}$ is $T$-invariant.
Indeed, writing $y_1(m)=Ty(m)$, we have
\begin{align}
\nu_m(T^{-1}B)
&
=\frac{1}{q_{d_m}}\sum_{j=0}^{q_{d_m}-1}\chi_B(T^{j}y_1(m))
=\frac{1}{q_{d_m}}\sum_{j=1}^{q_{d_m}}\chi_B(T^{j}y(m))
\nonumber
\\
&
=\frac{1}{q_{d_m}}\sum_{j=0}^{q_{d_m}-1}\chi_B(T^{j}y(m))
+\frac{\chi_B(T^{q_{d_m}}y(m))-\chi_B(y(m))}{q_{d_m}}
\nonumber
\\
&
=\nu_m(B)
+\frac{\chi_B(T^{q_{d_m}}y(m))-\chi_B(y(m))}{q_{d_m}},
\nonumber
\end{align}
and
\begin{displaymath}
\left\vert\frac{\chi_B(T^{q_{d_m}}y(m))-\chi_B(y(m))}{q_{d_m}}\right\vert\le\frac{1}{q_{d_m}}\to0
\quad\mbox{as $m\to\infty$}.
\end{displaymath}
Moreover, the limit measure $\nu_{\infty}$ clearly satisfies the requirement \eqref{eq6.4}, implying that $\nu_{\infty}\in\MMM_1$, and so $\MMM_1$ is a non-empty compact set.

To find an appropriate $\nu\in\MMM_1$ which guarantees that the measure-preserving system $(X,\BBB,\nu,T)$ is ergodic, we use the almost trivial fact that $\MMM_1$ is convex.
The well known Krein--Milman theorem in functional analysis implies that the non-empty convex set $\MMM_1$ is spanned by its \textit{extremal points}.
It is a well known general result in ergodic theory that the extremal points are precisely the ergodic $T$-invariant measures; see
\cite[Proposition~3.4]{F}.
Thus we can choose our measure $\nu\in\MMM_1$ to be such an extremal point, and this completes the deduction of the lemma.
It remains to establish the Claim.

To establish the Claim, we use the $\alpha$-representation, or Ostrowski representation, of an arbitrary integer $N\ge1$.

It is well known that every integer $N\ge1$ has a unique representation in the form
\begin{displaymath}
N=\sum_{k=0}^nb_kq_k=\sum_{k=0}^nb_k(N)q_k,
\end{displaymath}
where the integer coefficients $b_0,\ldots,b_n$ satisfy the conditions
\begin{equation}\label{eq6.9}
0\le b_0<a_1,
\quad
0<b_n\le a_{n+1},
\quad
0\le b_k\le a_{k+1},
\quad
k=1,\ldots,n-1,
\end{equation}
as well as the restrictions
\begin{equation}\label{eq6.10}
b_{k-1}=0
\quad\mbox{if}\quad
b_k=a_{k+1},
\quad
k=1,\ldots,n,
\end{equation}
where $a_1,a_2,a_3,\ldots$ are digits of the continued fraction \eqref{eq5.2} of~$\alpha$, and $q_0,q_1,q_2,\ldots$ are the denominators of the successive convergents of~$\alpha$.
Furthermore, the value of the integer $n$ is determined by the inequalities $q_n\le N<q_{n+1}$.

We now do likewise for the sequence of integers $h_\ell$, $\ell=0,1,2,3,\ldots,$ in \eqref{eq6.5}, so that we have the Ostrowski representations
\begin{displaymath}
h_\ell=\sum_{k=0}^{n_\ell}b_{k,\ell}q_k,
\quad
\ell=0,1,2,3,\ldots,
\end{displaymath}
where the coefficients $b_{k,\ell}$ satisfy conditions analogous to \eqref{eq6.9} and \eqref{eq6.10}.

Next, observe that
\begin{equation}\label{eq6.11}
\{0,1,\ldots,h_\ell-1\}
=\bigcup_{k_0=0}^{n_\ell-1}\left\{\sum_{k=0}^{k_0-1}b_{k,\ell}q_k,\ldots,\sum_{k=0}^{k_0}b_{k,\ell}q_k-1\right\},
\end{equation}
where each set in the union is a collection of consecutive integers.
Write
\begin{equation}\label{eq6.12}
N_\ell(k_0)=\sum_{k=0}^{k_0-1}b_{k,\ell}q_k,
\quad
k_0=0,1,\ldots,n_\ell-1.
\end{equation}
Then for each $k_0=0,1,\ldots,n_\ell-1$, we have
\begin{equation}\label{eq6.13}
\left\{\sum_{k=0}^{k_0-1}b_{k,\ell}q_k,\ldots,\sum_{k=0}^{k_0}b_{k,\ell}q_k-1\right\}
=\bigcup_{b=0}^{b_{k_0,\ell}-1}\bigcup_{j_0=0}^{q_{k_0}-1}\{N_\ell(k_0)+bq_{k_0}+j_0\}.
\end{equation}
It now follows from \eqref{eq6.11}--\eqref{eq6.13} that
\begin{equation}\label{eq6.14}
\sum_{j=0}^{h_\ell-1}\chi_{J_0}(T^jy_0)
=\sum_{k_0=0}^{n_\ell-1}\sum_{b=0}^{b_{k_0,\ell}-1}\sum_{j_0=0}^{q_{k_0}-1}\chi_{J_0}(T^{N_\ell(k_0)+bq_{k_0}+j_0}y_0).
\end{equation}
Write $y_\ell(k_0,b)=T^{N_\ell(k_0)+bq_{k_0}}y_0$.
Replacing the dummy variables $k_0$ and $j_0$ by $k$ and $j$ respectively in \eqref{eq6.14}, we then conclude that
\begin{equation}\label{eq6.15}
\sum_{j=0}^{h_\ell-1}\chi_{J_0}(T^jy_0)
=\sum_{k=0}^{n_\ell-1}\sum_{b=0}^{b_{k,\ell}-1}\sum_{j=0}^{q_k-1}\chi_{J_0}(T^jy_\ell(k,b)).
\end{equation}
Motivated by \eqref{eq6.15}, write
\begin{displaymath}
\eps(\ell;k;b)=\left\vert\frac{1}{q_k}\sum_{j=0}^{q_k-1}\chi_{J_0}(T^jy_\ell(k,b))-\frac{\lambda(J_0)}{s}\right\vert,
\end{displaymath}
and
\begin{displaymath}
\eps(k)=\sup_{\substack{{\ell\ge0}\\{0\le b<b_{k,\ell}}}}\eps(\ell;k;b).
\end{displaymath}
Noting that the limit \eqref{eq6.6} exists and is not equal to $\lambda(J_0)/s$, it is clear that $\eps(k)$ does not tend to zero as $k\to\infty$.
Hence there exists an infinite sequence
\begin{displaymath}
1\le d_1<d_2<d_3<\ldots
\end{displaymath}
of integers and a positive $\eps_0>0$ such that $\eps(d_k)\ge\eps_0$ for all $k\ge1$.
This clearly implies that the limit \eqref{eq6.7} exists, but is not equal to $\lambda(J_0)/s$.
This completes the proof of the Claim and also of Lemma~\ref{lem62}.
\end{proof}

Since the measure-preserving system $(X,\BBB,\nu,T)$ given by Lemma~\ref{lem62} is ergodic, it follows from Birkhoff's ergodic theorem that for every Borel set $B\in\BBB$ and for $\nu$-almost every $y\in X$, we have
\begin{equation}\label{eq6.16}
\lim_{m\to\infty}\frac{1}{m}\sum_{j=0}^{m-1}\chi_B(T^jy)=\nu(B).
\end{equation}

Let $W$ be an arbitrarily large but fixed positive integer.
We claim that there exists a non-empty open interval $Q=Q(W)\subset X$ such that
\begin{equation}\label{eq6.17}
\frac{\nu(Q)}{\lambda(Q)}>W.
\end{equation}
Thus the measure $\nu$ can be \textit{arbitrarily more dense} in some subintervals $Q\subset X$ than the Lebesgue measure~$\lambda$.  

To prove \eqref{eq6.17}, we choose $B=J_0$ in \eqref{eq6.3}, and consider the set
\begin{equation}\label{eq6.18}
Y=\left\{y\in X:\lim_{m\to\infty}\frac{1}{m}\sum_{j=0}^{m-1}\chi_{J_0}(T^jy)=\nu(J_0)\right\}.
\end{equation}
We already know that for $\lambda$-almost every $y\in X$, we have
\begin{equation}\label{eq6.19}
\lim_{m\to\infty}\frac{1}{m}\sum_{j=0}^{m-1}\chi_{J_0}(T^jy)
=\frac{\lambda(J_0)}{s}.
\end{equation}
Combining \eqref{eq6.4}, \eqref{eq6.16}, \eqref{eq6.18} and \eqref{eq6.19}, we conclude that
\begin{equation}\label{eq6.20}
\nu(Y)=1
\quad\mbox{and}\quad
\lambda(Y)=0.
\end{equation}
Let $\delta>0$ be arbitrarily small but fixed.
Since $\lambda(Y)=0$, there exists an infinite sequence $R_i$, $i\ge1$, of open intervals such that
\begin{equation}\label{eq6.21}
\sum_{i=1}^{\infty}\lambda(R_i)<\delta
\quad\mbox{and}\quad
Y\subset\bigcup_{i=1}^{\infty}R_i.
\end{equation}
By \eqref{eq6.20} and \eqref{eq6.21}, we have
\begin{equation}\label{eq6.22}
\sum_{i=1}^{\infty}\nu(R_i)\ge1.
\end{equation}
It follows from \eqref{eq6.21} and \eqref{eq6.22} that there exists an integer $i_0\ge1$ such that
\begin{equation}\label{eq6.23}
\frac{\lambda(R_{i_0})}{\nu(R_{i_0})}<\delta.
\end{equation}
Choosing $\delta=1/W$ in \eqref{eq6.23}, the inequality \eqref{eq6.17} follows with the choice $Q=R_{i_0}$.

Next we derive a contradiction from \eqref{eq6.8} and \eqref{eq6.17}.
In \eqref{eq6.8} the orbits
\begin{equation}\label{eq6.24}
\XXX=\{T^jy(m):0\le j<q_{d_m}\}
\end{equation}
of $\nu_m$ have sizes equal to the denominator of a convergent of~$\alpha$.
We shall show that they are \textit{uniformly not crowded}.
Then $\nu$, being a limit point of the set of counting measures~$\nu_m$, cannot be \textit{arbitrarily more dense} than~$\lambda$.
This then contradicts \eqref{eq6.17}.

To prove the sets $\XXX$ in \eqref{eq6.24} are \textit{uniformly not crowded}, we recall that $T$ modulo~$1$ is the $\alpha$-shift in the unit torus/circle
$[0,1)$.
We also recall from \eqref{eq4.36} that
\begin{displaymath}
\left\vert j\alpha-\frac{jp_{d_m}}{q_{d_m}}\right\vert\le\frac{j}{q_{d_m+1}q_{d_m}} \le\frac{1}{q_{d_m+1}}<\frac{1}{q_{d_m}},
\quad
0\le j<q_{d_m},
\end{displaymath}
which implies
\begin{equation}\label{eq6.25}
\left\vert\{j\alpha\}-\left\{\frac{jp_{d_m}}{q_{d_m}}\right\}\right\vert<\frac{1}{q_{d_m}},
\quad
0\le j<q_{d_m}.
\end{equation}
Since
\begin{equation}\label{eq6.26}
\left\{\frac{jp_{d_m}}{q_{d_m}}\right\},
\quad
0\le j<q_{d_m},
\end{equation}
gives an \textit{equipartition} of the unit torus/circle $[0,1)$, the points in \eqref{eq6.26} exhibit a best possible form of quantitative uniformity.
Combining this fact with \eqref{eq6.25}, we deduce that  for every subinterval $J\subset X=[0,s)$ with $\lambda(J)\ge s/q_{d_m}$, we have
\begin{equation}\label{eq6.27}
\vert J\cap\XXX\vert\le s\left(q_{d_m}\frac{\lambda(J)}{s}+2\right),
\end{equation}
where $\vert J\cap\XXX\vert$ denotes the number of elements of $\XXX$ in the interval $J$ of length $\lambda(J)$, and the factor $s$ comes from
the fact that there are $s$ atomic squares in~$\PPP$.
The bound \eqref{eq6.27} proves that the set $\XXX\subset X=[0,s)$ in \eqref{eq6.24} is \textit{uniformly not crowded}.
This completes the proof of Lemma~\ref{lem61}.
\end{proof}

The slope in Theorem~\ref{thm32} is badly approximable.
However, in this section this special property of $\alpha$ is never used, only that it is irrational.
Thus we can routinely repeat the same extension argument to Lemma~\ref{lem41} for the $2$-square-$b$ surface, and obtain the unique ergodicity for \textit{every} irrational slope and complete the proof of Theorem~\ref{thm21}. 

\begin{remark}
The expert reader may well be wondering why we have not established Theorem~\ref{thm32} by using Boshernitzan's Criterion for unique ergodicity of an interval exchange transformation as given in~\cite{V2}.
Unfortunately, we are not able to see how we may prove Theorem~\ref{thm21} via this method, and have already developed a different technique in Section~\ref{sec4}.
It therefore seems natural to adapt this other technique in Section~\ref{sec5} in the case of Theorem~\ref{thm32}.
\end{remark}

%
%

\section{proof of the separation lemma}\label{sec7}

The proof of Lemma~\ref{lem52} depends on the following very simple property of the given badly approximable number~$\alpha$.

\begin{lem}\label{lem71}
Let $\alpha\in(0,\infty)$ be badly approximable, with continued fraction \eqref{eq5.2} and digits satisfying \eqref{eq5.3}, where $A$ is a fixed positive integer.
Then for every integer $n\ge1$, we have
\begin{displaymath}
\Vert n\alpha\Vert>\frac{1}{(A+2)n}.
\end{displaymath}
\end{lem}

\begin{proof}
For every integer $n\ge1$, we can find an integer $k\ge0$ such that $q_k\le n<q_{k+1}$, where $q_k=q_k(\alpha)$ denotes the denominator of the $k$-th convergent of~$\alpha$.
Using well known diophantine approximation properties of continued fractions, we have
\begin{displaymath}
\Vert n\alpha\Vert
\ge\Vert q_k\alpha\Vert
\ge\frac{1}{q_k+q_{k+1}}
>\frac{1}{q_k+(a_{k+1}+1)q_k}
\ge\frac{1}{(A+2)q_k}
\ge\frac{1}{(A+2)n},
\end{displaymath}
as required.
\end{proof}

\begin{proof}[Proof of Lemma~\ref{lem52}]
Suppose that the integer $k_0\ge1$ violates \eqref{eq5.12} or \eqref{eq5.13}.
Then at least one of the conditions (V1)--(V3) holds:

(V1)
There exist $i=1,\ldots,R$ and $\sigma=\pm$ such that
\begin{equation}\label{eq7.1}
h_{k_0}(i;\sigma)\le\delta q_{k_0+1}.
\end{equation}
In this case, it follows from \eqref{eq5.10} that
\begin{equation}\label{eq7.2}
\Vert(h_{k_0}(i;\sigma)+1)\alpha-r_ib\Vert
=\Vert\{h_{k_0}(i;\sigma)\alpha\}-\{r_ib-\alpha\}\Vert
<\frac{2}{q_{k_0+1}}.
\end{equation}
Writing $r_i=u_i/v_i$ and multiplying by $v_i$ leads to the inequality
\begin{displaymath}
\Vert v_i(h_{k_0}(i;\sigma)+1)\alpha-u_ib\Vert
<\frac{2v_i}{q_{k_0+1}}.
\end{displaymath}
Thus there exists a real number $x_1>1$ such that
\begin{equation}\label{eq7.3}
\Vert v_i(h_{k_0}(i;\sigma)+1)\alpha-u_ib\Vert
=\frac{2v_i}{x_1q_{k_0+1}}.
\end{equation}

(V2)
There exist $i=1,\ldots,R$ and $\sigma=\pm$ such that
\begin{displaymath}
h_{k_0}(i;\sigma)\ge(1-\delta)q_{k_0+1}.
\end{displaymath}
In this case, write
\begin{displaymath}
h_{k_0}^-(i;\sigma)=q_{k_0+1}-h_{k_0}(i;\sigma).
\end{displaymath}
This is equivalent to
\begin{displaymath}
(-h_{k_0}^-(i;\sigma)\alpha-(r_ib-\alpha))
-(h_{k_0}(i;\sigma)\alpha-(r_ib-\alpha))
=q_{k_0+1}\alpha,
\end{displaymath}
and it follows from the triangle inequality and \eqref{eq4.36} that
\begin{displaymath}
\vert\Vert-h_{k_0}^-(i;\sigma)\alpha-(r_ib-\alpha)\Vert
-\Vert h_{k_0}(i;\sigma)\alpha-(r_ib-\alpha)\Vert\vert
\le\Vert q_{k_0+1}\alpha\Vert
\le\frac{1}{q_{k_0+2}},
\end{displaymath}
and then from \eqref{eq5.10} that
\begin{equation}\label{eq7.4}
\Vert(-h_{k_0}^-(i;\sigma)+1)\alpha-r_ib\Vert
\le\Vert(h_{k_0}(i;\sigma)+1)\alpha-r_ib\Vert
+\Vert q_{k_0+1}\alpha\Vert
<\frac{3}{q_{k_0+1}}.
\end{equation}
Writing $r_i=u_i/v_i$ and multiplying by $v_i$ leads to the inequality
\begin{displaymath}
\Vert v_i(-h_{k_0}^-(i;\sigma)+1)\alpha-u_ib\Vert
<\frac{3v_i}{q_{k_0+1}}.
\end{displaymath}
Thus there exists a real number $x_2>1$ such that
\begin{displaymath}
\Vert v_i(-h_{k_0}^-(i;\sigma)+1)\alpha-u_ib\Vert
=\frac{3v_i}{x_2q_{k_0+1}}.
\end{displaymath}
Note in particular that
\begin{displaymath}
2\le h_{k_0}^-(i;\sigma)\le\delta q_{k_0+1}.
\end{displaymath}

(V3)
There exist $i_1,i_2=1,\ldots,R$ and $\sigma_1,\sigma_2=\pm$ such that $(i_1,\sigma_1)\ne(i_2,\sigma_2)$ and
\begin{displaymath}
\vert h_{k_0}(i_1;\sigma_1)-h_{k_0}(i_2;\sigma_2)\vert\le\delta q_{k_0+1}.
\end{displaymath}
For notational simplicity, we assume that
\begin{displaymath}
h_{k_0}(i_1;\sigma_1)>h_{k_0}(i_2;\sigma_2)
\quad\mbox{and}\quad
\{r_{i_1}b-\alpha\}>\{r_{i_2}b-\alpha\}.
\end{displaymath}
The argument for the other possibilities requires only minor modification.
Then
\begin{displaymath}
\{r_{i_1}b-\alpha\}-\{r_{i_2}b-\alpha\}=\{(r_{i_1}-r_{i_2})b\},
\end{displaymath}
and it follows from the triangle inequality and \eqref{eq5.10} that
\begin{align}
&
\Vert(h_{k_0}(i_1;\sigma_1)-h_{k_0}(i_2;\sigma_2))\alpha-(r_{i_1}-r_{i_2})b\Vert
\nonumber
\\
&\quad
=\Vert(h_{k_0}(i_1;\sigma_1)-h_{k_0}(i_2;\sigma_2))\alpha-\{(r_{i_1}-r_{i_2})b\}\Vert
\nonumber
\\
&\quad
=\Vert(h_{k_0}(i_1;\sigma_1)-h_{k_0}(i_2;\sigma_2))\alpha-\{r_{i_1}b-\alpha\}+\{r_{i_2}b-\alpha\}\Vert
\nonumber
\\
&\quad
\le\Vert h_{k_0}(i_1;\sigma_1)\alpha-\{r_{i_1}b-\alpha\}\Vert
+\Vert h_{k_0}(i_2;\sigma_2)\alpha-\{r_{i_2}b-\alpha\}\Vert
<\frac{4}{q_{k_0+1}}.
\nonumber
\end{align}
Write
\begin{displaymath}
h_{k_0}(i_1,i_2;\sigma_1,\sigma_2)=h_{k_0}(i_1;\sigma_1)-h_{k_0}(i_2;\sigma_2).
\end{displaymath}
Then it clearly follows that
\begin{equation}\label{eq7.5}
\Vert h_{k_0}(i_1,i_2;\sigma_1,\sigma_2)\alpha-(r_{i_1}-r_{i_2})b\Vert<\frac{4}{q_{k_0+1}}.
\end{equation}
Writing $r_{i_1}=u_{i_1}/v_{i_1}$, $r_{i_2}=u_{i_2}/v_{i_2}$ and multiplying by $v_{i_1}v_{i_2}$ leads to the inequality
\begin{displaymath}
\Vert v_{i_1}v_{i_2}h_{k_0}(i_1,i_2;\sigma_1,\sigma_2)\alpha-(u_{i_1}v_{i_2}-u_{i_2}v_{i_1})b\Vert<\frac{4v_{i_1}v_{i_2}}{q_{k_0+1}}.
\end{displaymath}
Thus there exists a real number $x_3>1$ such that
\begin{displaymath}
\Vert v_{i_1}v_{i_2}h_{k_0}(i_1,i_2;\sigma_1,\sigma_2)\alpha-(u_{i_1}v_{i_2}-u_{i_2}v_{i_1})b\Vert=\frac{4v_{i_1}v_{i_2}}{x_3q_{k_0+1}}.
\end{displaymath}
Note in particular that
\begin{displaymath}
1\le h_{k_0}(i_1,i_2;\sigma_1,\sigma_2)\le\delta q_{k_0+1}.
\end{displaymath}

Next let $k>k_0$ be any integer which violates \eqref{eq5.12} or \eqref{eq5.13}.
We distinguish various cases.

\begin{case1a}
Suppose that $k_0$ satisfies (V1) and $k$ satisfies the $k$-analog of (V1).
Then there exist $i^*=1,\ldots,R$ and $\sigma^*=\pm$ such that
\begin{equation}\label{eq7.6}
h_k(i^*;\sigma^*)\le\delta q_{k+1}.
\end{equation}
Writing $r_{i^*}=u_{i^*}/v_{i^*}$ and multiplying the analog of \eqref{eq7.2} by $u_iv_{i^*}$ leads to the inequality
\begin{equation}\label{eq7.7}
\Vert u_iv_{i^*}(h_k(i^*;\sigma^*)+1)\alpha-u_iu_{i^*}b\Vert
<\frac{2\vert u_iv_{i^*}\vert}{q_{k+1}}.
\end{equation}
On the other hand, multiplying \eqref{eq7.3} by $u_{i^*}$, we obtain
\begin{equation}\label{eq7.8}
\Vert v_iu_{i^*}(h_{k_0}(i;\sigma)+1)\alpha-u_iu_{i^*}b\Vert
=\frac{2\vert v_iu_{i^*}\vert}{x_1q_{k_0+1}}.
\end{equation}
We shall show that the integer
\begin{equation}\label{eq7.9}
d_{11}=v_iu_{i^*}(h_{k_0}(i;\sigma)+1)-u_iv_{i^*}(h_k(i^*;\sigma^*)+1)\ne0
\end{equation}
if the condition
\begin{equation}\label{eq7.10}
\frac{UV}{q_{k+1}}\le\frac{1}{x_1q_{k_0+1}}
\end{equation}
holds.
Indeed, combining \eqref{eq5.9}, \eqref{eq7.7}, \eqref{eq7.8} and \eqref{eq7.10}, we see that
\begin{align}
&
\Vert u_iv_{i^*}(h_k(i^*;\sigma^*)+1)\alpha-u_iu_{i^*}b\Vert
<\frac{2UV}{q_{k+1}}
\le\frac{2}{x_1q_{k_0+1}}
\nonumber
\\
&\quad
\le\frac{2\vert v_iu_{i^*}\vert}{x_1q_{k_0+1}}
=\Vert v_iu_{i^*}(h_{k_0}(i;\sigma)+1)\alpha-u_iu_{i^*}b\Vert.
\nonumber
\end{align}
This clearly implies that $d_{11}\ne0$.
\end{case1a}

\begin{case1b}
Suppose that $k_0$ satisfies (V1) and $k$ satisfies the $k$-analog of (V2).
Then there exist $i^*=1,\ldots,R$ and $\sigma^*=\pm$ such that
\begin{displaymath}
h_k^-(i^*;\sigma^*)=q_{k+1}-h_k(i^*;\sigma^*)\le\delta q_{k+1}.
\end{displaymath}
Writing $r_{i^*}=u_{i^*}/v_{i^*}$ and multiplying the analog of \eqref{eq7.4} by $u_iv_{i^*}$ leads to the inequality
\begin{equation}\label{eq7.11}
\Vert u_iv_{i^*}(-h_k^-(i^*;\sigma^*)+1)\alpha-u_iu_{i^*}b\Vert
<\frac{3\vert u_iv_{i^*}\vert}{q_{k+1}}.
\end{equation}
On the other hand, multiplying \eqref{eq7.3} by $u_{i^*}$, we obtain \eqref{eq7.8}.
We shall show that the integer
\begin{displaymath}
d_{12}=v_iu_{i^*}(h_{k_0}(i;\sigma)+1)-u_iv_{i^*}(-h_k^-(i^*;\sigma^*)+1)\ne0
\end{displaymath}
if the condition
\begin{equation}\label{eq7.12}
\frac{3UV}{2q_{k+1}}\le\frac{1}{x_1q_{k_0+1}}
\end{equation}
holds.
Indeed, combining \eqref{eq5.9}, \eqref{eq7.8}, \eqref{eq7.11} and \eqref{eq7.12}, we see that
\begin{align}
&
\Vert u_iv_{i^*}(-h_k^-(i^*;\sigma^*)+1)\alpha-u_iu_{i^*}b\Vert
<\frac{3UV}{q_{k+1}}
\le\frac{2}{x_1q_{k_0+1}}
\nonumber
\\
&\quad
\le\frac{2\vert v_iu_{i^*}\vert}{x_1q_{k_0+1}}
=\Vert v_iu_{i^*}(h_{k_0}(i;\sigma)+1)\alpha-u_iu_{i^*}b\Vert.
\nonumber
\end{align}
This clearly implies that $d_{12}\ne0$.
\end{case1b}

\begin{case1c}
Suppose that $k_0$ satisfies (V1) and $k$ satisfies the $k$-analog of (V3).
Then there exist $i_1^*,i_2^*=1,\ldots,R$ and $\sigma_1^*,\sigma_2^*=\pm$ such that $(i_1^*,\sigma_1^*)\ne(i_2^*,\sigma_2^*)$ and
\begin{displaymath}
\vert h_k(i_1^*;\sigma_1^*)-h_k(i_2^*;\sigma_2^*)\vert\le\delta q_{k+1}.
\end{displaymath}
For notational simplicity, we assume that
\begin{displaymath}
h_k(i_1^*;\sigma_1^*)>h_k(i_2^*;\sigma_2^*)
\quad\mbox{and}\quad
\{r_{i_1^*}b-\alpha\}>\{r_{i_2^*}b-\alpha\}.
\end{displaymath}
The argument for the other possibilities requires only minor modification.

Writing $r_{i_1^*}=u_{i_1^*}/v_{i_1^*}$, $r_{i_2^*}=u_{i_2^*}/v_{i_2^*}$ and multiplying the analog of \eqref{eq7.5} by $u_iv_{i_1^*}v_{i_2^*}$ leads to the inequality
\begin{equation}\label{eq7.13}
\Vert u_iv_{i_1^*}v_{i_2^*}h_k(i_1^*,i_2^*;\sigma_1^*,\sigma_2^*)\alpha-u_i(u_{i_1^*}v_{i_2^*}-v_{i_1^*}u_{i_2^*})b\Vert
<\frac{4\vert u_iv_{i_1^*}v_{i_2^*}\vert}{q_{k+1}}.
\end{equation}
On the other hand, multiplying \eqref{eq7.3} by $u_{i_1^*}v_{i_2^*}-v_{i_1^*}u_{i_2^*}$, we obtain
\begin{align}\label{eq7.14}
&
\Vert v_i(u_{i_1^*}v_{i_2^*}-v_{i_1^*}u_{i_2^*})(h_{k_0}(i;\sigma)+1)\alpha-u_i(u_{i_1^*}v_{i_2^*}-v_{i_1^*}u_{i_2^*})b\Vert
\nonumber
\\
&\quad
=\frac{2\vert v_i(u_{i_1^*}v_{i_2^*}-v_{i_1^*}u_{i_2^*})\vert}{x_1q_{k_0+1}}.
\end{align}
We shall show that the integer
\begin{displaymath}
d_{13}=v_i(u_{i_1^*}v_{i_2^*}-v_{i_1^*}u_{i_2^*})(h_{k_0}(i;\sigma)+1)-u_iv_{i_1^*}v_{i_2^*}h_k(i_1^*,i_2^*;\sigma_1^*,\sigma_2^*)\ne0
\end{displaymath}
if the condition
\begin{equation}\label{eq7.15}
\frac{2UV^2}{q_{k+1}}\le\frac{1}{x_1q_{k_0+1}}
\end{equation}
holds.
Indeed, combining \eqref{eq5.9}, \eqref{eq7.13}, \eqref{eq7.14} and \eqref{eq7.15}, we see that
\begin{align}
&
\Vert u_iv_{i_1^*}v_{i_2^*}h_k(i_1^*,i_2^*;\sigma_1^*,\sigma_2^*)\alpha-u_i(u_{i_1^*}v_{i_2^*}-v_{i_1^*}u_{i_2^*})b\Vert
\nonumber
\\
&\quad
<\frac{4UV^2}{q_{k+1}}
\le\frac{2}{x_1q_{k_0+1}}
\le\frac{2\vert v_i(u_{i_1^*}v_{i_2^*}-v_{i_1^*}u_{i_2^*})\vert}{x_1q_{k_0+1}}.
\nonumber
\\
&\quad
=\Vert v_i(u_{i_1^*}v_{i_2^*}-v_{i_1^*}u_{i_2^*})(h_{k_0}(i;\sigma)+1)\alpha-u_i(u_{i_1^*}v_{i_2^*}-v_{i_1^*}u_{i_2^*})b\Vert.
\nonumber
\end{align}
This clearly implies that $d_{13}\ne0$.
\end{case1c}

Let us compare the requirements \eqref{eq7.10}, \eqref{eq7.12} and \eqref{eq7.15}.
Clearly the last one is the strongest requirement.
Let $k_1$ be the smallest integer such that the inequality
\begin{displaymath}
\frac{2UV^2}{q_{k_1+1}}\le\frac{1}{x_1q_{k_0+1}}
\end{displaymath}
holds, so that in particular, we have
\begin{displaymath}
q_{k_1+1}\ge2UV^2x_1q_{k_0+1}.
\end{displaymath}
Note that $q_{k_1+1}$ is the denominator of a convergent of the continued fraction of the badly approximable number~$\alpha$, and clearly $k_1>k_0$.
Recall \eqref{eq5.3} that the continued fraction digits are bounded by an integer~$A$.
Using the recurrence relations \eqref{eq4.37}, we see that for every real number $X\ge1$, there exists a denominator $q_k$ between $X$ and $(1+A)X$.
The minimality property of $k_1$ then ensures that
\begin{equation}\label{eq7.16}
q_{k_1+1}\le2(A+1)UV^2x_1q_{k_0+1}.
\end{equation}
We shall show that this $k_1>k_0$ belongs to~$\KKK_0$, and prove this by contradiction.

Suppose on the contrary that $k_1\not\in\KKK_0$.
Then one of the cases 1A, 1B and 1C holds with $k=k_1$.

Suppose that Case~1A holds with $k=k_1$.
Starting with \eqref{eq7.7} with $k=k_1$, \eqref{eq7.8} and \eqref{eq7.9}, applying the triangle inequality, and then using \eqref{eq7.16}, we obtain
\begin{align}\label{eq7.17}
\Vert d_{11}\alpha\Vert
&
=\Vert v_iu_{i^*}(h_{k_0}(i;\sigma)+1)\alpha-u_iv_{i^*}(h_{k_1}(i^*;\sigma^*)+1)\alpha\Vert
\nonumber
\\
&
\le\Vert v_iu_{i^*}(h_{k_0}(i;\sigma)+1)\alpha-u_iu_{i^*}b\Vert
+\Vert u_iv_{i^*}(h_{k_1}(i^*;\sigma^*)+1)\alpha-u_iu_{i^*}b\Vert
\nonumber
\\
&
<\frac{2\vert v_iu_{i^*}\vert}{x_1q_{k_0+1}}+\frac{2\vert u_iv_{i^*}\vert}{q_{k_1+1}}
\le\frac{4(A+1)UV^2\vert v_iu_{i^*}\vert}{q_{k_1+1}}+\frac{2\vert u_iv_{i^*}\vert}{q_{k_1+1}}
\nonumber
\\
&
\le\frac{4(A+2)U^2V^3}{q_{k_1+1}}.
\end{align}
Let $n=\vert d_{11}\vert$ with $k=k_1$.
Then $n\ge1$.
Using \eqref{eq7.1} and \eqref{eq7.6} with $k=k_1$, we have
\begin{align}\label{eq7.18}
n
&
=\vert v_iu_{i^*}(h_{k_0}(i;\sigma)+1)-u_iv_{i^*}(h_{k_1}(i^*;\sigma^*)+1)\vert
\nonumber
\\
&
\le2\delta UVq_{k_0+1}+2\delta UVq_{k_1+1}
<4\delta UVq_{k_1+1}.
\end{align}
Applying Lemma~\ref{lem71} and using \eqref{eq7.18}, we deduce that
\begin{equation}\label{eq7.19}
\Vert n\alpha\Vert
>\frac{1}{(A+2)n}
>\frac{1}{4\delta(A+2)UVq_{k_1+1}}.
\end{equation}
Combining \eqref{eq7.17} and \eqref{eq7.19}, and noting that $\Vert n\alpha\Vert=\Vert d_{11}\alpha\Vert$, we conclude that
\begin{equation}\label{eq7.20}
\delta>\frac{1}{16(A+2)^2U^3V^4}.
\end{equation}
Clearly \eqref{eq7.20} contradicts the definition of $\delta$ as given by \eqref{eq5.11}.
It then follows that Case~1A does not hold with $k=k_1$.

Essentially similar arguments show that Case~1B and Case~1C also do not hold with $k=k_1$.

It follows that if $k_0$ satisfies (V1), then there exists $k_1>k_0$ such that $k_1\in\KKK_0$.
Similar arguments then show that if $k_0$ satisfies (V2) or (V3), then there exists $k_1>k_0$ such that $k_1\in\KKK_0$.
This clearly implies that the set $\KKK_0$ is infinite, and completes the proof of Lemma~\ref{lem52}.
\end{proof}

%
%

\section{Proving time-quantitative anti-uniformity}\label{sec8}

Our goal in this section is to establish quite serious violations of uniformity.
More precisely, we establish Theorem~\ref{thm34} which concerns half-infinite $\alpha$-geodesics that start from explicitly given points on the
$2$-square-$b$ surface.

The proof is based on a rather complicated parity formula for certain counting number of the irrational rotation sequence.
To describe this, we need the concept of continued fractions as well as the concept of $\alpha$-representations, or Ostrowski representations, both introduced earlier in this paper.

The rudiments of continued fractions are given in the proof of Lemma~\ref{lem41}, so we give here only a brief summary of what we need.

The irrational slope $\alpha\in(0,1)$ has an infinite continued fraction expansion
\begin{equation}\label{eq8.1}
\alpha=[a_1,a_2,a_3,\ldots]=\frac{1}{a_1+\frac{1}{a_2+\frac{1}{a_3+\cdots}}},
\end{equation}
where $a_i\ge1$, $i=1,2,3,\ldots,$ are integers.
The rational numbers
\begin{displaymath}
\frac{p_k}{q_k}=\frac{p_k(\alpha)}{q_k(\alpha)}=[a_1,\ldots,a_k],
\quad
k=1,2,3,\ldots,
\end{displaymath}
where $p_k\in\Zz$ and $q_k\in\Nn$ are coprime, are the $k$-convergents of~$\alpha$.
Write also
\begin{displaymath}
\eta_k=q_k\alpha-p_k,
\quad
k=1,2,3,\ldots.
\end{displaymath}
We have the recurrence relations
\begin{equation}\label{eq8.2}
p_{k+1}=a_{k+1}p_k+p_{k-1},
\ q_{k+1}=a_{k+1}q_k+q_{k-1},
\ \eta_{k+1}=a_{k+1}\eta_k+\eta_{k-1},
\ k\ge1,
\end{equation}
with initial conditions
\begin{equation}\label{eq8.3}
p_0=0,
\quad
q_0=1,
\quad
\eta_0=\alpha,
\quad
p_1=1,
\quad
q_1=a_1,
\quad
\eta_1=a_1\alpha -1.
\end{equation}
It is well known that the $k$-convergents satisfy \eqref{eq4.34} and give rise to the best rational approximations of~$\alpha$, and
\begin{equation}\label{eq8.4}
\eta_k=(-1)^k\vert\eta_k\vert\left\{\begin{array}{ll}
>0,&\mbox{if $k$ is even},\\
<0,&\mbox{if $k$ is odd}.
\end{array}\right.
\end{equation}
We also have the crucial diophantine approximation property
\begin{equation}\label{eq8.5}
\frac{1}{q_{k+1}+q_k}\le\vert\eta_k\vert=\vert q_k\alpha-p_k\vert=\Vert q_k\alpha\Vert\le\frac{1}{q_{k+1}},
\end{equation}
where $\Vert y\Vert$ denotes the distance of a real number $y$ from the nearest integer.

The concept of $\alpha$-representations, or Ostrowski representations, is first introduced in Section~\ref{sec6}.
Every integer $N\ge1$ has a unique representation in the form
\begin{equation}\label{eq8.6}
N=\sum_{i=0}^kb_iq_i=\sum_{i=0}^kb_i(N)q_i,
\end{equation}
where the integer coefficients $b_0,\ldots,b_k$ satisfy the conditions
\begin{equation}\label{eq8.7}
0\le b_0<a_1,
\quad
0<b_k\le a_{k+1},
\quad
0\le b_i\le a_{i+1},
\quad
i=1,\ldots,k-1,
\end{equation}
as well as the restrictions
\begin{equation}\label{eq8.8}
b_{i-1}=0
\quad\mbox{if}\quad
b_i=a_{i+1},
\quad
i=1,\ldots,k,
\end{equation}
where $a_1,a_2,a_3,\ldots$ are digits of the continued fraction \eqref{eq8.1} of~$\alpha$, and $q_0,q_1,q_2,\ldots$ are the denominators of the successive convergents of~$\alpha$.
Furthermore, the value of the integer $k$ is determined by the inequalities $q_k\le N<q_{k+1}$.
We also say that a sequence $b_0,b_1,\ldots,b_k$ that satisfies \eqref{eq8.6}--\eqref{eq8.8} is $\alpha$-legitimate.

We need one more concept that is not so well known.
If $\alpha$ is irrational, then any real number $\beta\in(-\alpha,1-\alpha)$ can be written in the form
\begin{equation}\label{eq8.9}
\beta=\sum_{i=0}^\infty c_i\eta_i,
\end{equation}
where the integers $c_0,\ldots,c_k$ satisfy the conditions
\begin{equation}\label{eq8.10}
0\le c_0<a_1,
\quad
0\le c_i\le a_{i+1},
\quad
i=1,2,3,\ldots,
\end{equation}
and
\begin{equation}\label{eq8.11}
c_{i-1}=0
\quad\mbox{if}\quad
c_i=a_{i+1},
\quad
i=1,2,3,\ldots,
\end{equation}
where $a_1,a_2,a_3,\ldots$ are digits of the continued fraction \eqref{eq8.1} of~$\alpha$.
Furthermore, if we exclude the case
\begin{equation}\label{eq8.12}
c_{u_0+2i}=a_{u_0+2i+1}
\quad\mbox{for some $u_0$ and all $i\ge0$},
\end{equation}
then the representation \eqref{eq8.9} under the conditions \eqref{eq8.10} and \eqref{eq8.11} is unique.
We call this the $\alpha$-expansion of the real number $\beta\in(-\alpha,1-\alpha)$.

The $\alpha$-expansion of the real number $\beta\in(-\alpha,1-\alpha)$ follows from the density of $n\alpha$ mod~$1$ and the
$\alpha$-representations of positive integers.
Indeed, since $n\alpha$ mod~$1$ is dense in the unit interval, for any real number $\beta\in(-\alpha,1-\alpha)$, there is an infinite sequence of positive integers $1\le n_1<n_2<\ldots<n_r<\ldots$ such that
\begin{equation}\label{eq8.13}
\lim_{r\to\infty}n_r\alpha=\beta\bmod{1}.
\end{equation}
For each integer $n_r$ of this sequence, consider the $\alpha$-representation
\begin{displaymath}
n_r=\sum_{i=0}^{k(r)}b_i(n_r)q_i,
\end{displaymath}
where the digits $b_0(n_r),\ldots,b_{k(r)}(n_r)$ satisfy conditions analogous to \eqref{eq8.7} and \eqref{eq8.8}.
Using \eqref{eq8.2}--\eqref{eq8.4} and \eqref{eq8.7}, we have the upper bound
\begin{align}\label{eq8.14}
\sum_{i=0}^{k(r)}b_i(n_r)\eta_i
&
\le(a_1-1)\eta_0+a_3\eta_2+a_5\eta_4+a_7\eta_6+\ldots
\nonumber
\\
&
=(a_1-1)\eta_0+(\eta_3-\eta_1)+(\eta_5-\eta_3)+(\eta_7-\eta_5)+\ldots
\nonumber
\\
&
=(a_1-1)\eta_0-\eta_1
=(a_1-1)\alpha-(a_1\alpha-1)
=1-\alpha,
\end{align}
and the lower bound
\begin{align}\label{eq8.15}
\sum_{i=0}^{k(r)}b_i(n_r)\eta_i
&
\ge a_2\eta_1+a_4\eta_3+a_6\eta_5+\ldots
\nonumber
\\
&
=(\eta_2-\eta_0)+(\eta_4-\eta_2)+(\eta_6-\eta_4)+\ldots
=-\eta_0
=-\alpha.
\end{align}
Note next that
\begin{equation}\label{eq8.16}
\sum_{i=0}^{k(r)}b_i(n_r)\eta_i
=\sum_{i=0}^{k(r)}b_i(n_r)(q_i\alpha-p_i)
=n_r\alpha-\sum_{i=0}^{k(r)}b_i(n_r)p_i
=n_r\alpha\bmod{1}.
\end{equation}
Since $\beta\in(-\alpha,1-\alpha)$, it now follows on combining \eqref{eq8.13}--\eqref{eq8.16} that
\begin{displaymath}
\lim_{r\to\infty}\sum_{i=0}^{k(r)}b_i(n_r)\eta_i=\beta.
\end{displaymath}
Combining this with a standard compactness argument, we obtain the existence of an $\alpha$-expansion \eqref{eq8.9} with the coefficients satisfying \eqref{eq8.10} and \eqref{eq8.11}.
Indeed, since $0\le b_0(n_r)<a_1$, there exists an infinite set $R_0$ such that $b_0(n_r)$ for every $r\in R_0$ has the same value~$c_0$, say.
Since $0\le b_1(n_r)\le a_2$, there exists an infinite subset $R_1\subset R_0$ such that $b_1(n_r)$ for every $r\in R_1$ has the same value~$c_1$, say.
And so on.
Compactness defines the infinite sequence of coefficients~$c_i$.
On the other hand, the convergence of the series on the right hand side of \eqref{eq8.9} is clear from the bound \eqref{eq8.5} and the exponent growth of the sequence $q_{k+1}$ as shown by \eqref{eq8.2} which gives the estimate $q_{k+1}\ge q_k+q_{k-1}\ge2q_{k-1}$. 

The fact that \eqref{eq8.12} guarantees uniqueness of the $\alpha$-expansion is left to the reader as an exercise.

We shall be concerned with real numbers $\beta$ satisfying $0<\beta<1-\alpha$.
It is well known that for any $\beta$ with $\alpha$-representation \eqref{eq8.9}, we have $0<\beta<1-\alpha$ if and only if
\begin{equation}\label{eq8.17}
\min\{i=0,1,2,3,\ldots:c_i\ge1\}
\mbox{ is even}.
\end{equation}
Let $\alpha\in(0,1)$ be a fixed irrational number, and let $N\ge1$ be an integer.
For any non-zero real number $\beta$ satisfying $0<\beta<1-\alpha$, let
\begin{equation}\label{eq8.18}
\Phi(\alpha;\beta;N)=\vert\{q=0,\ldots,N-1:\{q\alpha\}\in[0,\beta)\}\vert.
\end{equation}

The next result gives a fairly complicated parity formula for the difference of two counting numbers $\Phi(\alpha;\beta';N)$ and $\Phi(\alpha;\beta'';N)$ in terms of the continued fraction of~$\alpha$, the $\alpha$-representation of $N$ and the $\alpha$-expansions of $\beta'$ and~$\beta''$ under some very special circumstances.

\begin{lem}\label{lem81}
Suppose that $\alpha\in(0,1)$ is a fixed irrational number, and that the integer $N$ satisfies $1\le N<q_{k+1}$.
Suppose further that
\begin{equation}\label{eq8.19}
N=\sum_{i=0}^kb_iq_i=\sum_{i=0}^kb_i(N)q_i
\end{equation}
denotes the $\alpha$-representation of an integer $N\ge1$, and that
\begin{equation}\label{eq8.20}
\beta'=\sum_{i=0}^\infty c'_i\eta_i
\quad\mbox{and}\quad
\beta''=\sum_{i=0}^\infty c''_i\eta_i
\end{equation}
denote respectively the $\alpha$-expansions of two real numbers $\beta',\beta''\in(0,1-\alpha)$, where the digits $c'_i$ and $c''_i$ are all even, and where $c'_i$ and $c''_i$ are non-zero whenever $i$ is even.
For every integer $j=0,\ldots,k$, let
\begin{equation}\label{eq8.21}
N_j=N_j(N)=\sum_{i=0}^jb_iq_i
\end{equation}
denote an integer defined in terms of the $\alpha$-representation of~$N$, and let
\begin{equation}\label{eq8.22}
C'_j=C'_j(N)=\sum_{i=0}^jc'_iq_i
\quad\mbox{and}\quad
C''_j=C''_j(N)=\sum_{i=0}^jc''_iq_i
\end{equation}
denote respectively integers defined in terms of the $\alpha$-expansions of $\beta'$ and~$\beta''$.
For every integer $\ell=1,\ldots,k$, let
\begin{equation}\label{eq8.23}
\Delta'_\ell
=\Delta'_\ell(N)
=\left\{\begin{array}{ll}
1,&\mbox{if $\ell$ is even and $C'_{\ell-1}<N_{\ell-1}\le N_\ell<C'_\ell$},\\
-1,&\mbox{if $\ell$ is odd and $N_{\ell-1}\le C'_{\ell-1}\le C'_\ell<N_\ell$},\\
0,&\mbox{otherwise},
\end{array}\right.
\end{equation}
and
\begin{equation}\label{eq8.24}
\Delta''_\ell
=\Delta''_\ell(N)
=\left\{\begin{array}{ll}
1,&\mbox{if $\ell$ is even and $C''_{\ell-1}<N_{\ell-1}\le N_\ell<C''_\ell$},\\
-1,&\mbox{if $\ell$ is odd and $N_{\ell-1}\le C''_{\ell-1}\le C''_\ell<N_\ell$},\\
0,&\mbox{otherwise}.
\end{array}\right.
\end{equation}
Then, provided that the coefficients $b_1,\ldots,b_k$ in \eqref{eq8.19} and \eqref{eq8.21} satisfy
\begin{equation}\label{eq8.25}
b_i<a_{i+1},
\quad
i=1,\ldots,k,
\end{equation}
we have
\begin{align}\label{eq8.26}
&
\Phi(\alpha;\beta'';N)-\Phi(\alpha;\beta';N)
\nonumber
\\
&\quad
=\sum_{\ell=0}^k\min\{b_\ell,c'_\ell\}
+\sum_{\ell=0}^k\min\{b_\ell,c''_\ell\}
+\sum_{\ell=1}^k\Delta'_\ell
+\sum_{\ell=1}^k\Delta''_\ell
\bmod{2}.
\end{align}
\end{lem}

We shall prove Lemma~\ref{lem81} in Section~\ref{sec9}.
There the reader will see that a parity formula for a single counting number $\Phi(\alpha;\beta;N)$ contains a \textit{translation} term which we are not able to handle.
Thus by studying the difference of two counting numbers, this term appears twice and therefore cancel each other modulo~$2$.

\begin{remark}
The counting number $\Phi(\alpha;\beta;N)$ is related to the discrepancy function
\begin{displaymath}
D(\alpha;\beta;N)=\vert\{q=0,\ldots,N-1:\{q\alpha\}\in[0,\beta)\}\vert-N\beta
\end{displaymath}
for which there is an explicit formula due to S\'{o}s~\cite{So3}.
In fact, one can derive our parity formula using the ideas of S\'{o}s.
However, it would be most unkind to ask the reader to work out the details.
Instead, we include in the next section a detailed proof by closely following the method of S\'{o}s.
\end{remark}

The conditions in \eqref{eq8.23} and \eqref{eq8.24} may look elegant, but as they stand, they are not of much use.
For applications later, we need a less elegant but more convenient form.
We summarize it below.
The proof is almost trivial.

\begin{lem}\label{lem82}
Suppose that for every integer $j=0,\ldots,k$, the integer $C_j$ is defined in terms of \eqref{eq8.9} in precisely the same way as the integers $C'_j$ and $C''_j$ are defined by \eqref{eq8.22} in terms of the $\alpha$-expansions \eqref{eq8.20} of $\beta'$ and $\beta''$ respectively.

\emph{(i)}
The condition $C_{\ell-1}<N_{\ell-1}\le N_\ell<C_\ell$ is equivalent to
\begin{equation}\label{eq8.27}
b_\ell<c_\ell,
\end{equation}
together with the existence of an integer $m<\ell$ such that
\begin{equation}\label{eq8.28}
c_m<b_m
\quad\mbox{and}\quad
c_i=b_i,
\quad
m<i<\ell.
\end{equation}

\emph{(ii)}
The condition $N_{\ell-1}\le C_{\ell-1}\le C_\ell<N_\ell$ is equivalent to
\begin{equation}\label{eq8.29}
c_\ell<b_\ell,
\end{equation}
together with either
\begin{equation}\label{eq8.30}
b_i=c_i,
\quad
i<\ell,
\end{equation}
or the existence of an integer $m<\ell$ such that
\begin{equation}\label{eq8.31}
b_m<c_m
\quad\mbox{and}\quad
b_i=c_i,
\quad
m<i<\ell.
\end{equation}
\end{lem}

To prove Theorem~\ref{thm34}, the simple basic idea is to use discretization to convert the continuous problem of the distribution of an $\alpha$-geodesic on the $2$-square-$b$ surface to the discrete problem of the distribution of the sequence of points at which the $\alpha$-geodesic hits a vertical edge of the surface.
The latter gives the irrational rotation sequence $\{q\alpha\}$, $q\ge0$.
The question of left or right square clearly leads to a \textit{parity problem}, where left or right is converted to even or odd.

Let an irrational number $\alpha\in(0,1)$ be given and fixed.
Let $N\ge1$ be an arbitrary integer, and consider the unique $\alpha$-representation of $N$ as given by \eqref{eq8.6}--\eqref{eq8.8}.

Next, we use the unique $\alpha$-expansion of real numbers, and define the length $\beta'$ of a first gate in terms of the $\alpha$-expansion
\begin{equation}\label{eq8.32}
\beta'=\sum_{i=0}^\infty c'_i\eta_i,
\quad
c'_i=\left\{\begin{array}{ll}
2,&\mbox{if $i$ is even},\\
0\mbox{ or }2,&\mbox{if $i$ is odd},
\end{array}\right.
\end{equation}
and the length $\beta''$ of a second gate in terms of the $\alpha$-expansion
\begin{equation}\label{eq8.33}
\beta''=\sum_{i=0}^\infty c''_i\eta_i,
\quad
c''_i=\left\{\begin{array}{ll}
4,&\mbox{if $i$ is even},\\
0,&\mbox{if $i$ is odd}.
\end{array}\right.
\end{equation}
Clearly the condition \eqref{eq8.17} is satisfied by both $\beta'$ and~$\beta''$, so $0<\beta',\beta''<1-\alpha$.
In fact, we have
\begin{equation}\label{eq8.34}
0<\beta'<\beta''<1-\alpha.
\end{equation}

\begin{lem}\label{lem83}
Suppose that the irrational number $\alpha\in(0,1)$ satisfies \eqref{eq3.1}, and that the lengths $\beta'$ and $\beta''$ of the gates satisfy \eqref{eq8.32} and \eqref{eq8.33}.
For every integer $k\ge0$, consider the set
\begin{displaymath}
\BBB(k)=\{0,1,\ldots,q_{k+1}-1\}.
\end{displaymath}
Then we have the lower bound
\begin{equation}\label{eq8.35}
\vert\{N\in\BBB(k):\parity(\Phi(\alpha;\beta'';N)-\Phi(\alpha;\beta';N))=0\}\vert
>(1-\eps)q_{k+1},
\end{equation}
provided that $\eps>0$ is sufficiently small.
\end{lem}

\begin{proof}
The condition \eqref{eq3.1} clearly guarantees that $a_i\ge6$, $i\ge1$, so that our choices of $\beta'$ and $\beta''$ in \eqref{eq8.32} and \eqref{eq8.33} are valid.

For each element $N\in\BBB(k)$, we can write
\begin{displaymath}
N=\sum_{i=0}^kb_i(N)q_i,
\end{displaymath}
where, for each $i=0,\ldots,k$, the coefficient $b_i=b_i(N)$ satisfies \eqref{eq8.7} and \eqref{eq8.8}, and with the convention that $b_0(0)=\ldots=b_k(0)=0$.
In this way, we see that the set $\BBB(k)$ is in one-to-one correspondence with the collection of $\alpha$-legitimate sequences $b_0,b_1,\ldots,b_k$ together with the trivial sequence $0,\ldots,0$.

Suppose that for an integer $N\in\BBB(k)$, we have
\begin{equation}\label{eq8.36}
5\le b_\ell<a_{\ell+1},
\quad
\ell=1,\ldots,k.
\end{equation}
In view of \eqref{eq8.32}, the first sum in \eqref{eq8.26} modulo~$2$ is equal to
\begin{align}\label{eq8.37}
\sum_{\ell=0}^k\min\{b_\ell,c'_\ell\}
&
=\min\{b_0,c'_0\}+\sum_{\ell=1}^k\min\{b_\ell,c'_\ell\}
\nonumber
\\
&
=\min\{b_0,2\}+\sum_{\ell=1}^kc'_\ell
=\left\{\begin{array}{ll}
1,&\mbox{if $b_0=1$},\\
0,&\mbox{if $b_0\ne1$}.
\end{array}\right.
\end{align}
In view of \eqref{eq8.33}, the second sum in \eqref{eq8.26} modulo~$2$ is equal to
\begin{align}\label{eq8.38}
\sum_{\ell=0}^k\min\{b_\ell,c''_\ell\}
&
=\min\{b_0,c''_0\}+\sum_{\ell=1}^k\min\{b_\ell,c''_\ell\}
\nonumber
\\
&
=\min\{b_0,4\}+\sum_{\ell=1}^kc''_\ell
=\left\{\begin{array}{ll}
1,&\mbox{if $b_0=1,3$},\\
0,&\mbox{if $b_0\ne1,3$}.
\end{array}\right.
\end{align}
For the third sum in \eqref{eq8.26}, note that \eqref{eq8.27} does not hold with $c'_\ell=2$, so it follows from \eqref{eq8.23} and Lemma~\ref{lem82}(i) that $\Delta'_\ell=0$ for even $\ell\ge2$.
Also \eqref{eq8.30} and \eqref{eq8.31} do not hold with $c_i=0$ or $c_i=2$, so it follows from \eqref{eq8.23} and Lemma~\ref{lem82}(ii) that
$\Delta'_\ell=0$ for odd $\ell\ge3$.
Thus the third sum in \eqref{eq8.26} is equal to
\begin{equation}\label{eq8.39}
\sum_{\ell=1}^k\Delta'_\ell=\Delta'_1.
\end{equation}
For $\ell=1$, it is clear that \eqref{eq8.29} holds.
For $b_0=0,1$, it is clear that \eqref{eq8.31} holds.
For $b_0=2$, it is clear that \eqref{eq8.30} holds.
For $b_0\ge3$, it is clear that neither \eqref{eq8.30} nor \eqref{eq8.31} holds.
Thus it follows from \eqref{eq8.23} and Lemma~\ref{lem82}(ii) that
\begin{equation}\label{eq8.40}
\Delta'_1=\left\{\begin{array}{ll}
-1,&\mbox{if $b_0=0,1,2$},\\
0,&\mbox{if $b_0\ge3$}.
\end{array}\right.
\end{equation}
For the fourth sum in \eqref{eq8.26}, note that \eqref{eq8.27} does not hold with $c''_\ell=4$, so it follows from \eqref{eq8.24} and Lemma~\ref{lem82}(i) that $\Delta''_\ell=0$ for even $\ell\ge2$.
Also \eqref{eq8.30} and \eqref{eq8.31} do not hold with $c_i=0$ or $c_i=4$, so it follows from \eqref{eq8.23} and Lemma~\ref{lem82}(ii) that
$\Delta'_\ell=0$ for odd $\ell\ge3$.
Thus the fourth sum in \eqref{eq8.26} is equal to
\begin{equation}\label{eq8.41}
\sum_{\ell=1}^k\Delta''_\ell=\Delta''_1.
\end{equation}
For $\ell=1$, it is clear that \eqref{eq8.29} holds.
For $b_0=0,1,2,3$, it is clear that \eqref{eq8.31} holds.
For $b_0=4$, it is clear that \eqref{eq8.30} holds.
For $b_0\ge5$, it is clear that neither \eqref{eq8.30} nor \eqref{eq8.31} holds.
Thus it follows from \eqref{eq8.24} and Lemma~\ref{lem82}(ii) that
\begin{equation}\label{eq8.42}
\Delta''_1=\left\{\begin{array}{ll}
-1,&\mbox{if $b_0=0,1,2,3,4$},\\
0,&\mbox{if $b_0\ge5$}.
\end{array}\right.
\end{equation}
Hence if \eqref{eq8.36} holds, then it follows from \eqref{eq8.26} and \eqref{eq8.37}--\eqref{eq8.42} that
\begin{equation}\label{eq8.43}
\parity(\Phi(\alpha;\beta'';N)-\Phi(\alpha;\beta';N))
=\left\{\begin{array}{ll}
1,&\mbox{if $b_0=4$},\\
0,&\mbox{if $b_0\ne4$}.
\end{array}\right.
\end{equation}

We have the trivial bounds
\begin{equation}\label{eq8.44}
\prod_{i=1}^{k+1}(a_i-1)
\le\vert\BBB(k)\vert
=q_{k+1}
\le\prod_{i=1}^{k+1}(a_i+1).
\end{equation}
Using the condition \eqref{eq3.1}, we have
\begin{align}\label{eq8.45}
\prod_{i=1}^{k+1}(a_i+1)
&
\le\prod_{i=1}^\infty\left(1+\frac{2}{a_i-1}\right)\prod_{i=1}^{k+1}(a_i-1)
\le\prod_{i=1}^\infty\left(1+\frac{3}{a_i}\right)\prod_{i=1}^{k+1}(a_i-1)
\nonumber
\\
&
\le\exp\left(\sum_{i=1}^\infty\frac{3}{a_i}\right)\prod_{i=1}^{k+1}(a_i-1)
\le\ee^{\eps/100}\prod_{i=1}^{k+1}(a_i-1),
\end{align}
where $\exp(x)=\ee^x$ is the exponential function.
Thus \eqref{eq8.44} and \eqref{eq8.45} give
\begin{equation}\label{eq8.46}
\prod_{i=1}^{k+1}(a_i-1)
\le\vert\BBB(k)\vert
=q_{k+1}
\le\ee^{\eps/100}\prod_{i=1}^{k+1}(a_i-1).
\end{equation}

We wish to find a lower bound for the cardinality of the set
\begin{displaymath}
\{N\in\BBB(k):\parity(\Phi(\alpha;\beta'';N)-\Phi(\alpha;\beta';N))=0\}.
\end{displaymath}
Observe from \eqref{eq8.43} that $\parity(\Phi(\alpha;\beta'';N)-\Phi(\alpha;\beta';N))=0$ except possibly when
\begin{displaymath}
b_0(N)=4,
\end{displaymath}
or \eqref{eq8.36} fails, so that
\begin{displaymath}
b_\ell(N)\in J_\ell=\{0,1,2,3,4,a_{\ell+1}\}
\quad
\mbox{for some $\ell=1,\ldots,k$}.
\end{displaymath}
Accordingly, we need to find a lower bound for the cardinality of the set
\begin{displaymath}
\BBB(k;0)=\{N\in\BBB(k):b_0(N)\ne4\},
\end{displaymath}
as well as upper bounds for the cardinality of each of the sets
\begin{displaymath}
\BBB(k;\ell;j)=\{N\in\BBB(k):b_\ell(N)=j\},
\quad
\ell=1,\ldots,k,
\quad
j\in J_\ell,
\end{displaymath}
and combine these with the inequality
\begin{align}\label{eq8.47}
&
\vert\{N\in\BBB(k):\parity(\Phi(\alpha;\beta''_0;N)-\Phi(\alpha;\beta';N))=0\}\vert
\nonumber
\\
&\quad
\ge\vert\BBB(k;0)\vert-\sum_{\ell=1}^k\sum_{j\in J_\ell}\vert\BBB(k;\ell;j)\vert.
\end{align}
Indeed, we have the trivial lower bound
\begin{displaymath}
\vert\BBB(k;0)\vert\ge(a_1-1)\prod_{i=2}^{k+1}(a_i-1)=\prod_{i=1}^{k+1}(a_i-1).
\end{displaymath}
Combining this with \eqref{eq3.1} and \eqref{eq8.46}, we obtain
\begin{equation}\label{eq8.48}
\vert\BBB(k;0)\vert\ge\ee^{-\eps/100}q_{k+1}.
\end{equation}
Also, for each $\ell=1,\ldots,k$ and $j\in J_\ell$, we have the trivial upper bound
\begin{align}
\vert\BBB(k;\ell;j)\vert
&
\le\prod_{\substack{{i=1}\\{i\ne\ell+1}}}^{k+1}(a_i+1)
\le\prod_{i=1}^\infty\left(1+\frac{3}{a_i}\right)\prod_{\substack{{i=1}\\{i\ne\ell+1}}}^{k+1}(a_i-1)
\nonumber
\\
&
\le\exp\left(\sum_{i=1}^\infty\frac{3}{a_i}\right)\prod_{\substack{{i=1}\\{i\ne\ell+1}}}^{k+1}(a_i-1)
\le\frac{2}{a_{\ell+1}}\exp\left(\sum_{i=1}^\infty\frac{3}{a_i}\right)\prod_{i=1}^{k+1}(a_i-1).
\nonumber
\end{align}
Combining this with \eqref{eq3.1} and \eqref{eq8.46}, we obtain
\begin{equation}\label{eq8.49}
\sum_{\ell=1}^k\sum_{j\in J_\ell}\vert\BBB(k;\ell;j)\vert
\le6\ee^{\eps/100}\left(\sum_{\ell=1}^k\frac{2}{a_{\ell+1}}\right)q_{k+1}
<\frac{\ee^{\eps/100}\eps}{25}q_{k+1}.
\end{equation}
Finally, combining \eqref{eq8.47}--\eqref{eq8.49}, we obtain the desired lower bound
\begin{align}
&
\vert\{N\in\BBB(k):\parity(\Phi(\alpha;\beta'';N)-\Phi(\alpha;\beta';N))=0\}\vert
\nonumber
\\
&\quad
>\left(\ee^{-\eps/100}-\frac{\ee^{\eps/100}\eps}{25}\right)q_{k+1}
>(1-\eps)q_{k+1},
\nonumber
\end{align}
provided that $\eps>0$ is sufficiently small.
\end{proof}

Next, for every $b=1,\ldots,a_{k+2}-1$, we consider the sets
\begin{displaymath}
\BBB^*(k;b)=\{bq_{k+1},bq_{k+1}+1,\ldots,(b+1)q_{k+1}-1\}=bq_{k+1}+\BBB(k),
\end{displaymath}
where every element $N\in\BBB^*(k;b)$ can be written in the form
\begin{displaymath}
N=bq_{k+1}+\sum_{i=0}^kb_i(N)q_i=q_{k+1}+\sum_{i=0}^kb_i(N-bq_{k+1})q_i,
\end{displaymath}
and the coefficients $b_i=b_i(N)=b_i(N-bq_{k+1})$ satisfy \eqref{eq8.7} and \eqref{eq8.8}.
For every element $N\in\BBB^*(k)$, the analog of \eqref{eq8.26} is
\begin{align}\label{eq8.50}
&
\Phi(\alpha;\beta'';N)-\Phi(\alpha;\beta';N)
\nonumber
\\
&\quad
=\sum_{\ell=0}^{k+1}\min\{b_\ell,c'_\ell\}
+\sum_{\ell=0}^{k+1}\min\{b_\ell,c''_\ell\}
+\sum_{\ell=1}^{k+1}\Delta'_\ell
+\sum_{\ell=1}^{k+1}\Delta''_\ell,
\end{align}
where $b_{k+1}=b$.

\begin{lem}\label{lem84}
Under the hypotheses of Lemma~\ref{lem83}, suppose further that $k$ is odd.
Then for every $b=1,\ldots,a_{k+2}-1$ apart from $b=2$, we have the lower bound
\begin{equation}\label{eq8.51}
\vert\{N\in\BBB^*(k;b):\parity(\Phi(\alpha;\beta'';N)-\Phi(\alpha;\beta';N))=0\}\vert
>(1-\eps)q_{k+1},
\end{equation}
provided that $\eps>0$ is sufficiently small.
Furthermore, we have the lower bound
\begin{equation}\label{eq8.52}
\vert\{N\in\BBB^*(k;2):\parity(\Phi(\alpha;\beta'';N)-\Phi(\alpha;\beta';N))=1\}\vert
>(1-\eps)q_{k+1},
\end{equation}
provided that $\eps>0$ is sufficiently small.
\end{lem}

\begin{proof}
Clearly $k+1$ is even, so it follows from \eqref{eq8.32} and \eqref{eq8.33} that $c'_{k+1}=2$ and $c''_{k+1}=4$.
Then the first sum in \eqref{eq8.50}, compared modulo~$2$ to the corresponding sum in \eqref{eq8.26}, has an extra term
\begin{equation}\label{eq8.53}
\min\{b_{k+1},c'_{k+1}\}
=\min\{b,2\}
=\left\{\begin{array}{ll}
1,&\mbox{if $b=1$},\\
0,&\mbox{if $b\ne1$}.
\end{array}\right.
\end{equation}
The second sum in \eqref{eq8.50}, compared modulo~$2$ to the corresponding sum in \eqref{eq8.26}, has an extra term
\begin{equation}\label{eq8.54}
\min\{b_{k+1},c''_{k+1}\}
=\min\{b,4\}
=\left\{\begin{array}{ll}
1,&\mbox{if $b=1,3$},\\
0,&\mbox{if $b\ne1,3$}.
\end{array}\right.
\end{equation}
The third sum in \eqref{eq8.50}, compared to the corresponding sum in \eqref{eq8.26}, has an extra term $\Delta'_{k+1}$.
If $b=1$, then \eqref{eq8.27} and \eqref{eq8.28} hold with $\ell=k+1$ and $m=k$, so it follows from \eqref{eq8.23} and Lemma~\ref{lem82}(ii) that
$\Delta'_{k+1}=1$.
If $b\ne1$, then \eqref{eq8.27} does not hold with $\ell=k+1$, so it follows from \eqref{eq8.23} and Lemma~\ref{lem82}(ii) that $\Delta'_{k+1}=0$.
Thus
\begin{equation}\label{eq8.55}
\Delta'_{k+1}=\left\{\begin{array}{ll}
1,&\mbox{if $b=1$},\\
0,&\mbox{if $b\ne1$}.
\end{array}\right.
\end{equation}
The fourth sum in \eqref{eq8.50}, compared to the corresponding sum in \eqref{eq8.26}, has an extra term $\Delta''_{k+1}$.
If $b=1,2,3$, then \eqref{eq8.27} and \eqref{eq8.28} hold with $\ell=k+1$ and $m=k$, so it follows from \eqref{eq8.24} and Lemma~\ref{lem82}(ii) that
$\Delta'_{k+1}=1$.
If $b\ne1,2,3$, then \eqref{eq8.27} does not hold with $\ell=k+1$, so it follows from \eqref{eq8.24} and Lemma~\ref{lem82}(ii) that $\Delta''_{k+1}=0$.
Thus
\begin{equation}\label{eq8.56}
\Delta''_{k+1}=\left\{\begin{array}{ll}
1,&\mbox{if $b=1,2,3$},\\
0,&\mbox{if $b\ne1,2,3$}.
\end{array}\right.
\end{equation}
We now combine \eqref{eq8.53}--\eqref{eq8.56}.
If $b\ne2$, then compared to \eqref{eq8.26}, there is no change in parity, so that if \eqref{eq8.36} holds, then
\begin{displaymath}
\parity(\Phi(\alpha;\beta'';N)-\Phi(\alpha;\beta';N))=\left\{\begin{array}{ll}
1,&\mbox{if $b_0=4$},\\
0,&\mbox{if $b_0\ne4$},
\end{array}\right.
\end{displaymath}
resulting in a lower bound \eqref{eq8.51}, provided that $\eps>0$ is sufficiently small.
If $b=2$, then compared to \eqref{eq8.26}, there is a change in parity, so that if \eqref{eq8.36} holds, then
\begin{displaymath}
\parity(\Phi(\alpha;\beta'';N)-\Phi(\alpha;\beta';N))=\left\{\begin{array}{ll}
0,&\mbox{if $b_0=4$},\\
1,&\mbox{if $b_0\ne4$},
\end{array}\right.
\end{displaymath}
resulting in a lower bound \eqref{eq8.52}, provided that $\eps>0$ is sufficiently small.
\end{proof}

\begin{lem}\label{lem85}
Under the hypotheses of Lemma~\ref{lem83}, suppose further that $k$ is even.

If $c'_{k+1}=0$, then for every $b=1,\ldots,a_{k+2}-1$, we have the lower bound
\begin{equation}\label{eq8.57}
\vert\{N\in\BBB^*(k;b):\parity(\Phi(\alpha;\beta'';N)-\Phi(\alpha;\beta';N))=0\}\vert
>(1-\eps)q_{k+1},
\end{equation}
provided that $\eps>0$ is sufficiently small.

If $c'_{k+1}=2$, then for every $b=1,\ldots,a_{k+2}-1$ apart from $b=1$, we have the lower bound
\begin{equation}\label{eq8.58}
\vert\{N\in\BBB^*(k;b):\parity(\Phi(\alpha;\beta'';N)-\Phi(\alpha;\beta';N))=0\}\vert
>(1-\eps)q_{k+1},
\end{equation}
provided that $\eps>0$ is sufficiently small.
Furthermore. we have the lower bound
\begin{equation}\label{eq8.59}
\vert\{N\in\BBB^*(k;1):\parity(\Phi(\alpha;\beta'';N)-\Phi(\alpha;\beta';N))=1\}\vert
>(1-\eps)q_{k+1},
\end{equation}
provided that $\eps>0$ is sufficiently small.
\end{lem}

\begin{proof}
Clearly $k+1$ is odd, and it follows from \eqref{eq8.33} that $c''_{k+1}=0$.
The second sum in \eqref{eq8.50}, compared to the corresponding sum in \eqref{eq8.26}, has an extra term
\begin{equation}\label{eq8.60}
\min\{b_{k+1},c''_{k+1}\}
=\min\{b,0\}
=0.
\end{equation}
The fourth sum in \eqref{eq8.50}, compared to the corresponding sum in \eqref{eq8.26}, has an extra term $\Delta''_{k+1}$.
But then \eqref{eq8.30} and \eqref{eq8.31} do not hold with $\ell=k+1$, so it follows from \eqref{eq8.24} and Lemma~\ref{lem82}(ii) that
\begin{equation}\label{eq8.61}
\Delta''_{k+1}=0.
\end{equation}

Suppose that $c'_{k+1}=0$.
The first sum in \eqref{eq8.50}, compared to the corresponding sum in \eqref{eq8.26}, has an extra term
\begin{equation}\label{eq8.62}
\min\{b_{k+1},c'_{k+1}\}
=\min\{b,0\}
=0.
\end{equation}
The third sum in \eqref{eq8.50}, compared to the corresponding sum in \eqref{eq8.26}, has an extra term $\Delta'_{k+1}$.
Then \eqref{eq8.30} and \eqref{eq8.31} do not hold with $\ell=k+1$, so it follows from \eqref{eq8.23} and Lemma~\ref{lem82}(ii) that
\begin{equation}\label{eq8.63}
\Delta'_{k+1}=0.
\end{equation}
We now combine \eqref{eq8.60}--\eqref{eq8.63}.
Compared to \eqref{eq8.26}, there is no change in parity, so that if \eqref{eq8.36} holds, then
\begin{displaymath}
\parity(\Phi(\alpha;\beta'';N)-\Phi(\alpha;\beta';N))=\left\{\begin{array}{ll}
1,&\mbox{if $b_0=4$},\\
0,&\mbox{if $b_0\ne4$},
\end{array}\right.
\end{displaymath}
resulting in a lower bound \eqref{eq8.57}, provided that $\eps>0$ is sufficiently small.

Suppose that $c'_{k+1}=2$.
Then the first sum in \eqref{eq8.50}, compared modulo~$2$ to the corresponding sum in \eqref{eq8.26}, has an extra term
\begin{equation}\label{eq8.64}
\min\{b_{k+1},c'_{k+1}\}
=\min\{1,2\}
=\left\{\begin{array}{ll}
1,&\mbox{if $b=1$},\\
0,&\mbox{if $b\ne1$}.
\end{array}\right.
\end{equation}
The third sum in \eqref{eq8.50}, compared to the corresponding sum in \eqref{eq8.26}, has an extra term $\Delta'_{k+1}$.
Then \eqref{eq8.30} and \eqref{eq8.31} do not hold with $\ell=k+1$, so it follows from \eqref{eq8.23} and Lemma~\ref{lem82}(ii) that
\begin{equation}\label{eq8.65}
\Delta'_{k+1}=0.
\end{equation}
We now combine \eqref{eq8.60}, \eqref{eq8.61}, \eqref{eq8.64} and \eqref{eq8.65}.
If $b\ne1$, then compared to \eqref{eq8.26}, there is no change in parity, so that if \eqref{eq8.36} holds, then
\begin{displaymath}
\parity(\Phi(\alpha;\beta'';N)-\Phi(\alpha;\beta';N))=\left\{\begin{array}{ll}
1,&\mbox{if $b_0=4$},\\
0,&\mbox{if $b_0\ne4$},
\end{array}\right.
\end{displaymath}
resulting in a lower bound \eqref{eq8.58}, provided that $\eps>0$ is sufficiently small.
If $b=1$, then compared to \eqref{eq8.26}, there is a change in parity, so that if \eqref{eq8.36} holds, then
\begin{displaymath}
\parity(\Phi(\alpha;\beta'';N)-\Phi(\alpha;\beta';N))=\left\{\begin{array}{ll}
0,&\mbox{if $b_0=4$},\\
1,&\mbox{if $b_0\ne4$},
\end{array}\right.
\end{displaymath}
resulting in a lower bound \eqref{eq8.59}, provided that $\eps>0$ is sufficiently small.
\end{proof}

\begin{proof}[Proof of Theorem~\ref{thm34}]
In view of the inequalities \eqref{eq8.34}, note that the difference
\begin{equation}\label{eq8.66}
\Phi(\alpha;\beta'';N)-\Phi(\alpha;\beta';N)
\end{equation}
corresponds to a surface with a gate $[\beta',\beta'')$, as shown in the picture on the left in Figure~8.1, and an $\alpha$-geodesic that starts from the bottom of the left vertical edge of the surface.
With the horizontal edge identifications shown, it is not difficult to see that this is equivalent to the $2$-square-$(\beta''-\beta')$ surface shown in the picture on the right in Figure~8.1, and an $\alpha$-geodesic that starts at a point on the left vertical edge at a distance $\beta'$ from the top left vertex.

\begin{displaymath}
\begin{array}{c}
\includegraphics{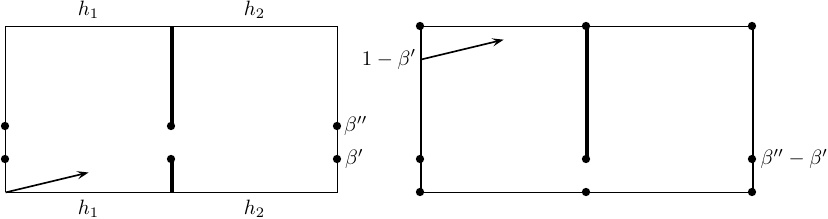}
\\
\mbox{Figure 8.1: two equivalent surfaces}
\end{array}
\end{displaymath}

Note also that when we the integer parameter $N$ in \eqref{eq8.66} progresses by~$1$, this corresponds to the $\alpha$-geodesic travelling from one vertical edge of the $2$-square-$b$ surface to the next vertical edge, and the length of this geodesic segment is clearly $(1+\alpha^2)^{1/2}$.

The length $\beta_0=\beta''_0-\beta'_0$ of the gate, given in terms of the $\alpha$-expansions
\begin{displaymath}
\beta'_0=\sum_{i=0}^\infty c'_i\eta_i,
\quad
c'_i=2,
\end{displaymath}
and
\begin{displaymath}
\beta''_0=\sum_{i=0}^\infty c''_i\eta_i,
\quad
c''_i=\left\{\begin{array}{ll}
4,&\mbox{if $i$ is even},\\
0,&\mbox{if $i$ is odd},
\end{array}\right.
\end{displaymath}
satisfies \eqref{eq8.32} and \eqref{eq8.33}, so that Lemmas \ref{lem83}--\ref{lem85} are valid.

(i)
For every positive integer~$n$, let
\begin{displaymath}
T^*_n=(1+\alpha^2)^{1/2}q_{2n+1}.
\end{displaymath}
Then \eqref{eq3.2} follows from \eqref{eq8.35} for $b=0$ and from \eqref{eq8.58} for $b=2,\ldots,C$.
Meanwhile, \eqref{eq3.3} follows from \eqref{eq8.59}.

(ii)
For every positive integer~$n$, let
\begin{displaymath}
T^{**}_n=(1+\alpha^2)^{1/2}q_{2n}.
\end{displaymath}
Then \eqref{eq3.4} follows from \eqref{eq8.35} for $b=0$ and from \eqref{eq8.51} for $b=1$ or $b=3,\ldots,C$.
Meanwhile, \eqref{eq3.5} follows from \eqref{eq8.52}.

The length $\beta_1=\beta''_1-\beta'_1$ of the gate, given in terms of the $\alpha$-expansions
\begin{displaymath}
\beta'_1=\sum_{i=0}^\infty c'_i\eta_i,
\quad
c'_i=\left\{\begin{array}{ll}
2,&\mbox{if $i$ is even},\\
2,&\mbox{if $i$ is odd and $i<2n+2$},\\
0,&\mbox{if $i$ is odd and $i>2n+2$},
\end{array}\right.
\end{displaymath}
and
\begin{displaymath}
\beta''_1=\sum_{i=0}^\infty c''_i\eta_i,
\quad
c''_i=\left\{\begin{array}{ll}
4,&\mbox{if $i$ is even},\\
0,&\mbox{if $i$ is odd},
\end{array}\right.
\end{displaymath}
satisfies \eqref{eq8.32} and \eqref{eq8.33}, so that Lemmas \ref{lem83}--\ref{lem85} are valid.

(iii)
For every positive integer $i=1,\ldots,n$, let
\begin{displaymath}
W_i=(1+\alpha^2)^{1/2}q_{2n+1},
\end{displaymath}
as in Part (i).
Then \eqref{eq3.6} follows from \eqref{eq8.35}, and \eqref{eq3.7} follows from \eqref{eq8.59}.

(iv)
Let $Q^\star=q_{2n+3}$.
For any integer $Q>Q^\star$, there clearly exists a unique integer $k\ge2n+2$ such that $q_{k+1}\le Q<q_{k+2}$.
Furthermore, either
\begin{equation}\label{eq8.67}
Q\in\BBB^*(k;b^*)
\quad
\mbox{for some $b^*=1,\ldots,a_{k+2}-1$},
\end{equation}
or $Q$ is almost as large as $q_{2k+2}$, in the sense that
\begin{equation}\label{eq8.68}
Q\in[a_{k+2}q_{k+1},q_{k+2}).
\end{equation}
Suppose first that \eqref{eq8.67} holds.
Then
\begin{displaymath}
\{0,1,\ldots,Q-1\}=\BBB(k)\cup\bigcup_{b=1}^{b^*-1}\BBB^*(k;b)\cup\{b^*q_{k+1},\ldots,Q-1\},
\end{displaymath}
so that
\begin{align}\label{eq8.69}
&
\vert\{N\in[0,Q):\parity(\Phi(\alpha;\beta'';N)-\Phi(\alpha;\beta';N))=0\}\vert
\nonumber
\\
&\quad
\ge\vert\{N\in\BBB(k):\parity(\Phi(\alpha;\beta'';N)-\Phi(\alpha;\beta';N))=0\}\vert
\nonumber
\\
&\quad\quad
+\sum_{b=1}^{b^*-1}\vert\{N\in\BBB^*(k;b):\parity(\Phi(\alpha;\beta'';N)-\Phi(\alpha;\beta';N))=0\}\vert
\nonumber
\\
&\quad\quad
+\vert\{N\in[b^*q_{k+1},Q):\parity(\Phi(\alpha;\beta'';N)-\Phi(\alpha;\beta';N))=0\}\vert.
\end{align}
If $k$ is even, then it follows from \eqref{eq8.35}, \eqref{eq8.57} and \eqref{eq8.69} that
\begin{align}\label{eq8.70}
&
\vert\{N\in[0,Q):\parity(\Phi(\alpha;\beta'';N)-\Phi(\alpha;\beta';N))=0\}\vert
\nonumber
\\
&\quad
>b^*q_{k+1}(1-\eps)+(Q-b^*q_{k+1}-q_{k+1}\eps)
=Q-(b^*+1)q_{k+1}\eps
\nonumber
\\
&\quad
\ge Q\left(1-\frac{b^*+1}{b^*}\eps\right)
\ge Q(1-2\eps).
\end{align}
If $k$ is odd and $b^*=1$, then the middle sum on the right hand side of \eqref{eq8.69} is empty, and it follows from \eqref{eq8.35} and \eqref{eq8.51} that
\begin{align}\label{eq8.71}
&
\vert\{N\in[0,Q):\parity(\Phi(\alpha;\beta'';N)-\Phi(\alpha;\beta';N))=0\}\vert
\nonumber
\\
&\quad
>q_{k+1}(1-\eps)+(Q-q_{k+1}-q_{k+1}\eps)
=Q-2q_{k+1}\eps
\ge Q(1-2\eps).
\end{align}
If $k$ is odd and $b^*=2$, then we ignore the last term on the right hand side of \eqref{eq8.69}, and it follows from \eqref{eq8.35} and \eqref{eq8.51} that
\begin{align}\label{eq8.72}
&
\vert\{N\in[0,Q):\parity(\Phi(\alpha;\beta'';N)-\Phi(\alpha;\beta';N))=0\}\vert
\nonumber
\\
&\quad
>2q_{k+1}(1-\eps)
>\frac{2}{3}Q(1-\eps)
>Q\left(\frac{2}{3}-\eps\right).
\end{align}
If $k$ is odd and $b^*\ge3$, then we ignore the term corresponding to $b=2$ in the middle sum on the right hand side of \eqref{eq8.69}, and it follows from \eqref{eq8.35} and \eqref{eq8.51} that
\begin{align}\label{eq8.73}
&
\vert\{N\in[0,Q):\parity(\Phi(\alpha;\beta'';N)-\Phi(\alpha;\beta';N))=0\}\vert
\nonumber
\\
&\quad
>(b^*-1)q_{k+1}(1-\eps)+(Q-b^*q_{k+1}-q_{k+1}\eps)
=(Q-q_{k+1})-(b^*+1)q_{k+1}\eps
\nonumber
\\
&\quad
\ge\frac{2}{3}Q-(b^*+1)q_{k+1}\eps
\ge Q\left(\frac{2}{3}-\frac{b^*+1}{b^*}\eps\right)
\ge Q\left(\frac{2}{3}-2\eps\right).
\end{align}
Suppose next that \eqref{eq8.68} holds.
Then analogous to \eqref{eq8.69}, we have
\begin{align}\label{eq8.74}
&
\vert\{N\in[0,Q):\parity(\Phi(\alpha;\beta'';N)-\Phi(\alpha;\beta';N))=0\}\vert
\nonumber
\\
&\quad
\ge\vert\{N\in\BBB(k):\parity(\Phi(\alpha;\beta'';N)-\Phi(\alpha;\beta';N))=0\}\vert
\nonumber
\\
&\quad\quad
+\sum_{\substack{{b=1}\\{b\ne2}}}^{a_{k+2}-1}\vert\{N\in\BBB^*(k;b):\parity(\Phi(\alpha;\beta'';N)-\Phi(\alpha;\beta';N))=0\}\vert,
\end{align}
where we have ignored the term corresponding to $b=2$ and the last term.
Combining \eqref{eq8.35} and \eqref{eq8.74} with \eqref{eq8.51} if $k$ is odd and with \eqref{eq8.57} if $k$ is even, we have
\begin{align}\label{eq8.75}
&
\vert\{N\in[0,Q):\parity(\Phi(\alpha;\beta'';N)-\Phi(\alpha;\beta';N))=0\}\vert
\nonumber
\\
&\quad
>(a_{k+2}-2)q_{k+1}(1-\eps)
>\frac{2}{3}Q(1-\eps)
>Q\left(\frac{2}{3}-\eps\right).
\end{align}
Since $\eps>0$ is arbitrary, we see that \eqref{eq3.8} follows immediately from \eqref{eq8.70}--\eqref{eq8.73} and \eqref{eq8.75} if we take $W^\star=(1+\alpha^2)^{1/2}Q^\star$.
We leave the deduction of the inequality $\vert\beta_1-\beta_0\vert<\eps$ to the reader.
\end{proof}

%
%

\section{Establishing the parity formula}\label{sec9}

Throughout this section, we assume that the integers $c_0,c_1,c_2,\ldots$ are even, and that $c_i$ is non-zero whenever $i$ is even.

\begin{proof}[Proof of Lemma~\ref{lem81}]
Recall the definition of $\Phi(\alpha;\beta;N)$ as given by \eqref{eq8.18}.
We need to find a description of the condition $0\le q\le N-1$ in terms of the $\alpha$-representations of $q$ and~$N$, as well as a description of the condition $\{q\alpha\}\in[0,\beta)$ in terms of the $\alpha$-representation of $q$ and the $\alpha$-expansion of~$\beta$.
For these, we recall some elementary facts in the theory of continued fractions.

\begin{fact1}
Suppose that
\begin{displaymath}
q=\sum_{i=0}^kx_iq_i
\quad\mbox{and}\quad
N=\sum_{i=0}^kb_iq_i
\end{displaymath}
are the $\alpha$-representations of $q$ and $N$ respectively, with the convention that when $q=0$, we have $x_0=\ldots=x_k=0$.
Then $0\le q\le N-1$ if and only if there exists some integer $m=0,\ldots,k$ such that
\begin{displaymath}
x_m<b_m,
\quad
x_i=b_i,
\quad
i=m+1,\ldots,k.
\end{displaymath}
\end{fact1}

\begin{fact2}
We have
\begin{displaymath}
\{q\alpha\}
=\left\{\sum_{i=0}^kx_iq_i\alpha\right\}
=\left\{\sum_{i=0}^kx_i(q_i\alpha-p_i)\right\}
=\left\{\sum_{i=0}^kx_i\eta_i\right\}
=\sum_{i=0}^kx_i\eta_i
\end{displaymath}
if and only if
\begin{equation}\label{eq9.1}
(x_0,\ldots,x_k)=(0,\ldots,0)
\quad\mbox{or}\quad
\mbox{$\min\{i=0,\ldots,k:x_i\ge1\}$ is even}.
\end{equation}
Suppose further that
\begin{displaymath}
\beta=\sum_{i=0}^kc_i\eta_i
\end{displaymath}
is the $\alpha$-expansion of~$\beta$, and that $0<\beta<1-\alpha$.
Then $\{q\alpha\}\in[0,\beta)$ if and only if there exists an integer $\ell$ such that
\begin{displaymath}
\sign(c_\ell-x_\ell)=(-1)^\ell,
\quad
x_i=c_i,
\quad
i=0,\ldots,\ell-1,
\end{displaymath}
and \eqref{eq9.1} holds.
\end{fact2}

For $N\in\BBB(k)=\{0,1,\ldots,q_{k+1}-1\}$, combining Facts 1 and~2, we have
\begin{equation}\label{eq9.2}
\Phi(\alpha;\beta;N)=\sum_{m=0}^k\sum_{\ell=0}^k\Phi_{\ell,m}(\alpha;\beta;N),
\end{equation}
where $\Phi_{\ell,m}(\alpha;\beta;N)$ denotes the number of integer sequences $(x_0,\ldots,x_k)$ such that \eqref{eq9.1} holds,
\begin{equation}\label{eq9.3}
\begin{array}{c}
x_0=c_0,\quad
\ldots,
\quad x_{\ell-1}=c_{\ell-1},
\vspace{5pt}\\
\sign(c_\ell-x_\ell)=(-1)^\ell,
\vspace{5pt}\\
x_m<b_m,
\vspace{5pt}\\
x_{m+1}=b_{m+1},
\quad
\ldots,
\quad
x_k=b_k,
\end{array}
\end{equation}
and the remaining terms satisfy
\begin{equation}\label{eq9.4}
0\le x_i\le a_{i+1},
\quad
x_{i-1}=0\mbox{ if }x_i=a_{i+1},
\quad
i=\ell+1,\ldots,m-1.
\end{equation}
To study some of these terms $\Phi_{\ell,m}(\alpha;\beta;N)$, we need a technical lemma.

\begin{lem}\label{lem91}
Let $\AAA_{h,r}(s)$ denote the number of integer sequences $(y_h,y_{h+1},\ldots,y_r)$ such that $y_h=s$,
\begin{displaymath}
0\le y_i\le a_{i+1},
\quad
i=h+1,\ldots,r,
\end{displaymath}
and
\begin{displaymath}
y_{i-1}=0
\quad\mbox{if}\quad
y_i=a_{i+1},
\quad
i=h+1,\ldots,r.
\end{displaymath}
Then
\begin{equation}\label{eq9.5}
\AAA_{h,r}(s)=(-1)^h(q_hp_{r+1}-p_hq_{r+1}),
\quad
s\ge1,
\end{equation}
and
\begin{equation}\label{eq9.6}
\AAA_{h,r}(0)=(-1)^h((p_{h+1}-p_h)q_{r+1}-(q_{h+1}-q_h)p_{r+1}).
\end{equation}
\end{lem}

\begin{proof}
For $s\ge1$, we shall prove \eqref{eq9.5} by induction on~$r$, starting with $r=h$.
Clearly $\AAA_{h,h}(s)=1$, and
\begin{displaymath}
\AAA_{h,h+1}
=a_{h+2}
=a_{h+2}(-1)^h(q_hp_{h+1}-p_hq_{h+1})
=(-1)^h(q_hp_{h+2}-p_hq_{h+2}).
\end{displaymath}
Note next that we have the recurrence relation
\begin{displaymath}
\AAA_{h,j}(s)=a_{j+1}\AAA_{h,j-1}(s)+\AAA_{h,j-2}(s).
\end{displaymath}
To see this, note that for each of the $a_{j+1}$ choices of $y_j$ satisfying $0\le y_j<a_{j+1}$, there are $\AAA_{h,j-1}(s)$ choices for the integer sequence $(y_h,y_{h+1},\ldots,y_{j-1})$, and for $y_j=a_{j+1}$, we must have $y_{j-1}=0$ and so there are $\AAA_{h,j-2}(s)$ choices for the integer sequence $(y_h,y_{h+1},\ldots,y_{j-2})$.
Using the induction hypothesis for the right hand side, we have
\begin{align}
\AAA_{h,j}(s)
&
=a_{j+1}(-1)^h(q_hp_j-p_hq_j)+(-1)^h(q_hp_{j-1}-p_hq_{j-1})
\nonumber
\\
&
=(-1)^h(q_hp_{j+1}-p_hq_{j+1}).
\nonumber
\end{align}
The identity \eqref{eq9.5} follows from the Principle of induction.
Finally, note that
\begin{displaymath}
\AAA_{h,r}(0)=\AAA_{h,r}(1)+\AAA_{h+1,r}(a_{h+2}).
\end{displaymath}
Applying \eqref{eq9.5} to the terms on the right now leads to the identity \eqref{eq9.6}.
\end{proof}

To evaluate $\Phi(\alpha;\beta;N)$, we need to split into cases.

\begin{case1}
Suppose that $m>\ell$, with $\ell$ even, so $c_\ell\ge1$.
The conditions \eqref{eq9.3} become
\begin{displaymath}
\begin{array}{c}
x_0=c_0,\quad
\ldots,
\quad x_{\ell-1}=c_{\ell-1},
\vspace{5pt}\\
x_\ell\in\{0,1,\ldots,c_\ell-1\},
\vspace{5pt}\\
x_m\in\{0,1,\ldots,b_m-1\},
\vspace{5pt}\\
x_{m+1}=b_{m+1},
\quad
\ldots,
\quad
x_k=b_k,
\end{array}
\end{displaymath}
and the remaining terms satisfy \eqref{eq9.4}.
Let $\Phi^*_{\ell,m}(\alpha;\beta;N)$ denote the number of integer sequences $(x_0,\ldots,x_k)$ such that \eqref{eq9.3} and \eqref{eq9.4} hold.
The restriction \eqref{eq8.25} gives $x_{m+1}\ne a_{m+2}$, so
\begin{displaymath}
\Phi^*_{\ell,m}(\alpha;\beta;N)=b_m\Omega_{m-1},
\end{displaymath}
where $\Omega_{m-1}$ is the number of those sequences $(x_\ell,x_{\ell+1},\ldots,x_{m-1})$ such that
\begin{displaymath}
x_\ell\in\{0,1,\ldots,c_\ell-1\},
\end{displaymath}
and \eqref{eq9.4} holds.
Using Lemma~9.1, we have
\begin{align}
\Omega_{m-1}
&
=\sum_{s=1}^{c_\ell-1}\AAA_{\ell,m-1}(s)+\AAA_{\ell,m-1}(0)
\nonumber
\\
&
=(c_\ell-1)(p_mq_\ell-q_mp_\ell)+(q_m(p_{\ell+1}-p_\ell)-p_m(q_{\ell+1}-q_\ell))
\nonumber
\\
&
=c_\ell(p_mq_\ell-q_mp_\ell)-(p_mq_{\ell+1}-q_mp_{\ell+1}).
\nonumber
\end{align}
Since $c_\ell$ is even, it follows that
\begin{displaymath}
\Phi^*_{\ell,m}(\alpha;\beta;N)=b_m(p_mq_{\ell+1}-q_mp_{\ell+1})\bmod{2}.
\end{displaymath}

Suppose that $\ell\ne0$.
The assumption that $c_0$ is non-zero and the requirement $x_0=c_0$ then guarantee that $x_0$ is non-zero, so that \eqref{eq9.1} is clearly satisfied, and so
$\Phi_{\ell,m}(\alpha;\beta;N)=\Phi^*_{\ell,m}(\alpha;\beta;N)$.
On the other hand, when $\ell=0$, we do not have $x_0=c_0$ but $\sign(c_0-x_0)=1$, so that $x_0<c_0$, and this does not guarantee that \eqref{eq9.1} holds.
To obtain $\Phi_{0,m}(\alpha;\beta;N)$, we then have to deduct from $\Phi^*_{0,m}(\alpha;\beta;N)$ the count of those sequences $(x_0,\ldots,x_k)$ that do not satisfy \eqref{eq9.1}.
We shall not give a precise value of this count.
Instead, it suffices to show that this count is independent of the sequence $c_0,c_1,c_2,c_3,\ldots.$
The details are slightly different, depending on whether $m$ is odd or even.

Suppose first of all that $m$ is odd.
Then those sequences $(x_0,\ldots,x_k)$ that need to be excluded from the count are the following: either $m\ge3$ and there exists an even integer $s=0,2,\ldots,m-3$ such that
\begin{displaymath}
\begin{array}{c}
x_0=x_1=\ldots=x_s=0,
\vspace{5pt}\\
x_{s+1}\in\{1,\ldots,a_{s+2}\},
\vspace{5pt}\\
x_{s+2}\in\{0,\ldots,a_{s+3}-1\},
\vspace{5pt}\\
\mbox{$x_{s+3},\ldots,x_{m-1}$ satisfy \eqref{eq9.4}},
\vspace{5pt}\\
x_m<b_m,
\vspace{5pt}\\
x_{m+1}=b_{m+1},\quad\ldots,\quad x_k=b_k,
\end{array}
\end{displaymath}
or, if $b_{m+1}\ne a_{m+1}$,
\begin{displaymath}
\begin{array}{c}
x_0=x_1=\ldots=x_{m-1}=0,
\vspace{5pt}\\
x_m\in\{1,\ldots,b_m-1\},
\vspace{5pt}\\
x_{m+1}=b_{m+1},\quad\ldots,\quad x_k=b_k,
\end{array}
\end{displaymath}
or
\begin{displaymath}
\begin{array}{c}
x_0=x_1=\ldots=x_m=0,
\vspace{5pt}\\
\mbox{$\min\{i=m+1,\ldots,k:b_i\ge1\}$ is odd}.
\end{array}
\end{displaymath}
This count is clearly independent of the sequence $c_0,c_1,c_2,c_3,\ldots.$

Suppose next that $m$ is even.
Then those sequences $(x_0,\ldots,x_k)$ that need to be excluded from the count are the following: either $m\ge4$ and there exists an even integer $s=0,2,\ldots,m-4$ such that
\begin{displaymath}
\begin{array}{c}
x_0=x_1=\ldots=x_s=0,
\vspace{5pt}\\
x_{s+1}\in\{1,\ldots,a_{s+2}\},
\vspace{5pt}\\
x_{s+2}\in\{0,\ldots,a_{s+3}-1\},
\vspace{5pt}\\
\mbox{$x_{s+3},\ldots,x_{m-1}$ satisfy \eqref{eq9.4}},
\vspace{5pt}\\
x_m<b_m,
\vspace{5pt}\\
x_{m+1}=b_{m+1},\quad\ldots,\quad x_k=b_k,
\end{array}
\end{displaymath}
or $m\ge2$ and
\begin{displaymath}
\begin{array}{c}
x_0=x_1=\ldots=x_{m-2}=0,
\vspace{5pt}\\
x_{m-1}\in\{1,\ldots,a_m\},
\vspace{5pt}\\
x_m<b_m,
\vspace{5pt}\\
x_{m+1}=b_{m+1},\quad\ldots,\quad x_k=b_k,
\end{array}
\end{displaymath}
or $m\ge2$ and
\begin{displaymath}
\begin{array}{c}
x_0=x_1=\ldots=x_m=0,
\vspace{5pt}\\
\mbox{$\min\{i=m+1,\ldots,k:b_i\ge1\}$ is odd}.
\end{array}
\end{displaymath}
This count is also independent of the sequence $c_0,c_1,c_2,c_3,\ldots.$

In summary, corresponding to these values of $\ell$ and~$m$, we can write
\begin{equation}\label{eq9.7}
\III_1
=\mathop{\sum_{m=0}^k\sum_{\ell=0}^k}_{\substack{{m>\ell}\\{\ell\ \textrm{even}}}}
b_m(p_mq_{\ell+1}-q_mp_{\ell+1})
-\EEE_1,
\end{equation}
where $\EEE_1$ is independent of the sequence $c_0,c_1,c_2,c_3,\ldots.$
\end{case1}

\begin{case2}
Suppose that $m>\ell$, with $\ell$ odd, so $c_{\ell-1}\ge1$.
The conditions \eqref{eq9.3} become
\begin{displaymath}
\begin{array}{c}
x_0=c_0,\quad
\ldots,
\quad x_{\ell-1}=c_{\ell-1},
\vspace{5pt}\\
x_\ell\in\{c_\ell+1,\ldots,a_{\ell+1}-1\},
\vspace{5pt}\\
x_m\in\{0,1,\ldots,b_m-1\},
\vspace{5pt}\\
x_{m+1}=b_{m+1},
\quad
\ldots,
\quad
x_k=b_k,
\end{array}
\end{displaymath}
and the remaining terms satisfy \eqref{eq9.4}.
The restriction \eqref{eq8.25} gives $x_{m+1}\ne a_{m+2}$, so
\begin{displaymath}
\Phi_{\ell,m}(\alpha;\beta;N)=b_m\Omega_{m-1},
\end{displaymath}
where $\Omega_{m-1}$ is the number of those sequences $(x_\ell,x_{\ell+1},\ldots,x_{m-1})$ such that
\begin{displaymath}
x_\ell\in\{c_\ell+1,\ldots,a_{\ell+1}-1\},
\end{displaymath}
and \eqref{eq9.4} holds.
Using Lemma~9.1, we have
\begin{displaymath}
\Omega_{m-1}
=\sum_{s=c_\ell+1}^{a_{\ell+1}-1}\AAA_{\ell,m-1}(s)
=-(a_{\ell+1}-c_\ell-1)(p_mq_\ell-q_mp_\ell).
\end{displaymath}
Since $c_\ell$ is even, it follows that
\begin{displaymath}
\Phi_{\ell,m}(\alpha;\beta;N)=b_m(a_{\ell+1}-1)(p_mq_\ell-q_mp_\ell)\bmod{2}.
\end{displaymath}
Corresponding to these values of $\ell$ and~$m$, we can write, for instance
\begin{equation}\label{eq9.8}
\III_2
=\mathop{\sum_{m=0}^k\sum_{\ell=0}^k}_{\substack{{m>\ell}\\{\ell\ \textrm{odd}}}}
b_ma_{\ell+1}(p_mq_\ell-q_mp_\ell)
+\mathop{\sum_{m=0}^k\sum_{\ell=0}^k}_{\substack{{m>\ell}\\{\ell\ \textrm{odd}}}}
b_m(p_mq_\ell-q_mp_\ell)
=\III_2^{(1)}+\III_2^{(2)}.
\end{equation}
\end{case2}

\begin{case3}
Suppose that $m=\ell$, with $\ell$ even.
The conditions \eqref{eq9.3} and \eqref{eq9.4} become
\begin{displaymath}
\begin{array}{c}
x_0=c_0,\quad
\ldots,
\quad x_{\ell-1}=c_{\ell-1},
\vspace{5pt}\\
x_\ell\in\{0,1,\ldots,\min\{b_\ell,c_\ell\}-1\},
\vspace{5pt}\\
x_{\ell+1}=b_{\ell+1},
\quad
\ldots,
\quad
x_k=b_k.
\end{array}
\end{displaymath}
The restriction \eqref{eq8.25} gives $x_{\ell+1}\ne a_{\ell+2}$, and so $\Phi^*_{\ell,m}(\alpha;\beta;N)=\min\{b_\ell,c_\ell\}$.

As in Case~1, for $\ell=0$, we need to deduct from $\Phi^*_{0,m}(\alpha;\beta;N)$ the count of those sequences $(x_0,\ldots,x_k)$ that do not satisfy \eqref{eq9.1}, \textit{i.e.}, those that satisfy
\begin{displaymath}
\begin{array}{c}
x_0=0,
\vspace{5pt}\\
\mbox{$\min\{i=1,\ldots,k:b_i\ge1\}$ is odd}.
\end{array}
\end{displaymath}
As before, this count is independent of the sequence $c_0,c_1,c_2,c_3,\ldots.$

In summary, corresponding to these values of $\ell$ and~$m$, we can write
\begin{equation}\label{eq9.9}
\III_3
=\mathop{\sum_{m=0}^k\sum_{\ell=0}^k}_{\substack{{m=\ell}\\{\ell\ \textrm{even}}}}
\min\{b_\ell,c_\ell\}
-\EEE_3
=\sum_{\substack{{\ell=0}\\{\ell\ \textrm{even}}}}^k
\min\{b_\ell,c_\ell\}
-\EEE_3,
\end{equation}
where $\EEE_3$ is independent of the sequence $c_0,c_1,c_2,c_3,\ldots.$
\end{case3}

\begin{case4}
Suppose that $m=\ell$, with $\ell$ odd.
The conditions \eqref{eq9.3} and \eqref{eq9.4} become
\begin{displaymath}
\begin{array}{c}
x_0=c_0,\quad
\ldots,
\quad x_{\ell-1}=c_{\ell-1},
\vspace{5pt}\\
x_\ell\in\{c_\ell+1,\ldots,a_{\ell+1}-1\},
\vspace{5pt}\\
x_\ell\in\{0,1,\ldots,b_\ell-1\},
\vspace{5pt}\\
x_{\ell+1}=b_{\ell+1},
\quad
\ldots,
\quad
x_k=b_k.
\end{array}
\end{displaymath}
The restriction \eqref{eq8.25} gives $x_{\ell+1}\ne a_{\ell+2}$, and it is not difficult to see that
\begin{displaymath}
\Phi_{\ell,m}(\alpha;\beta;N)=\max\{b_\ell-c_\ell-1,0\}.
\end{displaymath}
Corresponding to these values of $\ell$ and~$m$, we can write
\begin{equation}\label{eq9.10}
\III_4
=\mathop{\sum_{m=0}^k\sum_{\ell=0}^k}_{\substack{{m=\ell}\\{\ell\ \textrm{odd}}}}
\max\{b_\ell-c_\ell-1,0\}
=\sum_{\substack{{\ell=0}\\{\ell\ \textrm{odd}}}}^k
\max\{b_\ell-c_\ell-1,0\}.
\end{equation}
\end{case4}

\begin{case5}
Suppose that $\ell>m$.
Note here that the condition \eqref{eq9.3} is very restrictive and the condition \eqref{eq9.4} is void.
It is easy to see that
\begin{displaymath}
\Phi_{\ell,m}(\alpha;\beta;N)=\delta(\ell,m),
\end{displaymath}
where
\begin{equation}\label{eq9.11}
\delta(\ell,m)=\left\{\begin{array}{ll}
1,&\mbox{if $c_m<b_m$, $c_i=b_i$, $m<i<\ell$, and $\sign(c_\ell-b_\ell)=(-1)^\ell$},\\
0,&\mbox{otherwise}.
\end{array}\right.
\end{equation}
Corresponding to these values of $\ell$ and~$m$, we can write
\begin{equation}\label{eq9.12}
\III_5=\mathop{\sum_{m=0}^k\sum_{\ell=0}^k}_{\ell>m}\delta(\ell,m).
\end{equation}
\end{case5}

Combining \eqref{eq9.2}, \eqref{eq9.7}--\eqref{eq9.10} and \eqref{eq9.12}, we see that
\begin{equation}\label{eq9.13}
\Phi(\alpha;\beta;N)=\III_1+\III_2^{(1)}+\III_2^{(2)}+\III_3+\III_4+\III_5+\EEE_1+\EEE_3\bmod{2}.
\end{equation}
We shall show that
\begin{equation}\label{eq9.14}
\III_1+\III_2^{(2)}=0\bmod{2},
\end{equation}
and that
\begin{equation}\label{eq9.15}
\III_2^{(1)}=\sum_{m=0}^kb_m(p_mq_0-q_mp_0)
+\sum_{\substack{{\ell=0}\\{\ell\ \textrm{odd}}}}^k
b_\ell
\bmod{2}.
\end{equation}
We then show that
\begin{equation}\label{eq9.16}
\III_3+\III_4
=\sum_{\ell=0}^k\min\{b_\ell,c_\ell\}
+\sum_{\substack{{\ell=0}\\{\ell\ \textrm{odd}}}}^kb_\ell
+\sum_{\substack{{\ell=0}\\{\ell\ \textrm{odd}}\\{b_\ell>c_\ell}}}^k1
\bmod{2},
\end{equation}
and that
\begin{equation}\label{eq9.17}
\III_5
=\sum_{\ell=1}^k\Delta_\ell
+\sum_{\substack{{\ell=0}\\{\ell\ \textrm{odd}}\\{b_\ell>c_\ell}}}^k1
\bmod{2}.
\end{equation}
Then combining \eqref{eq9.13}--\eqref{eq9.17}, we deduce that
\begin{displaymath}
\Phi(\alpha;\beta;N)
=\sum_{m=0}^kb_m(p_mq_0-q_mp_0)
+\sum_{\ell=0}^k\min\{b_\ell,c_\ell\}
+\sum_{\ell=1}^k\Delta_\ell
+\EEE_1+\EEE_3
\bmod{2},
\end{displaymath}
where the translation term
\begin{displaymath}
\sum_{m=0}^kb_m(p_mq_0-q_mp_0)+\EEE_1+\EEE_3
\end{displaymath}
proves to be a nuisance.

Taking two numbers $\beta'$ and $\beta''$ satisfying $0<\beta'<\beta''<1-\alpha$, with $\alpha$-expansion digits $c'_i$ and $c''_i$ respectively and taking the difference $\Phi(\alpha;\beta'';N)-\Phi(\alpha;\beta';N)$, we remove this translation term, and the deduction of the parity formula \eqref{eq8.26} is essentially complete, apart from the deduction of the congruences \eqref{eq9.14}--\eqref{eq9.17}.

To establish \eqref{eq9.14}, note simply from \eqref{eq9.7} and \eqref{eq9.8} that
\begin{align}
\III_2^{(2)}
&
=\mathop{\sum_{m=0}^k\sum_{\ell=0}^k}_{\substack{{m>\ell}\\{\ell\ \textrm{odd}}}}
b_m(p_mq_\ell-q_mp_\ell)
=\mathop{\sum_{m=0}^k\sum_{\ell=0}^k}_{\substack{{m>\ell+1}\\{\ell\ \textrm{even}}}}
b_m(p_mq_{\ell+1}-q_mp_{\ell+1})
\nonumber
\\
&
=\mathop{\sum_{m=0}^k\sum_{\ell=0}^k}_{\substack{{m>\ell}\\{\ell\ \textrm{even}}}}
b_m(p_mq_{\ell+1}-q_mp_{\ell+1})
-b_{\ell+1}(p_{\ell+1}q_{\ell+1}-q_{\ell+1}p_{\ell+1})
=\III_1.
\nonumber
\end{align}

To establish \eqref{eq9.15}, note that using the recurrence relations \eqref{eq8.2}, we can write
\begin{equation}\label{eq9.18}
\III_2^{(1)}
=\mathop{\sum_{m=0}^k\sum_{\ell=0}^k}_{\substack{{m>\ell}\\{\ell\ \textrm{odd}}}}
b_m(p_m(q_{\ell+1}-q_{\ell-1})-q_m(p_{\ell+1}-p_{\ell-1}))
=\III_2^{(1+)}+\III_2^{(1-)},
\end{equation}
where
\begin{align}\label{eq9.19}
\III_2^{(1+)}
&
=\mathop{\sum_{m=0}^k\sum_{\ell=0}^k}_{\substack{{m>\ell}\\{\ell\ \textrm{odd}}\\{m\ \textrm{even}}}}
b_m(p_m(q_{\ell+1}-q_{\ell-1})-q_m(p_{\ell+1}-p_{\ell-1}))
\nonumber
\\
&
=\sum_{\substack{{m=0}\\{m\ \textrm{even}}}}^k
b_m(p_m(q_m-q_0)-q_m(p_m-p_0))
=\sum_{\substack{{m=0}\\{m\ \textrm{even}}}}^k
b_m(p_mq_0-q_mp_0),
\end{align}
and
\begin{align}\label{eq9.20}
\III_2^{(1-)}
&
=\mathop{\sum_{m=0}^k\sum_{\ell=0}^k}_{\substack{{m>\ell}\\{\ell\ \textrm{odd}}\\{m\ \textrm{odd}}}}
b_m(p_m(q_{\ell+1}-q_{\ell-1})-q_m(p_{\ell+1}-p_{\ell-1}))
\nonumber
\\
&
=\sum_{\substack{{m=0}\\{m\ \textrm{odd}}}}^k
b_m(p_m(q_{m-1}-q_0)-q_m(p_{m-1}-p_0))
\nonumber
\\
&
=\sum_{\substack{{m=0}\\{m\ \textrm{odd}}}}^k
b_m(p_m(q_m-q_0)-q_m(p_m-p_0))
\nonumber
\\
&\quad
+\sum_{\substack{{m=0}\\{m\ \textrm{odd}}}}^k
b_m(p_m(q_{m-1}-q_m)-q_m(p_{m-1}-p_m))
\nonumber
\\
&
=\sum_{\substack{{m=0}\\{m\ \textrm{odd}}}}^k
b_m(p_mq_0-q_mp_0)
+\sum_{\substack{{\ell=0}\\{m\ \textrm{odd}}}}^k
b_m
\bmod{2}.
\end{align}
The congruence \eqref{eq9.15} now follows on combining \eqref{eq9.18}--\eqref{eq9.20}.

To establish \eqref{eq9.16}, note that
\begin{displaymath}
\max\{b_\ell-c_\ell-1,0\}=\left\{\begin{array}{ll}
b_\ell-c_\ell-1,&\mbox{if $b_\ell>c_\ell$},\\
0,&\mbox{if $b_\ell\le c_\ell$},
\end{array}\right.
\end{displaymath}
so that
\begin{displaymath}
\min\{b_\ell,c_\ell\}=\left\{\begin{array}{ll}
c_\ell=b_\ell-(b_\ell-c_\ell-1)-1,&\mbox{if $b_\ell>c_\ell$},\\
b_\ell=b_\ell-0,&\mbox{if $b_\ell\le c_\ell$},
\end{array}\right.
\end{displaymath}
and so
\begin{displaymath}
\min\{b_\ell,c_\ell\}=\left\{\begin{array}{ll}
b_\ell-\max\{b_\ell-c_\ell-1,0\}-1,&\mbox{if $b_\ell>c_\ell$},\\
b_\ell-\max\{b_\ell-c_\ell-1,0\},&\mbox{if $b_\ell\le c_\ell$}.
\end{array}\right.
\end{displaymath}
Hence
\begin{equation}\label{eq9.21}
\sum_{\substack{{\ell=0}\\{\ell\ \textrm{odd}}}}^k
\min\{b_\ell,c_\ell\}
=\sum_{\substack{{\ell=0}\\{\ell\ \textrm{odd}}}}^kb_\ell
+\sum_{\substack{{\ell=0}\\{\ell\ \textrm{odd}}}}^k
\max\{b_\ell-c_\ell-1,0\}
+\sum_{\substack{{\ell=0}\\{\ell\ \textrm{odd}}\\{b_\ell>c_\ell}}}^k1
\bmod{2}.
\end{equation}
The congruence \eqref{eq9.16} now follows on combining \eqref{eq9.9}, \eqref{eq9.10} and \eqref{eq9.21}.

Finally, to establish \eqref{eq9.17}, we note from \eqref{eq8.23}, \eqref{eq8.24} and Lemma~\ref{lem82} that $\Delta_\ell=1$ if and only if there exists an integer $m<\ell$ such that
\begin{equation}\label{eq9.22}
\mbox{$\ell$ is even, \eqref{eq8.27} holds and \eqref{eq8.28} holds},
\end{equation}
and $\Delta_\ell=-1$ if and only if
\begin{equation}\label{eq9.23}
\mbox{$\ell$ is odd, \eqref{eq8.29} holds and \eqref{eq8.30} holds}.
\end{equation}
or there exists an integer $m<\ell$ such that
\begin{equation}\label{eq9.24}
\mbox{$\ell$ is odd, \eqref{eq8.29} holds and \eqref{eq8.31} holds}.
\end{equation}
Naturally for a given~$\ell$, the integer $m$ for which \eqref{eq9.22} or \eqref{eq9.24} is valid is uniquely determined.
Suppose first that $\ell\ge1$ is even.
Then $\Delta_\ell=1$ if and only if there exists an integer $m<\ell$ such that \eqref{eq9.22} holds.
For the integer $m$ in question, it follows from \eqref{eq9.11} that $\delta(\ell,m)=1$.
Thus
\begin{equation}\label{eq9.25}
\sum_{\substack{{\ell=1}\\{\ell\ \textrm{even}}}}^k\Delta_\ell
=\mathop{\sum_{m=0}^k\sum_{\ell=0}^k}_{\substack{{\ell>m}\\{\ell\ \textrm{even}}}}\delta(\ell,m)
\end{equation}
Suppose next that $\ell\ge1$ is odd.
Then $\Delta_\ell=-1$ if and only if \eqref{eq9.23} holds or there exists an integer $m<\ell$ such that \eqref{eq9.24} holds.
In the latter case, for the integer $m$ in question, it follows from \eqref{eq9.11} that $\delta(\ell,m)=0$.
Thus
\begin{equation}\label{eq9.26}
\mathop{\sum_{m=0}^k\sum_{\ell=0}^k}_{\substack{{\ell>m}\\{\ell\ \textrm{odd}}}}\delta(\ell,m)=0,
\end{equation}
and
\begin{equation}\label{eq9.27}
\sum_{\substack{{\ell=1}\\{\ell\ \textrm{odd}}}}^k\Delta_\ell
=\sum_{\substack{{\ell=0}\\{\ell\ \textrm{odd}}\\{b_\ell>c_\ell}}}^k1.
\end{equation}
The congruence \eqref{eq9.17} now follows on combining \eqref{eq9.12} and \eqref{eq9.25}--\eqref{eq9.27}.

This completes the deduction of the parity formula \eqref{eq8.26}.
\end{proof}

%
%

\end{document}